\documentclass[11pt]{amsart}

\usepackage{amsmath}
\usepackage{amssymb}
\usepackage{bbm}
\usepackage{pdfsync}
\usepackage{esint} 
\usepackage[pagebackref]{hyperref}
\usepackage[utf8]{inputenc}
\usepackage[T1]{fontenc}
\usepackage{mathtools}

\renewcommand*{\backref}[1]{}
\renewcommand*{\backrefalt}[4]{%
	\ifcase #1 (Not cited.)%
	\or        (Cited on page~#2)%
	\else      (Cited on pages~#2)%
	\fi}

\usepackage[capitalize]{cleveref} 
\crefname{equation}{}{} 
\crefname{claim}{Claim}{Claims} 
\crefname{page}{p.}{pp.}
\crefname{rem}{Remark}{Remarks} 
\crefname{coro}{Corollary}{Corollaries} 
\crefname{enumi}{item}{items}

\usepackage{enumitem}

\hypersetup{ 
	linktoc=page, 
	colorlinks=true,
	linkcolor={red!90!black}, 
	citecolor={blue!90!black}, 
	urlcolor=magenta,
	filecolor=teal,
}

\let\oldtocsection=\tocsection
\let\oldtocsubsection=\tocsubsection
\let\oldtocsubsubsection=\tocsubsubsection
\renewcommand{\tocsection}[2]{\hspace{0em}\oldtocsection{#1}{#2}}
\renewcommand{\tocsubsection}[2]{\hspace{1em}\oldtocsubsection{#1}{#2}}
\renewcommand{\tocsubsubsection}[2]{\hspace{2em}\oldtocsubsubsection{#1}{#2}}

\makeatletter
\newcommand\@dotsep{4.5}
\def\@tocline#1#2#3#4#5#6#7{\relax
	\ifnum #1>\c@tocdepth 
	\else
	\par \addpenalty\@secpenalty\addvspace{#2}%
	\begingroup \hyphenpenalty\@M
	\@ifempty{#4}{%
		\@tempdima\csname r@tocindent\number#1\endcsname\relax
	}{%
		\@tempdima#4\relax
	}%
	\parindent\z@ \leftskip#3\relax
	\advance\leftskip\@tempdima\relax
	\rightskip\@pnumwidth plus1em \parfillskip-\@pnumwidth
	#5\leavevmode\hskip-\@tempdima #6\relax
	\leaders\hbox{$\m@th
		\mkern \@dotsep mu\hbox{.}\mkern \@dotsep mu$}\hfill
	\hbox to\@pnumwidth{\@tocpagenum{#7}}\par
	\nobreak
	\endgroup
	\fi}
\makeatother 

\newcommand{\C}{{\mathbb C}}       
\newcommand{\R}{{\mathbb R}}       
\newcommand{\N}{{\mathbb N}}

\newcommand{\HH}{{\mathcal H}}

\newcommand{\EE}{{\mathcal E}}

\newcommand{\SSS}{{\mathcal S}}

\newcommand{\OO}{{\mathcal O}}

\newcommand{\diam}{{\rm diam}}
\newcommand{\dist}{{\rm dist}}

\newcommand{\real}{{\rm Re}}
\newcommand{\imag}{{\rm Im}}
\newcommand{\rf}[1]{{(\ref{#1})}}

\newcommand{\supp}{\operatorname{supp}}

\newcommand{\vv}{{\vspace{2mm}}}

\newcommand{\vmo}{{\operatorname{VMO}}}

\usepackage{dsfont}
\newcommand{\characteristic}{\mathbf{1}} 

\newcommand{\divv}{\operatorname{div}}
\newcommand{\Capacity}{\operatorname{Cap}}
\newcommand{\curl}{\operatorname{curl}}
\newcommand{\bmo}{\operatorname{BMO}}

\newcommand{\loc}{\operatorname{loc}} 
\newcommand{\opt}{\operatorname{opt}} 

\newcommand{\rom}[1]{%
	\textup{\uppercase\expandafter{\romannumeral#1}}%
}

\def\Xint#1{\mathchoice
	{\XXint\displaystyle\textstyle{#1}}%
	{\XXint\textstyle\scriptstyle{#1}}%
	{\XXint\scriptstyle\scriptscriptstyle{#1}}%
	{\XXint\scriptscriptstyle\scriptscriptstyle{#1}}%
	\!\int}
\def\XXint#1#2#3{{\setbox0=\hbox{$#1{#2#3}{\int}$ }
		\vcenter{\hbox{$#2#3$ }}\kern-.58\wd0}}

\def\avint{\;\Xint-}

\usepackage[disable]{todonotes} 

\usepackage{pgf,tikz} 
\usetikzlibrary{lindenmayersystems}

\newtheorem{theorem}{Theorem}[section]
\newtheorem{lemma}[theorem]{Lemma}

\newtheorem{coro}[theorem]{Corollary}

\newtheorem*{claim*}{Claim}

\newtheorem*{theorem*}{Theorem}

\theoremstyle{definition}
\newtheorem{definition}[theorem]{Definition}
\newtheorem*{notation}{Notation}

\newtheorem{rem}[theorem]{Remark} 

\numberwithin{equation}{section}

\pgfdeclarelindenmayersystem{Koch}
{\rule {F -> F+F--F+F}}

\begin{document}
	\title[Planar elliptic measures via quasiconformal mappings]{The Hausdorff dimension of planar elliptic measures via quasiconformal mappings}
	
	\author[Ignasi Guill{\'e}n-Mola]{Ignasi Guill{\'e}n-Mola}
	\address{Ignasi Guill{\'e}n-Mola,
		Departament de Matem\`atiques, Universitat Aut\`onoma de Barcelona.
	}
	\email{ignasi.guillen@uab.cat}
	
	\date{\today}
	
	\thanks{Supported by the PIF UAB 2020/21 from the Universitat Autònoma de Barcelona, and by Generalitat de Catalunya’s Agency for Management of University and Research Grants (AGAUR) (2021 FI\_B 00637 and 2023 FI-3 00151). Also partially supported by the Spanish State Research Agency (AEI) project PID2021-125021NAI00 (Spain) and by MICINN (Spain) under the grant PID2020-114167GB-I00. The author thanks CERCA Programme/Generalitat de Catalunya for institutional support.}



	\begin{abstract}
        In this paper, we obtain new bounds for the Hausdorff dimension of planar elliptic measure via the application of quasiconformal mappings, with these bounds depending solely on the ellipticity constant of the matrix. In fact, in our case studies, we find a quasiconformal mapping that relates the elliptic measure in a domain to the harmonic measure in its image domain, allowing us to deduce bounds for the dimension of the elliptic measure from the known results on the harmonic side. This extends previous works of Makarov, Jones and Wolff.
	\end{abstract}
	
	\maketitle
	
	{
		\tableofcontents
	}
	
	
	\section{Introduction and main results}\label{sec:intro}

    In this paper, we provide new bounds for the dimension of planar elliptic measures. The bounds depend solely on the ellipticity constant of the matrix. When the uniformly elliptic matrix is symmetric and has constant determinant, we obtain finer estimates.
	
	Let $\Omega \subset \R^{n+1}$, $n\geq 1$, be a domain (open and connected set), and let $A \in \R^{(n+1)\times (n+1)} (\Omega)$ be a uniformly elliptic matrix. That is, a real and measurable matrix $A(\cdot)=(a_{ij} (\cdot))_{1\leq i,j \leq n+1}$ such that there exists $\lambda \geq 1$, the so-called ellipticity constant, with
	\begin{equation*}
		\lambda^{-1}|\xi|^{2} \leq\langle A(x) \xi, \xi\rangle \text{ and } 
		\langle A(x) \xi, \eta\rangle \leq \lambda |\xi||\eta|,
	\end{equation*}
	for all $\xi,\eta \in \R^{n+1}$ and a.e.\ $x\in \Omega$. The uniform ellipticity condition implies that the matrix has bounded coefficients.
 
    We consider the second-order operator in divergence form $L_A u \coloneqq -\divv (A\nabla u)$, understood in the weak sense. We say that a function $u\in W^{1,2}_{\loc} (\Omega)$ is $L_A$-harmonic in $\Omega$ if
	$$
	\int_\Omega A(y) \nabla u(y) \nabla \varphi (y) \, dy = 0 \text{ for all } \varphi \in C_c^\infty (\Omega) .
	$$
	In this case, we write $L_A u=0$ in $\Omega$ for short.
	
	The elliptic ($L_A$-harmonic) measure in $\Omega\subset \R^{n+1}$ (bounded if $n\geq 2$) with pole $p\in \Omega$ is the unique Radon probability measure $\omega_{\Omega,A}^p$, given by the Riesz representation theorem, such that
	$$
	u_f^{L_A}(p) = \int_{\partial\Omega} f(\xi) \, d\omega_{\Omega,A}^p (\xi)\text{ for every }f\in C(\partial\Omega),
	$$
	where $u_f^{L_A}$ is the $L_A$-harmonic extension of $f$ in $\Omega$ via Perron's method. If the set $\Omega$, the matrix $A$, or the pole $p\in \Omega$ is clear from the context, we will omit them in $\omega_{\Omega, A}^p$. For a detailed construction of elliptic/harmonic measures, see \cite[Section 11]{Heinonen2006}.

    Note that by the Harnack inequality of positive $L_A$-harmonic functions, given any two poles $x,y \in \Omega$ we have
	\begin{equation*}
		\omega_{\Omega, A}^x \ll \omega_{\Omega, A}^y \ll \omega_{\Omega, A}^x.
	\end{equation*}
	Therefore, we can omit the pole from now on. Given two measures $\mu$ and $\nu$ on $X$, we recall that $\mu$ is absolutely continuous with respect to $\nu$, we write $\mu \ll \nu$ in this case, if $\mu(S)=0$ for all $S\subset X$ with $\nu (S) = 0$.

    Harmonic and elliptic measure theory has seen significant development recently, particularly concerning the rectifiability of the boundary and dimensional bounds. For instance, when the boundary of the domain has codimension one, the relation between harmonic measure and rectifiability of the boundary has been studied by several authors. Roughly speaking, under some topological assumptions on the domain $\Omega \subset \mathbb{R}^{n+1}$, if $\partial \Omega$ is $n$-rectifiable, then $\omega_{\Omega,Id} \ll \mathcal{H}^n|_{\partial\Omega}$, see \cite{Akman-Bortz-Hofmann-Martell-ARKIV-2019,Azzam-Hofmann-Martell-Mourgoglou-Tolsa-INVENTIONES-2020}, and if $\partial \Omega$ is purely unrectifiable, then $\omega_{\Omega,Id}$ is singular with respect to $\mathcal{H}^n|_{\partial \Omega}$, see \cite{Azzam2016,Azzam2017a}. Similar results hold for elliptic measure if the uniformly elliptic matrix is close to the identity, see \cite{Kenig-Pipher-2001,Hofmann-Martell-Toro-2017,Hofmann2021-published,Azzam2022,Hofmann-Martell-Mayboroda-ANALYSISANDPDE-2024}. The closeness of the matrix to the identity is usually measured in terms of a suitable Carleson measure condition on the gradient or the oscillation of the matrix. However, for general uniformly elliptic matrices, the elliptic measure can behave completely differently from the harmonic case, see examples \cite{Caffarelli1981,Modica-Mortola-1980,Sweezy1992,David2021} in the planar case.
	
	In this work, we are interested in studying the Hausdorff dimension of elliptic measures arising from uniformly elliptic matrices, defined as
	$$
	\dim_\HH \omega_{\Omega, A} \coloneqq \inf\{\dim_\HH F : \omega_{\Omega, A} (F^c) = 0\}.
	$$
    
    Let us introduce some notation to present the forthcoming results. Given a measure $\mu$ we will say that $\omega_{\Omega,A}$ has $\sigma$-finite with respect to $\mu$ if there is a set $F\subset\partial\Omega$ satisfying $\omega_{\Omega,A}(F)=1$ and with $\sigma$-finite $\mu$, that is, $F$ is of the form $F=\bigcup_{i\geq 1} F_i$ with $\mu(F_i)<\infty$. For short we will write that $\omega_{\Omega,A}$ has $\sigma$-finite $\mu$. The measure $\HH^1$ will also be referred to as length.
	
    In the plane, the most precise results regarding the dimension of the harmonic measure $\omega = \omega_{Id}$ are attributed to Makarov \cite[Theorem 1]{Makarov1985} and Wolff \cite{Wolff1993}. Makarov showed that there is a constant $C_M>0$ such that for any simply connected domain $\Omega$, its harmonic measure $\omega_\Omega$ satisfies $\omega_\Omega \ll \HH^{\varphi_{1,C_M}}$, where $\varphi_{\cdot,\cdot}$ is the function defined for $\rho\geq 1$ and $C>0$ as
    \begin{equation}\label{makarov gauge function}
		\varphi_{\rho,C} (r) = r^{\frac{2}{\rho+1}} \exp \left( C \frac{2\rho}{\rho+1} \sqrt{ \log r^{-\rho} \log \log \log r^{-\rho} } \right).
	\end{equation}
    We note that $\HH^{\varphi_{\rho,C}} \ll \HH^\alpha$ for any $\alpha <\frac{2}{\rho+1}$ because $\varphi_{\rho,C} (t)/t^\alpha \to 0$ as $t \to 0$. This in particular implies $\dim_\HH \omega_\Omega \geq 1$. Actually, Makarov also proved in \cite[Theorem 3]{Makarov1985} that $\dim_\HH \omega = 1$ for simply connected domains. On the other hand, the upper bound $\dim_\HH \omega_\Omega \leq 1$ for any planar domain $\Omega$ is due to Jones and Wolff in \cite{Jones1988}. This was later improved by Wolff in \cite{Wolff1993}, who showed that $\omega_\Omega$ has $\sigma$-finite length. For more geometric properties of the harmonic measure, see the introduction of \cite{Guillen-Prats-Tolsa} and the references therein.
	
	For future reference, we write here the precise statements of Makarov and Wolff mentioned above.
	
	\begin{theorem}[Makarov]\label{makarov thm}
		There exists a universal constant $C_M>0$ such that $\omega_\Omega \ll \HH^{\varphi_{1,C_M}}$ for every simply connected domain $\Omega \subset \R^2$.
	\end{theorem}
	
	\begin{theorem}[Wolff]\label{wolff thm}
		For any domain $\Omega \subset \R^2$, $\omega_\Omega$ has $\sigma$-finite length.
	\end{theorem}

    Despite this paper is focused in the planar case, let us recall some results about the dimensional behavior of harmonic measure $\omega=\omega_{Id}$ in $\R^{n+1}$ with $n\geq 2$. Bourgain proved in \cite{Bourgain1987} that there exists a dimensional constant $b_n >0$ such that $\dim_\HH \omega_\Omega\leq n+1-b_n$ for any domain $\Omega\subset\R^{n+1}$. Contrary to the planar case, where we can take $b_{1} = 1$ with optimal value by the work of Jones and Wolff \cite{Jones1988}, Wolff constructed a domain $\Omega_n\subset \R^{n+1}$ in \cite{Wolff1995} satisfying $\dim\omega_{\Omega_n} >n$, that is, $b_n<1$. It remains a difficult open problem in the area to determine the optimal value $b_{n,\opt} \coloneqq n+1-\sup\{\dim\omega_\Omega : \Omega\subset \R^{n+1}\}\in (0,1)$.

    The behavior of the Hausdorff dimension of elliptic measures is different. For every $\varepsilon>0$, using quasiconformal mappings, Sweezy constructed a domain $\Omega\subset \R^{2}$ and a uniformly elliptic matrix $A$ such that $\dim_\HH \omega_{\Omega,A} \geq 2-\varepsilon$, see \cite{Sweezy1992}. Sweezy also provided the higher-dimensional analog in \cite{Sweezy1994}. Through a different technique, David and Mayboroda in \cite{David2021} defined a continuous scalar function $a\approx 1$ on the complementary of the $1$-dimensional four corner Cantor set $C\subset \R^2$ such that $\omega_{\R^2\setminus C, a Id}$ is comparable to $\HH^1|_C$. Operators of the form $L_{aId}=-\divv a\nabla$ that satisfy the above comparability are the so-called ``good elliptic operators'' in the terminology used in \cite{David2021}. By the same strategy, for $1<d<\log 4 /\log 3$, Perstneva in \cite{Perstneva2025} constructed a ``good elliptic operator'' $L_{a_d Id}$ on the (unbounded component of the) complementary of a $d$-dimensional Koch-type snowflake $S_d \subset \R^2$ satisfying that $\omega_{\R^2\setminus S_d,a_d Id}$ is comparable to ${\HH^d}|_{S_d}$.
    
    When adding regularity to the matrix, one could expect that elliptic measures behave as the harmonic measure. Recently, in collaboration with Prats and Tolsa in \cite{Guillen-Prats-Tolsa}, for Lipschitz matrices we proved the analog of Wolff's \cref{wolff thm} for planar Reifenberg flat domains with small constant. More precisely, given a uniformly elliptic matrix $A$ with Lipschitz coefficients, and a Reifenberg flat domain $\Omega \subset \R^2$ with small constant depending on the ellipticity of $A$, $\omega_{\Omega,A}$ has $\sigma$-finite length.

    To the best of the author's knowledge, the above results are among the few concerning the dimension of the elliptic measure in the plane without any restriction on the dimension of the boundary. In the same spirit as Makarov's \cref{makarov thm} and Wolff's \cref{wolff thm}, in this paper we provide bounds for the dimension of the planar elliptic measure. Namely, we obtain $\sigma$-finiteness of the elliptic measure with respect to certain Hausdorff measures and we show absolute continuity of the elliptic measure on finitely connected domains with respect to a gauge Hausdorff measure $\HH^{\varphi_{\rho,C}}$, where $\varphi_{\rho,C}$ is the gauge function defined in \rf{makarov gauge function}. Recall that, as discussed above, we have that
    $$
    \omega_{\Omega,A}\ll \HH^{\varphi_{\rho,C}}
    \implies
    \dim_\HH \omega_{\Omega,A} \geq \frac{2}{\rho+1}.
    $$

    The first group of results concerns general planar domains and uniformly elliptic matrices.

    \begin{theorem}\label{main thm general domains}
    Let $\Omega \subset \R^2$ be a domain, $A \in \R^{2\times 2} \left(\Omega\right)$ be a uniformly elliptic matrix with ellipticity constant $\lambda \geq 1$, and
    \begin{equation}\label{K-lambda ellipticity condition beltrami coefficient}
    K_\lambda = \begin{cases}
        \lambda & \text{if } A \text{ is symmetric},\\
        \lambda + \sqrt{\lambda^2 - 1} & \text{otherwise}.
    \end{cases}
    \end{equation}
    Then $\omega_{\Omega,A}$ has $\sigma$-finite $\HH^{2K_\lambda /(K_\lambda +1)}$. If in addition $\Omega$ is finitely connected, then $\omega_{\Omega, A} \ll \HH^{\varphi_{K_\lambda, C_M}}$, where $C_M>0$ is the constant in Makarov's \cref{makarov thm}.
    \end{theorem}
    
	Note that this result is consistent with the results in the harmonic case of Makarov's \cref{makarov thm} and Wolff's \cref{wolff thm}. If $A=Id$, we can just use that the ellipticity constant is $\lambda = 1$ in \cref{main thm general domains} and hence $2/(\lambda+1)=2\lambda /(\lambda+1)=1$. The consistency of using \cref{main thm general domains} with $\lambda=1$ is due to the sharpness of the distortion of the Hausdorff dimension and measures under quasiconformal mappings, see \rf{dimension distortion theorem}, \cref{distortion of hausdorff measure thm}, and \cref{distortion zero measure of makarov gauge function}.

    The proof relies only on the uniform ellipticity condition on the matrix. The key point is that there exists a quasiconformal mapping sending the elliptic Green function to the harmonic Green function in the image domain, which allows us to relate the elliptic measure to the harmonic measure in the image domain. For the proof of \cref{main thm general domains}, see \cpageref{proof of main thm general domains} in \cref{part:general domains}.
	
    We remark that, in the above situation, the quasiconformal mapping depends not only on the matrix but also on the Green function of the domain. Therefore, we cannot easily deduce better regularity of the quasiconformal mapping even if the matrix has additional regularity, as we will do later in the case of symmetric matrices with determinant 1. However, if the matrix is continuous up to the boundary, we obtain the following result.

    \begin{theorem}\label{main thm general domains and continuous matrices}
        Let $\Omega \subset \R^2$ be a domain and $A\in \R^{2\times 2} (\overline \Omega)$ be a uniformly elliptic matrix with continuous coefficients up to $\partial\Omega$, that is, $A\in C^0 (\overline \Omega)$. Then $\dim_\HH \omega_{\Omega, A} \leq 1$. If in addition $\Omega$ is finitely connected, then $\dim_\HH \omega_{\Omega, A} = 1$.
    \end{theorem}

    To prove this result, we use a localization argument that relies on the fact that $K_\lambda \to 1$ as $\lambda \to 1$. Recall that in \cite{Guillen-Prats-Tolsa}, Prats, Tolsa, and the author obtained the $\sigma$-finite length of $\omega_{\Omega,A}$ under the assumptions that $A$ is Lipschitz and $\Omega\subset\R^2$ is Reifenberg flat. In contrast, the result above shows that the upper bound $\dim\omega_{\Omega,A} \leq 1$ holds under much weaker assumptions: $\Omega$ is any planar domain and $A$ is continuous up to the boundary.

    We note that \cref{main thm general domains and continuous matrices} does not contradict the result of Perstneva in \cite{Perstneva2025}. In fact, the scalar function $a_d \approx 1$ that she defined is continuous, but not uniformly continuous, in the domain $\Omega$, the (unbounded component of the) complement of a $d$-dimensional Koch-type snowflake, while in \cref{main thm general domains and continuous matrices} we require the matrix to be continuous up to the boundary. Note that if $A$ is uniformly continuous in every open subset of $\Omega$, then $A$ can be extended continuously in $\overline\Omega$, i.e., $A\in C(\overline\Omega)$.

    An interesting question is the validity of \cref{main thm general domains and continuous matrices} in higher dimensions, in the sense that whether $\dim_\mathcal{H} \omega_{\Omega,A} \leq n+1 - b_{n,\opt}$ for any domain $\Omega \subset \mathbb{R}^{n+1}$ and any uniformly elliptic matrix $A \in C(\overline{\Omega})$, where $b_{n,\opt}\in (0,1)$ is the optimal (unknown) constant such that $\dim_\mathcal{H} \omega_{\Omega, Id} \leq n+1 - b_{n,\opt}$ for any domain $\Omega \subset \mathbb{R}^{n+1}$.

    We now turn to the case of symmetric matrices with determinant 1, where finer estimates can be obtained thanks to the improved control of the associated quasiconformal mapping. The first result focuses on Hölder continuous matrices.

    \begin{theorem}\label{sym + det 1 Holder}
        Let $\Omega \subset \R^2$ be a domain, and $A\in\R^{2\times 2} (\R^2)$ be a uniformly elliptic and symmetric matrix with $\det A=1$. If $A\in C^\alpha_{\loc} (\R^2)$, $0<\alpha<1$, then $\omega_{\Omega, A}$ has $\sigma$-finite length. If in addition $\Omega$ is finitely connected, then $\omega_{\Omega, A} \ll \HH^{\varphi_{1,C_M}}$, where $C_M>0$ is the constant in Makarov's \cref{makarov thm}.
    \end{theorem}

	As in \cref{main thm general domains}, this result is also consistent with the results in the harmonic case of Makarov's \cref{makarov thm} and Wolff's \cref{wolff thm}. If $A=Id$, we can apply \cref{sym + det 1 Holder} using that the coefficients are Hölder.

    Contrary to what we had in the general case, for symmetric matrices with determinant 1, the quasiconformal mapping which relates the elliptic measure of a domain with the harmonic measure in the image domain is uniquely determined by the matrix, and in fact, the elliptic measure is precisely the push-forward of that harmonic measure. In particular, the more regular the coefficients of the matrix, the more regular the quasiconformal mapping. Proceeding as in the proof of \cref{main thm general domains} and using the additional regularity of the quasiconformal mapping when the coefficients enjoy extra regularity, we obtain \cref{sym + det 1 Holder} and the following cases:
    
    \begin{theorem}\label{sym + det 1 all cases}
        Let $\Omega \subset \R^2$ be a domain, and $A\in\R^{2\times 2} (\R^2)$ be a uniformly elliptic (with ellipticity constant $\lambda \geq 1$) and symmetric matrix with $\det A=1$. Then the following holds: 
        \begin{itemize}
            \item If $A(z) = A(\real(z))$ (or $A(z) = A(\imag(z))$), $A\in W_{\loc}^{1,p} (\R^2)$ with $p>\frac{2\lambda^2}{\lambda^2+1}$ or $A\in {\rm DMO}_p$ (see \cref{def:dini mean oscillation}) with $p>1$, then $\omega_{\Omega, A}$ has $\sigma$-finite length.
            \item If $A\in \vmo_{\log}$ (see \cref{def:log mean oscillation}), then $\omega_{\Omega, A}$ has $\sigma$-finite $\HH^{t/\log \frac{1}{t}}$.
            \item If $A\in \vmo$ (see \cref{def:vmo}), then $\dim_\HH \omega_{\Omega,A}\leq 1$.
        \end{itemize}
        If in addition $\Omega$ is finitely connected, the following hold:
        \begin{itemize}
            \item If $A(z) = A(\real(z))$ (or $A(z) = A(\imag(z))$), or $A\in {\rm DMO}_1$, then $\omega_{\Omega, A} \ll \HH^{\varphi_{1,C_M}}$.
            \item If $A\in \vmo_{\log}$, then $\omega_{\Omega, A} \ll \HH^{(\varphi_{1,C_M}) \circ (t\log\frac{1}{t})}$.
            \item If $A\in \vmo$, then $\dim_\HH \omega_{\Omega,A}=1$.
        \end{itemize}
        Here $C_M>0$ is the constant in Makarov's \cref{makarov thm}.
    \end{theorem}
    
    We remark that \cref{sym + det 1 Holder,sym + det 1 all cases} hold as long as the determinant is constant. The proofs of \cref{sym + det 1 Holder} and all cases of \cref{sym + det 1 all cases} can be found on \cpageref{proof of sym + det 1 holder} (Hölder case), \cpageref{proof of main thm only one variable dependence} ($A(z)=A(\real(z))$ case), \cpageref{proof of vmo matrix everywhere} ($\vmo$ case), \cpageref{proof of main thm with sobolev} (Sobolev case), \cpageref{proof of dini mean oscillation elliptic measure} (${\rm DMO}$ case) and \cpageref{proof of square mean oscillation elliptic measure} ($\vmo_{\log}$ case), all in \cref{part:sym + det 1}. Also, as noted in \cref{DMO inclusions + DINI and VMO,DMO subset VMOlog subset VMO}, for $1\leq p < \infty$ there holds
    $$
    \text{Hölder} \subset {\rm Dini} \subset {\rm DMO}_p \subset {\rm DMO}_1 \subset \vmo_{\log} \subset \vmo,
    $$
    where ${\rm Dini}$ consists of functions satisfying the pointwise Dini condition \cref{pointwise dini condition}.
	
	In the same way, as a corollary of \cite[Theorems 5.1 and 5.3]{Tolsa2024} and by the same approach of \cref{main thm general domains,sym + det 1 Holder}, using the fact that the quasiconformal mapping and its inverse are $C^1$ when the matrix is ${\rm DMO}_p$ with $p>\lambda$, we derive the following result. The details are left to the reader. 
	
	\begin{coro}
		Let $A\in \R^{2\times 2} (\R^2)$ be a uniformly elliptic (with ellipticity constant $\lambda \geq 1$) and symmetric matrix with constant $\det A$. Let $E\subset \R^2$ be a compact subset contained in a $C^1$ curve which is $(s,C_0)$-AD regular for some $s>0$, that is,
		$$
		C_0^{-1} r^s \leq \HH^s (E \cap B(x,r)) \leq C_0 r^s \text{ for all } x\in E \text{ and } 0<r\leq \diam(E).
		$$
		Let $\Omega = \R^2 \setminus E$, and assume also that $A\in {\rm DMO}_p$ with $p>\lambda$. We have that:
		\begin{itemize}
			\item If $s\in \left[\frac{1}{2},1\right)$, then $\dim_\HH \omega_{\Omega, A} < s$.
			\item There exists $\varepsilon = \varepsilon_{C_0} >0$ small enough such that if $s\in \left[\frac{1}{2} - \varepsilon,1\right)$, then $\dim_\HH \omega_{\Omega, A} < s$.
		\end{itemize}
	\end{coro}

    In the preceding result we concluded that $\dim_\HH \omega_{\Omega,A}$ is strictly smaller than the dimension of the boundary. This phenomenon is known as dimension drop. For additional background in dimension drop, see for instance \cite{Azzam2020,David2023,Tolsa2024} and the references therein.

    \vspace{3mm}
    \noindent\textbf{Acknowledgments.}
    The author would like to thank Martí Prats and Xavier Tolsa for their guidance and advice.
 
    \section{Preliminaries}

    \subsection{Notation}

    \begin{itemize}
	\item We use $c,C\geq 1$ to denote constants that may depend only on the dimension and the constants appearing in the hypotheses of the results, and whose values may change at each occurrence.
	
	\item We write $a\lesssim b$ if there exists a constant $C\geq 1$ such that $a\leq Cb$, and $a\approx b$ if $C^{-1} b \leq a \leq C b$.
	
	\item If we want to stress the dependence of the constant on a parameter $\eta$, we write $a\lesssim_\eta b$ or $a\approx_\eta b$ meaning that $C=C(\eta)=C_\eta$.
	
	\item The ambient space is $\R^2$. However, some auxiliary results will be stated in general, i.e., in $\R^{n+1}$ for $n\geq 1$.
	
	\item The diameter of a set $E\subset \R^{n+1}$ is denoted by $\diam\, E \coloneqq \sup_{x,y\in E} |x-y|$.
	
	\item We denote by $B_r (x)$ or $B(x,r)$ the open ball with center $x$ and radius $r$, i.e., $B_r (x)=B(x,r)=\{y\in \R^{n+1} : |y-x|<r\}$. We denote $B_r \coloneqq B_r (0)$.

        \item Given a ball $B$, we denote by $r_B$ or $r(B)$ its radius, and by $c_B$ or $c(B)$ its center.

        \item We denote by $Q_r(x)$ or $Q(x,r)$ the open cube with center $x$ and side length $2s$, i.e., $Q_r(x) = Q(x,r) = \{y\in \R^{n+1} : |y_i-x_i|<r\text{ for all }1\leq i \leq n+1\}$.

        \item Given a cube $Q$, we denote by $\ell(Q)$ its side length, and by $c_Q$ or $c(Q)$ its center. That is, $Q=Q(c_Q, \ell(Q)/2)$.
	
	\item We say that a matrix $A$ is Hölder continuous with exponent $\alpha \in (0,1]$ in a set $U$, or briefly $C^{0,\alpha} (U)$, if its coefficients are Hölder continuous with exponent $\alpha$. That is, there exists a constant $C_\alpha >0$ (called the Hölder seminorm) such that
	\begin{equation*}
		\left|a_{i j}(x)-a_{i j}(y)\right| \leq C_\alpha |x-y|^{\alpha} \text { for all } x, y \in U \text { and } 1 \leq i, j \leq n+1 .
	\end{equation*}
	For shortness we write $C^\alpha$ instead of $C^{0,\alpha}$ if $\alpha \in (0,1)$, and when $\alpha = 1$ we say ``Lipschitz continuous''. In this case we write $C_L$ instead of $C_1$, i.e.,
	\begin{equation*}
		\left|a_{i j}(x)-a_{i j}(y)\right| \leq C_L |x-y| \text { for all } x, y \in U \text { and } 1 \leq i, j \leq n+1 .
	\end{equation*}
	
	\item We say that a function $f$ is $\kappa$-Lipschitz in $U$ if $|f(x)-f(y)| \leq \kappa |x-y|$ for all $x,y\in U$.
	
	\item We denote the characteristic function of a set $E$ by $\characteristic_E$.
	
	\item Given $t>0$ and a set $E\subset \R^{n+1}$, we write $U_t (E) \coloneqq \{x\in \R^{n+1} : \dist (x,E) < t \}$ for the $t$-neighborhood $E$.

    \item Given a space of function $X$, we say that a matrix $A=(a_{ij})_{i,j=1}^{n+1}$ belongs to $X$, $A\in X$ for short, if $a_{ij}\in X$ for all $1\leq i,j\leq n+1$.

    \item Given a continuous nondecreasing function $h:[0,\infty)\to [0,\infty)$ with $h(0)=0$, the $h$-Hausdorff measure (resp. content) is denoted by $\HH^h$ (resp. $\HH^h_\infty$), see \cite[Section 4.9]{Mattila1995} for instance. It is clear that $\HH_\infty^h (E) \leq \HH^h (E)$ for any set $E$. Despite the converse inequality does not hold, $\HH_\infty^h (E)=0$ implies $\HH^h (E)=0$. Although $\HH_\infty^h$ is not measure, we write $\HH^h \ll \HH_\infty^h \ll \HH^h$. If $h(r)=r^\alpha$, $\alpha>0$, we simply write $\HH^\alpha$ and $\HH^\alpha_\infty$.
\end{itemize}

    \subsection{Reduction to bounded Wiener regular domains}\label{reduction: bounded domains-compact support-simply connected}
	
	We present the lemmas to see that for the study of the dimension of planar elliptic measures, we can reduce to bounded Wiener regular domains. For the proof of the results in \cref{sec:intro} for bounded Wiener regular domains, refer to \cref{part:general domains,part:sym + det 1}.

        \begin{definition}[Wiener regular]
            Let $\Omega\subset \R^{n+1}$, $n\geq 1$, be a domain and $A$ be a real and uniformly elliptic matrix. We say that $\xi\in \partial\Omega$ is a Wiener regular point for $L_A$ if for every $f \in C(\partial \Omega)$, the $L_A$-harmonic extension $u_f^{L_A}$ of $f$ in $\Omega$ via Perron's method satisfies $\lim_{\Omega\ni x \to \xi} u_f^{L_A}(x) = f(\xi)$. We say that $\Omega$ is Wiener regular for $L_A$ if every point $\xi\in\partial\Omega$ is a Wiener regular point for $L_A$.
        \end{definition}

        \begin{rem}
            In fact, the characterization of Wiener regular points does not depend on the operator $L_A$, see \cite{Littman-Stampacchia-Weinberger-1963}. Therefore, we say that a point $\xi\in \Omega$ is Wiener regular if and only if it is Wiener regular for some (and hence for all) $L_A$.
        \end{rem}
        
        In particular, using the Wiener regularity characterization for the Laplace operator, we directly note that finitely connected domains $\Omega\subset \R^2$ are Wiener regular, and in order to study the size of Wiener regular points (for any $L_A$) we can use the classical potential theory for the Laplacian operator. From this and the same proof of \cite[Lemma 9.8]{Prats-Tolsa-book-vDec2023}, we obtain the following.

        \begin{lemma}\label{reduction to Wiener regular domains}
            Let $A$ be a real and uniformly elliptic matrix. To prove that for any domain $\Omega\subset \R^2$, either $\dim_\HH \omega_{\Omega,A} \leq \beta$ (with $0<\beta\leq 2$) or $\omega_{\Omega,A}$ has $\sigma$-finite (gauge) $\HH^h$, it suffices to prove it for Wiener regular domains.
            \begin{proof}
                Indeed, in the proof of \cite[Lemma 9.8]{Prats-Tolsa-book-vDec2023} one can replace with no modification harmonic measure by elliptic measure, the set $G$ (given by assuming the result for Wiener regular domains) is constructed so that $\omega_{\Omega,A} (G)=1$ and either $\dim_\HH G\leq \beta$ or $\HH^h (G)$ is $\sigma$-finite respectively, and the set $F$ satisfies $\HH^\gamma (F)=0$ for all $0<\gamma\leq 2$, see also \cite[Lemma 6.19]{Prats-Tolsa-book-vDec2023}. In particular, the set $G\cup F$ has full elliptic measure $\omega_{\Omega,A}$ and either $\dim_\HH (G\cup F)\leq \beta$ or $\HH^h (F\cup G)$ is $\sigma$-finite respectively, as claimed.
            \end{proof}
        \end{lemma}
	
	Before presenting the lemmas for reducing to bounded domains, we first introduce the capacity.
	
	\begin{definition}
		Let $K$ be a compact subset of an open set $\Omega$. Its capacity is
		$$
		\Capacity (K,\Omega) = \inf \left\{\int_{\Omega} |\nabla u|^2 \, dx : u\in C^\infty_c (\Omega) \text{ and } u\geq 1 \text{ on } K  \right\}.
		$$
	\end{definition}
	
	Let $\Omega \subset \R^{n+1}$, with $n \geq 1$, be a (bounded if $n\geq 2$) Wiener regular domain and $B$ be a ball centered at $\partial \Omega$. By the maximum principle \cite[p.~46]{Gilbarg2001} we have 
	\begin{equation}\label{easy inequality localized elliptic measure}
		\omega^z_{\Omega \cap B,A} (E) \leq \omega^z_{\Omega,A} (E) \text{ for any } E\subset \partial \Omega \cap B \text{ and } z\in \Omega \cap B .
	\end{equation}
	The converse inequality may fail. However, the following weaker relation for (bounded if $n\geq 2$) Wiener regular domains holds, see \cite[Lemma 6.1]{Guillen-Prats-Tolsa}. If $\Capacity (B \cap \partial \Omega, 4B)\not = 0$, for a Borel set $E\subset \partial \Omega \cap B$ and $z_E \in \partial 1.5 B$ such that $\omega^{z_E}_{\Omega,A} (E) = \max_{z\in \partial 1.5 B} \omega^z_{\Omega,A} (E)$, then
    \begin{equation}\label{hard inequality localized elliptic measure}
        \omega^{z_E}_{\Omega,A} (E) \lesssim \frac{\Capacity (2B, 4B)}{\Capacity (B \cap \partial \Omega, 4B)}\omega^{z_E}_{\Omega \cap 4B,A} (E),
    \end{equation}
    where the constant involved depends on the ellipticity constant of the matrix $A$ and the dimension. We remark that $\Capacity (B \cap \partial \Omega, 4B)=0$ would imply $\omega_{D,A} (B\cap \partial \Omega)=0$ for any domain $D$ with $B\cap\partial\Omega\subset\partial D$, see \cite[Theorems 10.1 and 11.14]{Heinonen2006}.
	
	The following two lemmas allow us to reduce to the case of bounded domains.
	
	\begin{lemma}[See {\cite[Lemma 6.2]{Guillen-Prats-Tolsa}}]\label{localization argument}
		Let $\Omega \subset \R^2$ be a (possibly unbounded) Wiener regular domain and $A$ be a real uniformly elliptic matrix. Let $\{B_i\}_i$ be a disjoint collection of balls centered at $\partial \Omega$ with $\omega_{\Omega,A}\left(\partial \Omega \setminus \bigcup_i B_i \right)=0$, and let $F_i \subset \partial (\Omega \cap 4B_i)$ with $\omega_{\Omega \cap 4 B_i ,A} (F_i)=1$. Then $F\coloneqq \partial \Omega \cap \bigcup_i F_i$ satisfies $\omega_{\Omega,A} (F)=1$. The same also holds for bounded Wiener regular domains $\Omega \subset \R^{n+1}$ when $n\geq 2$.
	\end{lemma}
	
	\begin{lemma}\label{claim:candidate for each index}
		Let $\Omega \subset \R^2$ be a finitely connected domain and $A$ be a real and uniformly elliptic matrix. For any small enough ball $B$ centered at $\partial \Omega$, each connected component of $\Omega \cap B$ is simply connected, and ${(\omega_{\Omega\cap B,A})}|_{\partial \Omega}\ll\omega_{\Omega,A}$. In particular $\dim \omega_{\Omega,A} \geq \dim \omega_{\Omega\cap B,A}$.
		\begin{proof}
			Since $\Omega$ is finitely connected, each connected component of $\Omega \cap B$ is simply connected. By \rf{easy inequality localized elliptic measure} we have ${(\omega_{\Omega\cap B,A})}|_{\partial \Omega}\ll\omega_{\Omega,A}$.
		\end{proof}
	\end{lemma}
	
	\begin{rem}\label{reductions all}
		By \cref{reduction to Wiener regular domains,localization argument,claim:candidate for each index} it suffices to prove our results for bounded Wiener regular domains. By translation, we can also assume $\Omega \subset B_{1/4} (0)$ if necessary.
	\end{rem}
	
	\begin{rem}\label{rem:modify beltrami coef instead of the matrix}
		Assuming $\Omega \subset B_{1/4} (0)$, it may seem natural to assume that $A=Id$ in $B_1(0)^c$ by replacing $A$ by the uniformly elliptic (with the same ellipticity constant) matrix
		$$
		\widetilde A \coloneqq \varphi A + (1-\varphi) Id,
		$$
		for some function $\varphi$ such that $0\leq \varphi \leq 1$, $\varphi = 1$ in $B_{1/2} (0)$ and $\varphi=0$ in $B_1(0)^c$. If $\varphi$ is smooth, then $\widetilde A$ often has the same regularity as $A$. Note that $\widetilde A$ is symmetric wherever $A$ is symmetric. However, $\widetilde A$ may no longer have determinant $1$ even when $A$ has determinant $1$. If necessary, if we wanted to preserve also the determinant 1 property, we could replace the matrix $A$ by the uniformly elliptic (also with the same ellipticity constant) matrix $\widehat A$ defined in \cref{rem:interpolation to preserve determinant}.
  
        Instead of modifying the matrix, in most of the cases we will modify the Beltrami coefficient, see the proof of the $\vmo$ case of \cref{sym + det 1 all cases} on \cpageref{proof of vmo matrix everywhere}, for example.
	\end{rem}
	
	We also want to note that from \rf{easy inequality localized elliptic measure} and \rf{hard inequality localized elliptic measure}, the dimension of elliptic measures only depends on how the matrix is around the boundary of the domain. We prove this in the following result.
	
	\begin{coro}\label{coro:elliptic measure only depends on how is the matrix around the boundary}
		Let $\Omega \subset \R^{n+1}$, $n\geq 1$, be a (possibly unbounded if $n=1$) Wiener regular domain, $\varepsilon>0$ and let $A,B$ be two real uniformly elliptic matrices with $A=B$ in $\{x\in \Omega : \dist(x,\partial \Omega) < \varepsilon\}$. Then $\omega_{\Omega, A} (F) = 1$ if and only if $\omega_{\Omega, B} (F)=1$. That is, $\omega_{\Omega, A} \ll \omega_{\Omega, B} \ll \omega_{\Omega, A}$.
		\begin{proof}
			Let $\omega_{\cdot} \coloneqq \omega_{\cdot, A}$ and $\widetilde \omega_{\cdot} \coloneqq \omega_{\cdot, B}$. We prove $\widetilde \omega_\Omega (F) =1 \Rightarrow \omega_\Omega(F)=1$, and the converse direction follows by symmetry. Let $F\subset \partial \Omega$ with $\widetilde \omega_\Omega (F)=1$, and let $\{B_i\}_i$ be a countable subfamily of $\{B_{\varepsilon/100} (\xi)\}_{\xi \in \partial \Omega}$ such that $\partial \Omega\subset \bigcup_{i} B_i$. We have
    		$$
    		\omega_\Omega (F^c) \leq \sum_i \omega_\Omega (F^c \cap B_i).
    		$$
    		We claim that each term on the right-hand side is zero, and in particular $\omega_\Omega (F)=1$. Indeed, for each $i$ fix a pole $p_i \in \Omega\cap 4B_i$. Since both matrices $A$ and $B$ coincide in $\Omega \cap 4B_i$ we have $\omega^{p_i}_{\Omega \cap 4B_i} (\cdot) = \widetilde \omega^{p_i}_{\Omega \cap 4B_i} (\cdot)$, and using \rf{easy inequality localized elliptic measure} and $\widetilde \omega^{z_i}_\Omega (F)=1$ we get
    			$$
    			\omega^{p_i}_{\Omega \cap 4B_i} (F^c \cap B_i)
                    =\widetilde \omega^{p_i}_{\Omega \cap 4B_i} (F^c \cap B_i)
    			\leq \widetilde \omega^{p_i}_\Omega (F^c \cap B_i)
    			\leq
    			\widetilde \omega^{p_i}_\Omega (F^c) = 0.
    			$$
            Arguing as in the proof of \cite[Lemma 6.2]{Guillen-Prats-Tolsa}, where \rf{hard inequality localized elliptic measure} is used, $\omega^{p_i}_{\Omega \cap 4B_i} (F^c \cap B_i)=0$ implies $\omega_\Omega (F^c \cap B_i) =0$, as we claimed.
		\end{proof}
	\end{coro}
	
	\subsection{Small perturbation of the constant case}
	
	In this subsection, we study the dimension of elliptic measures arising from a uniformly elliptic matrix whose coefficients are close in the $L^\infty$ norm to a uniformly elliptic constant matrix. The main result of this section and its proof will be used in the proof of \cref{main thm general domains and continuous matrices} and in the Dini mean oscillation case in \cref{sym + det 1 all cases}.
	
	For $\lambda \geq 1$ we denote
	$$
	\begin{aligned}
		\overline\dim_\HH(\lambda) &= \sup \{\dim_\HH \omega_{\Omega, A} : \Omega\subset \R^{n+1} \text{ and } A \text{ has ellipticity constant }\lambda\},\\
		\underline\dim_\HH^{\rm{sc}} (\lambda) &= \inf\{\dim_\HH \omega_{\Omega, A} : \Omega\subset \R^{2} \text{ is simply connected and } A \text{ has ellipticity constant }\lambda\}.
	\end{aligned}
	$$
	Using this notation, in the harmonic case, i.e., $\lambda =1$, in the plane we have $\overline\dim_\HH (1)\leq 1$ and $\underline\dim_\HH^{\rm{sc}}(1)\geq 1$ by Wolff's \cref{wolff thm} and Makarov's \cref{makarov thm} respectively.
	
	In the result below, we estimate the dimension of the elliptic measure in terms of how close the matrix A is to a uniformly elliptic constant matrix.
	
	\begin{lemma}\label{perturbation of constant matrix}
		Let $\Omega\subset \R^{n+1}$ be a domain and $C$ be a constant uniformly elliptic matrix with ellipticity constant $\lambda\geq 1$. Let $A$ be a (not necessarily constant) matrix and
		$$
		\varepsilon \coloneqq \max_{i,j\in\{1,\dots,n+1\}} \|a_{ij} (\cdot)- c_{ij}\|_{L^\infty (\Omega)}.
		$$
		If $\varepsilon$ is small enough (depending on $\lambda$ and $n$) then $A$ is uniformly elliptic in $\Omega$ and $\dim_\HH \omega_{\Omega,A} \leq \overline\dim_\HH ((1-(n+1)\varepsilon\lambda)^{-1})$. If $\Omega\subset \R^2$ is simply connected, then $\dim_\HH \omega_{\Omega,A} \geq \underline\dim_\HH^{\rm{sc}} ((1-2\varepsilon\lambda)^{-1})$.
	\end{lemma}
	
	Note that when the matrix $A$ is constant, the result agrees with the dimension of the harmonic measure. Before the proof, we need the following lemma.
	
	\begin{lemma}\label{sym of operator L_A constant}
		Let $\Omega\subset \R^{n+1}$ be a domain and $C$ be a uniformly elliptic and constant matrix. Then, in the weak form, for any $u\in W^{1,2}_{\loc} (\Omega)$ we have
		$$
		\divv C\nabla u = \divv C_s \nabla u,
		$$
		where $C_s$ denotes the symmetric part of the matrix $C$, that is $C_s = (C+C^T)/2$.
		\begin{proof}
			For $u \in C^\infty (\Omega)$, and as $C$ is constant, we have
			$$
			\divv C\nabla u = \sum_{i,j} c_{ij} \partial_{ij} u = \sum_{ij} \frac{c_{ij}+c_{ji}}{2} \partial_{ij} u = \divv C_s \nabla u .
			$$
			Let us check the weak form of this identity for $u\in W^{1,2}_{\loc} (\Omega)$. Given $\varphi \in C_c^\infty(\Omega)$, let $U$ be an open set such that $\supp\varphi \subset U \subset \overline U \subset \Omega$, and we have $u \in W^{1,2} (U)$. Take $\{u_k\}_{k\geq 1} \subset C^\infty (U)$ so that $\lim_{k\to \infty}\|u_k - u\|_{W^{1,2}(U)} =0$. As $\divv C\nabla u_k = \divv C_s \nabla u_k$ in the weak form for each $k\geq 1$, we get
			\begin{align*}
				\int_\Omega C\nabla u \nabla \varphi 
				=& \int_\Omega C\nabla u_k \nabla \varphi + \int_\Omega C\nabla (u-u_k) \nabla \varphi \\
				=& \int_\Omega C_s \nabla u_k \nabla \varphi + \int_\Omega C\nabla (u-u_k) \nabla \varphi \\
				=&\int_\Omega C_s \nabla u \nabla \varphi + \int_\Omega C_s \nabla (u_k-u) \nabla \varphi + \int_\Omega C\nabla (u-u_k) \nabla \varphi.
			\end{align*}
			The lemma follows since $\int_\Omega |\nabla (u-u_k)| |\nabla \varphi| \lesssim \|\nabla (u-u_k)\|_{L^2 (U)} \to 0$ as $k \to \infty$.
		\end{proof}
	\end{lemma}
	
	Using the previous lemma, we prove \cref{perturbation of constant matrix}.
	\begin{proof}[Proof of \cref{perturbation of constant matrix}]
		If $\varepsilon$ is small enough depending on $\lambda$ and $n$ then $A$ is uniformly elliptic in $\Omega$.
		
		By \cref{sym of operator L_A constant}, in the weak form we have
		\begin{equation*}
			\divv A\nabla = \divv C\nabla  + \divv (A-C)\nabla 
			= \divv C_s\nabla  + \divv (A-C)\nabla
			= \divv (C_s+ A-C)\nabla.
		\end{equation*}
		Note that the matrix $C_s+ A-C$ is uniformly elliptic as it has bounded coefficients and $\langle B \xi, \xi \rangle = \langle B_s \xi, \xi \rangle$ for any (not necessarily constant) matrix $B$. Then, taking the linear change of variables given by $S^T$, where $S$ is the square root of $C_s$ (that is, $S=\sqrt{C_s}$, the matrix satisfying $S^T S=C_s$) and
		\begin{equation}\label{matrix after change of constant sym change of variables}
            \begin{aligned}
            M (z) \coloneqq{}& (S^T)^{-1} \left( C_s + A(S^T z)- C \right) S^{-1}\\
            ={}&Id + (S^T)^{-1} \left( A(S^T z)- C \right) S^{-1} \text{ for }z\in (S^T)^{-1} (\Omega),
            \end{aligned}
		\end{equation}
		by \cref{elliptic measure deformation qc} below we have
		\begin{equation}\label{relation elliptic measure after square root change of variables}
			\omega_{\Omega, A}^{p} (E) 
			= \omega_{(S^T)^{-1} \Omega, M}^{(S^T)^{-1} p} ((S^T)^{-1} E) \text{ for any } p \in \Omega\text{ and }E\subset \partial\Omega.
		\end{equation}
		
		We claim that the ellipticity constant of $M$ in $(S^T)^{-1}(\Omega)$ is at most $(1-(n+1)\varepsilon\lambda)^{-1}$. Indeed, for $\xi,\eta\in \R^{n+1}$ and $x\in \Omega$, by the Cauchy-Schwartz inequality and the well-known fact $|Bv|\leq (n+1)\delta |v|$ if $|b_{ij}|\leq \delta$ for all $1\leq i,j \leq n+1$, we have
		$$
		\langle\{(S^T)^{-1} \left( A(x)- C \right) S^{-1}\}\xi,\eta\rangle
            \leq |{\left( A(x)- C \right) S^{-1}\xi}||S^{-1}\eta|
            \leq (n+1)\varepsilon |S^{-1}\xi||S^{-1}\eta|.
		$$
            Note that for any vector $v$ there holds $|Sv|^2 = \langle Sv,Sv\rangle = \langle S^T S v,v\rangle = \langle C_s v, v \rangle \geq \lambda^{-1} |v|^2$. All in all,
            $$
            \langle\{(S^T)^{-1} \left( A(x)- C \right) S^{-1}\}\xi,\eta\rangle
            \leq (n+1)\varepsilon\lambda |\xi||\eta|,
            $$
		which implies
		$$
		\langle M\xi, \eta \rangle \leq (1+(n+1)\varepsilon\lambda)|\xi||\eta|\text{ and }
		\langle M\xi,\xi\rangle\geq (1-(n+1)\varepsilon\lambda)|\xi|^2.
		$$
		The claim follows as $1/(1-(n+1)\varepsilon\lambda)\geq 1+(n+1)\varepsilon\lambda$.
		
		Since the linear mapping given by $S^T$ is bi-Lipschitz and the ellipticity constant of $M$ is at most $(1-(n+1)\varepsilon\lambda)^{-1}$, the lemma follows from \rf{relation elliptic measure after square root change of variables} and the definition of $\overline\dim_\HH$ and $\underline\dim_\HH^{\textrm{sc}}$.
	\end{proof}
	
	Note that if the constant matrix $C$ is symmetric, i.e., $C=C_s$, then \rf{matrix after change of constant sym change of variables} reads as
	\begin{equation}\label{matrix after change of constant sym change of variables - sym case}
		M (z) \coloneqq (S^T)^{-1} A(S^T z) S^{-1} \text{ for }z\in (S^T)^{-1} (\Omega).
	\end{equation}
	
	As we mentioned in the introduction of this section, we will apply \cref{perturbation of constant matrix} and adapt its proof in some proofs below. More precisely, we will compare the matrix $A$ with the frozen coefficient matrix $C=A(\xi)$, for some point $\xi\in \partial\Omega$.
	
	\begin{rem}\label{matrix after change of constant sym change of variables - sym case - is sym and det 1}
		Taking $C=A(\xi)$ (for some point $\xi\in \partial\Omega$) when the matrix $A$ is symmetric with constant determinant, we have that the matrix $M$ in \rf{matrix after change of constant sym change of variables - sym case} is symmetric and has determinant $1$.
	\end{rem}
	
	\section{Quasiconformal mappings}
	
	In this section, we provide the definition and main properties of quasiconformal mappings. This is based on the monograph \cite{Astala2009}.
	
	\begin{definition}[Quasiconformal mapping]
		An orientation-preserving homeomorphism $f : \Omega\to \Omega^\prime$ between planar domains is called $K$-quasiconformal, $1\leq K <\infty$, if $f\in W^{1,2}_{\loc} (\Omega)$ and the directional derivatives satisfy
		\begin{equation}\label{quasiconformal condition 1}
			\max_{\alpha\in[0,2\pi]} |\partial_\alpha f| \leq K \min_{\alpha\in[0,2\pi]} |\partial_\alpha f| \text{ a.e.\ in } \Omega .
		\end{equation}
	\end{definition}
	
	\begin{rem}
		In particular, a mapping is $1$-quasiconformal if and only if it is conformal.
	\end{rem}
	
	The condition in \rf{quasiconformal condition 1} is equivalent to
	$$
	|\bar \partial f| \leq k |\partial f|  \text{ a.e.\ in } \Omega, \text{ for } k \coloneqq \frac{K-1}{K+1},
	$$
	where $\partial \coloneqq \partial_z = \frac{1}{2} (\partial_x - i\partial_y)$ and $\bar\partial \coloneqq \partial_{\bar z} = \frac{1}{2} (\partial_x + i\partial_y)$ are the Wirtinger derivatives.
	
	A fundamental property of onto $K$-quasiconformal mappings $f:\Omega \to \Omega^\prime$ is that its inverse function $f^{-1} : \Omega^\prime \to \Omega$ is also $K$-quasiconformal, see \cite[Theorem 3.1.2]{Astala2009}. Also, global $K$-quasiconformal mappings $f:\C \to \C$ are quasisymmetric, that is, there exists an increasing homeomorphism $\eta_K : [0,\infty) \to [0,\infty)$, only depending on $K$, such that for each triple $z_0, z_1, z_2 \in \C$ we have
	\begin{equation}\label{quasisymmetric condition}
		\frac{|f(z_0)-f(z_1)|}{|f(z_0)-f(z_2)|} \leq \eta_K \left(\frac{|z_0 - z_1|}{|z_0 - z_2|}\right) ,
	\end{equation}
	see \cite[Theorem 3.5.3]{Astala2009}. In this case, we say that $f$ is $\eta_K$-quasisymmetric. For a local equivalence between quasiconformal and quasisymmetric mappings, see \cite[Theorem 3.6.2]{Astala2009}.
	
	Another important property is that quasiconformal mappings are locally Hölder continuous, see \cite[Corollary 3.10.3]{Astala2009}.
	
	\begin{theorem}[Local Hölder regularity]\label{qc are locally holder}
		Every $K$-quasiconformal mapping $f:\Omega \to \Omega^\prime$ is locally $\frac{1}{K}$-Hölder continuous. More precisely, if $B\subset 2B \subset \Omega$, then
		$$
		|f(z)-f(w)| \leq C_K \diam(f (B)) \frac{|z-w|^{1/K}}{\diam(B)^{1/K}} \text{ for all } z,w\in B,
		$$
		where the constant $C_K$ depends only on $K$.
	\end{theorem}
	
	Quasiconformal mappings are related to solutions of the Beltrami equation, see \cite[Theorem 2.5.4]{Astala2009}. Given a homeomorphic $W^{1,2}_{\loc} (\Omega)$-mapping $f:\Omega \to \Omega^\prime$, write $\mu (z) = \bar\partial f(z) / \partial f(z)$ when $\partial f(z)\not=0$ and $\mu(z)=0$ otherwise. With this, we have the equivalence
	$$
	f \text{ is } K\text{-quasiconformal} \iff  \bar\partial f = \mu \partial f \text{ a.e.\ in }\Omega,
	$$
	where $\mu$, called the Beltrami coefficient of $f$, is a bounded measurable function satisfying $\|\mu\|_\infty \leq (K-1)/(K+1) = k < 1$.
	
	\begin{notation}
		With this relation, given $\mu$ as above, we say that a function $f \in W^{1,2}_{\loc} (\Omega)$ is $\mu$-quasiconformal if $f$ solves the Beltrami equation $\bar \partial f = \mu \partial f$ a.e.\ in $\Omega$.
	\end{notation}
	
	The measurable Riemann mapping theorem ensures the existence (and uniqueness up to normalization) of solutions to a Beltrami equation, see \cite[Theorem 5.3.4]{Astala2009}.
	
	\begin{theorem}[Measurable Riemann mapping theorem]\label{riemann mapping thm}
		Let $|\mu| \leq k < 1$ for almost every $z\in\C$. Then there is a solution $f:\C\to \C$ to the Beltrami equation
		$$
		\bar \partial f = \mu \partial f \text{ a.e.\ in } \C,
		$$
		which is a $K$-quasiconformal homeomorphism normalized by the three conditions $f(0)=0$, $f(1)=1$ and $f(\infty)=\infty$. Furthermore, the normalized solution $f$ is unique.
	\end{theorem}
	
	For a compactly supported Beltrami coefficient $\mu$, we call principal solution to the function $f\in W^{1,2}_{\loc}(\Omega)$ with $\bar\partial f=\mu\partial f$ normalized by the condition $f(z)=z+\OO(1/z)$ near infinity.
	
	\subsection{Distortion under quasiconformal mappings}\label{sec:distortion under qc mappings}
	
	Now we turn our attention to the distortion of the Hausdorff dimension and measure under $K$-quasiconformal mappings. Astala in \cite[Theorem 1.1]{Astala1994} (see also \cite[Theorem 13.1.5]{Astala2009}) proved the area distortion theorem, that is, there is a constant $C_K$ depending only on $K\geq 1$, such that for any $K$-quasiconformal mapping $f:\C \to \C$, for any disk $B\subset \C$ and for any subset $E\subset B$, we have
	\begin{equation}\label{area distortion theorem}
		\frac{1}{C_K} \left(\frac{|E|}{|B|}\right)^K \leq
		\frac{|f(E)|}{|f(B)|}
		\leq C_K \left(\frac{|E|}{|B|}\right)^{1/K} .
	\end{equation}
	As a corollary, see \cite[Corollary 1.3]{Astala1994} and also \cite[Theorem 13.2.10]{Astala2009}, the area distortion theorem governs the distortion of Hausdorff dimension
	\begin{equation}\label{dimension distortion theorem}
		\frac{1}{K} \left( \frac{1}{\dim_\HH (E)} - \frac{1}{2} \right)
		\leq \frac{1}{\dim_\HH (f(E))} - \frac{1}{2}
		\leq K \left( \frac{1}{\dim_\HH (E)} - \frac{1}{2} \right) .
	\end{equation}
	
	In the preprint \cite{Tolsa2009quasiconformal-arxiv} (see also \cite[Theorem 1.2]{Astala2013} for the published version), Tolsa generalized the previous two results by obtaining the control of the Hausdorff measure.
	
	\begin{theorem}[Distortion of Hausdorff measure]\label{distortion of hausdorff measure thm}
		Let $0<t<2$ and denote $t^\prime = \frac{2Kt}{2+(K-1)t}$. Let $f  :\C \to \C$ be $K$-quasiconformal. For any ball $B$ and any compact set $E\subset B$, we have
		\begin{equation}\label{distortion of hausdorff measure}
			\frac{\HH^{t^\prime} (f(E))}{\diam(f(B))^{t^\prime}} 
			\leq C_{K,t} \left( \frac{\HH^t (E)}{\diam(B)^t} \right)^{\frac{t^\prime}{Kt}} ,
		\end{equation}
        In particular, if $\HH^t (E)$ is finite (resp. $\sigma$-finite), then also $\HH^{t^\prime} (f(E))$ is finite (resp. $\sigma$-finite).
	\end{theorem}
	
	The case $t=t^\prime = 2$ is precisely the area distortion theorem in \rf{area distortion theorem}. The $\sigma$-finiteness in the particular case $t^\prime=1$ had been solved in \cite{Astala2008}. The same control in \rf{distortion of hausdorff measure} with Hausdorff content $\HH_\infty$ instead of Hausdorff measures was proved in \cite{Lacey2010}.
	
	\subsubsection{Distortion of gauge Hausdorff measures}

    Following Carathéodory's construction (see \cite[p.~54]{Mattila1995}), we define (gauge) spherical measures as in \cite[Section 2.3]{Astala2013}. Let $\mathcal B$ denote the family of all closed balls contained in $\C$, and let $h:\mathcal B \to [0,\infty)$ be any function defined in $\mathcal B$ with $h(\overline{B(x,r)})\to 0$ as $r\to 0$, for all $x\in \C$. Given $0<\delta \leq \infty$ and $E\subset \C$, we set
	$$
	\SSS_\delta^h (E) \coloneqq \inf\left\{ \sum_i h (\overline{B_i}) : E \subset \bigcup_i \overline{B_i}, \text{ with } \overline{B_i}\in \mathcal B \text{ and } r(B_i) \leq \delta \right\} .
	$$
	The $h$-spherical measure $\SSS^h$ is defined as
	\begin{equation*}
		\SSS^h (E) \coloneqq \lim_{\delta \to 0} \SSS^h_\delta (E) =\sup_{\delta >0} \SSS^h_\delta (E) ,
	\end{equation*}
	as $\SSS^h_\delta (E)$ increases as $\delta$ decreases.
	The $\delta = \infty$ case $\SSS_\infty^h$ is called $h$-spherical content. It is clear that $\SSS_\infty^h (E) \leq \SSS^h (E)$. Despite the converse inequality does not hold, $\SSS_\infty^h (E)=0$ implies $\SSS^h (E)=0$ if the function $r\mapsto h(\overline{B(x,r)})$ is nondecreasing for all $x\in \C$. Although $\SSS_\infty^h$ is not measure, in this case we write $\SSS^h \ll \SSS_\infty^h \ll \SSS^h$.

    With abuse of notation, we write $\SSS^h = \SSS^{\widetilde h}$ if $h(\overline B)=\widetilde h(r_B)$, $\overline{B}\in\mathcal B$, for some continuous nondecreasing function $\widetilde h:[0,\infty)\to [0,\infty)$ with $\widetilde h(0)=0$. If in addition $C_h\coloneqq\limsup_{s\to 0} \widetilde h(4r)/\widetilde h(r)<\infty$, a routinary computation gives
    $$
    \HH^{\widetilde h(\cdot/2)} (E) \leq \SSS^h (E) \leq C_h \HH^{\widetilde h(\cdot/2)} (E) \text{ for any set }E\subset \C.
    $$
	
	Let us introduce the notation used in \cref{distortion gauge functions general} below, see \cite[Section 2.3]{Astala2013}. Let $\varepsilon : \mathcal B \to [0,\infty)$ be any function defined on $\mathcal B$, the family of all closed balls in $\C$. We set $\varepsilon (x,r) \coloneqq \varepsilon (B(x,r))$. Given an arbitrary bounded set $S\subset \C$, let $\mathcal B (S)$ denote the family of balls with minimal diameter containing $S$, and set $\varepsilon (S) \coloneqq \inf_{B \in \mathcal B (S)} \varepsilon (B)$.
	
	As in \cite[Section 2.4]{Astala2013}, the families $\mathcal G_1$ and $\mathcal G_2$ are defined as:
	\begin{enumerate}[label=($\mathcal{G}_{\arabic*}$)]
		\item We say that a function $\varepsilon : \mathcal B \to [0,\infty)$ belongs to $\mathcal G_1$ if there is a constant $C_1$ such that
		$$
		C_1^{-1} \varepsilon (x,r) \leq \varepsilon (y,s) \leq C_1 \varepsilon (x,r)
		$$
		whenever $|x-y|\leq 2r$ and $r/2 \leq s \leq 2r$.
		
		\item For each fixed $0<t<2$, a function $\varepsilon = \varepsilon (x,r)$ belongs to the family $\mathcal G_2 = \mathcal G_2 (t)$ if there is a constant $C_2$ such that
		$$
		\sum_{k\geq 0} 2^{-k (2-t)} \varepsilon (x,2^k r) \leq C_2 \varepsilon (x,r) .
		$$
	\end{enumerate}

    We present a technical lemma used in the proof in \cite{Astala2013} of \cref{distortion of hausdorff measure thm}.
	
	\begin{lemma}[See {\cite[Lemma 4.4]{Astala2013}}]\label{distortion gauge functions general}
		Let $0<t<2$ and $t^\prime = \frac{2Kt}{2+(K-1)t}$. Let $\varepsilon \in \mathcal G_1$ and set $h(x,r)=r^{t^\prime} \varepsilon (x,r)$. Suppose that for any principal $K$-quasiconformal mapping $\psi :\C \to \C$ the function $(\varepsilon \circ \psi)^d$ belongs to $\mathcal G_2$ for any $\frac{t^\prime}{Kt} \leq d \leq 1$. Let $\phi : \C \to \C$ be a principal $K$-quasiconformal mapping, conformal on $\C \setminus \overline{\mathbb D}$, and set
		$$
		\widetilde \varepsilon (x,r) = \varepsilon (\phi (B(x,r)))^{\frac{Kt}{t^\prime}}, \quad \widetilde h (x,r) = r^t \widetilde \varepsilon (x,r).
		$$
		Then for every compact set $E\subset B(0,1/2)$ we have
		$$
		\SSS_\infty^h (\phi (E)) \leq C_K \SSS_\infty^{\widetilde h} (E)^{\frac{t^\prime}{Kt}}.
		$$
	\end{lemma}

    To study the distortion under quasiconformal mappings of the gauge Hausdorff measure in Makarov's \cref{makarov thm}, we aim to apply the previous lemma to $h=\varphi_{1,C}$, where $\varphi_{\rho,C}$ is defined in \rf{makarov gauge function}. We now recall some properties of $\varphi_{1,C}(r)/r$ for sufficiently small $r>0$. For $s>e$, we set\footnote{Note that $\mathbb{L}\mathrm{og} (s)$ is not defined when $0<s\leq e$, and $\mathbb{L}\mathrm{og} (s)\geq 0$ whenever $s\geq e^e$.}
    $$
    \mathbb{L}\mathrm{og} (s) \coloneqq \log s \log \log \log s,
        $$
    and for $s_0 > e^e$ and $s>0$, we set
    $$
    \mathbb{L}\mathrm{og}_{s_0} (s) \coloneqq
    \begin{cases}
        \mathbb{L}\mathrm{og} (s_0),&\text{if } 0<s\leq s_0,\\
        \mathbb{L}\mathrm{og} (s),&\text{if } s>s_0.
    \end{cases}
    $$
    Note that this function is nondecreasing for $s>0$.
    
    Given $C>0$ and $K\geq 1$, we define the nonincreasing function
    $$
    \varepsilon(r)\coloneqq \exp \left(C\sqrt{\mathbb{L}\mathrm{og}_{r_0^{-1}} (r^{-1})} \right),
    $$
    and we fix $0<r_0 = r_0(C,K)<e^{-e}$ (the existence of such $r_0$ is verified below) sufficiently small so that
    \begin{subequations}
        \begin{align}
            \label{eq:choice 1 of r0}&r^{\frac{1}{2K}} \varepsilon(r) \text{ is nondecreasing for }r>0, \text{ and}\\
            \label{eq:choice 2 of r0}&\varepsilon(r/2)\leq 2\varepsilon(r) \text{ for all }0<r\leq r_0.
        \end{align}
    \end{subequations}
    The condition \rf{eq:choice 1 of r0} implies that $\varphi$ and $\widetilde\varphi$ in \cref{distortion of makarov gauge function} below are nondecreasing for $r>0$, and \rf{eq:choice 2 of r0} guarantees that
    \begin{equation}\label{eq:choice 2 of r0 for all r}
    \varepsilon(r/2)\leq 2\varepsilon(r)\text{ for all }r>0.
    \end{equation}
    Indeed, $\varepsilon(r)$ is constant for $r>r_0$, and for $r_0<r\leq 2r_0$, since $\varepsilon(r)$ is nonincreasing, we obtain $\varepsilon(r/2)\leq \varepsilon(r_0/2)\leq 2\varepsilon(r_0)\leq 2\varepsilon(r)$, as claimed.
    
    The parameter $r_0>0$ in \rf{eq:choice 1 of r0} and \rf{eq:choice 2 of r0} exists because
    $$
    \frac{\varepsilon(r/2)}{\varepsilon(r)} 
    = \exp\left( C \left( \sqrt{Q_2(r)}-1 \right) \sqrt{\log \frac{1}{r} \log \log \log \frac{1}{r}} \right) \to 1 \text{ as }r\to 0,
    $$
    where $Q_2$ is the particular case $T=2$ of the function $Q_T$, defined for a fixed constant $T>0$ and sufficiently small $r>0$ as
    \begin{equation}\label{eq:function Q_T}
    Q_T(r)\coloneqq
    \frac{\log(T\frac{1}{r})\log\log\log(T\frac{1}{r})}{\log \frac{1}{r} \log \log \log \frac{1}{r}}
    =\left(\frac{\log T}{\log \frac{1}{r}}+1\right)
    \left(
    \frac{
        \log \left(
        \frac{
            \log\left(
            \frac{\log T}{\log\frac{1}{r}}
            +1
            \right)
        }{\log \log \frac{1}{r}}
        +1
        \right)
    }{\log \log \log \frac{1}{r}}
    +1
    \right),
    \end{equation}
    which satisfies
    \begin{equation}\label{eq:limit involving Q_T}
    \lim_{r\to 0} \left(\sqrt{Q_T(r)}-1\right)\sqrt{\log\frac{1}{r}\log\log\log\frac{1}{r}} = 0.
    \end{equation}
	
	Here we state the distortion of the gauge function in \cref{makarov thm} under quasiconformal mappings. The proof will be a routine computation to check the conditions of the previous lemma for this particular case. For this reason, its proof is located at the end of the subsection.
	
	\begin{coro}\label{distortion of makarov gauge function}
		Let $C>0$, $K\geq 1$, $0<r_0=r_0(C,K)<e^{-e}$ as in \rf{eq:choice 1 of r0} and \rf{eq:choice 2 of r0}, and $\phi : \C \to \C$ be a principal $K$-quasiconformal mapping and conformal on $\C \setminus \overline{\mathbb D}$. Then there exists a constant $C_{K,\phi}$ such that if for $r>0$ we define
		$$
		\varphi (r) = r \exp \left(C\sqrt{\mathbb{L}\mathrm{og}_{r_0^{-1}} (r^{-1})} \right),
        \quad
		\widetilde \varphi (r) = r^{\frac{2}{K+1}} \exp \left( C \frac{2K}{K+1} \sqrt{\mathbb{L}\mathrm{og}_{r_0^{-1}} \left(C^{-1}_{K,\phi} r^{-K} \right) } \right),
		$$
		then for every compact set $E\subset B(0,1/2)$ we have
		$$
		\SSS_\infty^\varphi (\phi (E)) \leq C_{K} \SSS_\infty^{\widetilde \varphi} (E)^{\frac{K+1}{2K}}.
		$$
	\end{coro}
	
	Since for sets with zero Hausdorff or spherical measure and content we only need to consider the gauge function for small values, and the relations $\SSS^h \ll \SSS_\infty^h \ll \SSS^h$ and $\HH^{\widetilde h(\cdot/2)} \ll \SSS^h \ll \HH^{\widetilde h(\cdot/2)}$ hold when $h(E)=\widetilde h(\diam(E)/2)$ for some function nondecreasing function $\widetilde h:[0,\infty)\to [0,\infty)$ satisfying $\widetilde h(r)\to 0$ as $r\to 0$ and $\limsup_{r\to 0} \widetilde h(4r)/\widetilde h(r)<\infty$, we obtain the following immediate consequence of \cref{distortion of makarov gauge function}. 
	
    \begin{coro}\label{distortion zero measure of makarov gauge function}
        Let $\phi : \C \to \C$ be a principal $K$-quasiconformal mapping and conformal on $\C \setminus \overline{\mathbb D}$, and $\varphi_{\cdot,\cdot}$ as in \rf{makarov gauge function}. Then $\HH^{\varphi_{K,C}} (E) = 0$ implies $\HH^{\varphi_{1,C}} (\phi (E)) = 0$ for all $E\subset B(0,1/2)$.
        \begin{proof}
            By the discussion before this lemma, it remains to see $\limsup_{r\to 0} \varphi_{\rho,C}(4r)/\varphi_{\rho,C}(r)<\infty$. For any constant $T>0$, recall the function $Q_T$ in \rf{eq:function Q_T}. For any constant $S>0$, by \rf{eq:limit involving Q_T} we have that if $r>0$ is small enough, then
            $$
            \begin{aligned}
                \varphi_{\rho,C} (Sr) &= S^{\frac{2}{\rho+1}}\varphi_{\rho,C} (r) \exp\left(C \frac{2\rho}{\rho+1}\left(\sqrt{Q_{S^{-\rho}}(r^\rho)}-1\right)\sqrt{\log r^{-\rho}\log\log\log r^{-\rho}}\right)\\
                &\leq 1.1 S^{\frac{2}{\rho+1}} \varphi_{\rho,C} (r),
            \end{aligned}
            $$
            and the result follows from this and \cref{distortion of makarov gauge function}.
        \end{proof}
    \end{coro}
	
	\begin{proof}[Proof of \cref{distortion of makarov gauge function}]
		Take $t^\prime = 1$, $t=\frac{2}{K+1}$, $\varepsilon (x,r) \coloneqq \varphi (r) / r$ and $h(x,r) \coloneqq r^{t^\prime} \varepsilon (x,r)$ in \cref{distortion gauge functions general}. Note that in this case the functions $\varepsilon$ and $h$ do not depend on $x$. So, we write $\varepsilon (r)$ and $h(r)$ instead of $\varepsilon (x,r)$ and $h(x,r)$ respectively. Recall that
		$$
		\varepsilon (r) \coloneqq \frac{\varphi (r)}{r} =  \exp \left(C\sqrt{\mathbb{L}\mathrm{og}_{r_0^{-1}} (r^{-1})} \right)
		$$
		is nonincreasing with $\lim_{r\to 0^+} \varepsilon (r) = +\infty$.
		
		For now, assume we may invoke \cref{distortion gauge functions general} (we will verify its hypotheses later). Let us first show how this result implies the corollary. By the choice of $t$, $t^\prime$, $\varepsilon(r)$ and $h(r)$ at the beginning of this proof, setting
		$$
		\widetilde \varepsilon (x,r) \coloneqq \varepsilon (\phi (B(x,r)))^{\frac{2K}{K+1}},
		\quad 
		\widetilde h(x,r) \coloneqq r^{\frac{2}{K+1}} \widetilde \varepsilon (x,r),
		$$
		by \cref{distortion gauge functions general} we obtain
		\begin{equation}\label{aplicacion of distortion gauge measures}
			\SSS_\infty^\varphi (\phi (E)) \leq C_K \SSS_\infty^{\widetilde h} (E)^{\frac{K+1}{2K}}.
		\end{equation}
		Now we want to remove the dependence of the quasiconformal mapping $\phi$ on the gauge function $\widetilde h$. Recall $\varepsilon (\phi (B(x,r))) = \inf_{B \in \mathcal{B} (\phi(B(x,r)))} \varepsilon (r_B)$ where $\mathcal{B} (\phi(B(x,r)))$ denotes the family of balls with minimal diameter containing $\phi(B(x,r))$. Since $\varepsilon$ does not depend on the point $x$, we get 
		$$
		\varepsilon (\phi (B(x,r))) = \inf_{B \in \mathcal{B} (\phi(B(x,r)))} \varepsilon (r_B)= \varepsilon \left(\frac{\diam(\phi(B(x,r)))}{2} \right) .
		$$
		From the local Hölder continuity of $K$-quasiconformal mappings in \cref{qc are locally holder}, there exists $C_{K,\phi}$ with $C_{K,\phi}^{-1} r^K \leq \diam(\phi(B(x,r)))/2 \leq C_{K,\phi} r^{1/K}$ for all $B(x,r)\subset B(0,100)$, and since $\varepsilon$ is nonincreasing we have
		$$
		\varepsilon (\phi (B(x,r))) \leq \varepsilon (C_{K,\phi}^{-1} r^K) .
		$$
		This implies 
		$$
		\widetilde h(x,r) 
		= r^{\frac{2}{K+1}} \varepsilon (\phi (B(x,r)))^{\frac{2K}{K+1}}
		\leq r^{\frac{2}{K+1}} \varepsilon (C_{K,\phi}^{-1} r^K)^{\frac{2K}{K+1}} 
		=\widetilde \varphi (r),
		$$
		which together with \rf{aplicacion of distortion gauge measures} concludes the lemma.
		
		It remain to check the assumptions in \cref{distortion gauge functions general}, i.e., $\varepsilon \in \mathcal G_1$, and $(\varepsilon \circ \psi)^d \in \mathcal G_2$ for any principal $K$-quasiconformal mapping $\psi :\C \to \C$ and any $\frac{t^\prime}{Kt} \leq d \leq 1$.

        Since the function $\varepsilon$ does not depend on the point, to verify $\varepsilon \in \mathcal G_1$ it suffices to check that $2^{-1} \varepsilon (r) \leq \varepsilon (s) \leq 2\varepsilon (r)$ whenever $r/2 \leq s \leq 2r$. It follows directly from \rf{eq:choice 2 of r0 for all r} and the fact that $\varepsilon$ is nonincreasing, because $\frac{1}{2} \varepsilon (r) \leq  \varepsilon (2r)\leq \varepsilon (s) \leq \varepsilon (r/2) \leq 2 \varepsilon (r)$, as claimed.
		
		Let us see now that $(\varepsilon \circ \psi)^d \in \mathcal G_2$ for any principal $K$-quasiconformal mapping $\psi :\C \to \C$ and any $\frac{t^\prime}{Kt} \leq d \leq 1$, that is, there exists a constant $C_2 >0$ such that
		$$
		\sum_{k\geq 0} 2^{-k (2-t)} \left(( \varepsilon \circ \psi )(x,2^k r)\right)^d \leq C_2 \left((\varepsilon \circ \psi) (x,r)\right)^d .
		$$
		Recall that $\mathcal{B} (S)$ denotes the family of balls with minimal diameter containing the set $S$, $(\varepsilon \circ \psi) (x,s) = \varepsilon (\psi (B(x,s))) = \inf_{B \in \mathcal{B} (\psi(B(x,s)))} \varepsilon (r_B)$, and that the function $\varepsilon$ does not depend on the point. Since $\psi$ is an homeomorphism, we have $\psi (B(x,2^k r)) \subset \psi(B(x,2^{k+1} r))$ for any $k\geq 0$. In particular, a minimizer ball for $\psi(B(x,2^{k+1} r))$ contains the set $\psi(B(x,2^{k} r))$, and therefore the minimizer ball of $\psi(B(x,2^{k} r))$ has smaller diameter than the minimizer ball of $\psi(B(x,2^{k+1} r))$. This and the fact that the function $\varepsilon$ is nonincreasing imply $(\varepsilon \circ \psi) (x,2^{k+1} r) \leq (\varepsilon \circ \psi) (x,2^{k} r)$ for any $k\geq 0$. Consequently, we have
		$$
		\sum_{k\geq 0} 2^{-k (2-t)} \left(( \varepsilon \circ \psi )(x,2^k r)\right)^d 
		\leq \left((\varepsilon \circ \psi) (x,r)\right)^d \sum_{k\geq 0} 2^{-k (2-t)}.
		$$
        That is, the $\mathcal{G}_2$ condition is satisfied with $C_2 \coloneqq \sum_{k\geq 0} 2^{-k (2-t)} < \infty$.
	\end{proof}

    \section{Harmonic factorization of \texorpdfstring{$L_A$}{A}-harmonic functions}

    Given a domain $\Omega\subset \R^2$, a uniformly elliptic matrix $A$, and an $L_A$-harmonic function $u$ in $\Omega$, in this section we see how to find a quasiconformal mapping $\phi$ such that $u\circ\phi^{-1}$ is harmonic in $\phi(\Omega)$. First, let us see how $L_A$-harmonic functions behave under quasiconformal mappings.

	\subsection{\texorpdfstring{$L_A$}{A}-harmonic functions under quasiconformal mappings}
	
	Given a uniformly elliptic matrix $A\in \R^{2\times 2}$ and a global quasiconformal mapping $f$ of $\C$, we define the uniformly elliptic matrix $\widetilde{A}_f$ as
	\begin{equation}\label{matrix pushforward}
		\widetilde A_f (x) \coloneqq (\det Df(x)) Df(x)^{-1} A(f(x)) (Df(x)^{-1})^T.
	\end{equation}
	We remark that the matrix $\widetilde A_f$ is uniformly elliptic by the quasiconformal condition of $f$, see \cite[Lemma 14.38]{Heinonen2006}, and in particular, planar uniformly elliptic operators are invariant under quasiconformal homeomorphisms. Moreover, these matrices satisfy
	$$
	\widetilde{(\widetilde A_f)}_{f^{-1}} = A = \widetilde{(\widetilde A_{f^{-1}})}_f.
	$$
	
	With the notation in \rf{matrix pushforward}, we observe the behavior of $L_A$-harmonic functions and elliptic measures under quasiconformal changes of variables in the following two lemmas.
	
	\begin{lemma}[See {\cite[14.39]{Heinonen2006}}]\label{linear change of variables qc}
		Let $\Omega\subset \R^2$ be an open set, $A\in \R^{2\times 2}$ be a uniformly elliptic matrix, and $f$ be a quasiconformal homeomorphism of $\C$. Then $u$ is $L_A$-harmonic in $\Omega$ if and only if $u\circ f$ is $L_{\widetilde A_f}$-harmonic in $f^{-1}(\Omega)$.
	\end{lemma}
	
	\begin{lemma}[See {\cite[14.51]{Heinonen2006}}]\label{elliptic measure deformation qc}
		Let $\Omega\subset \R^2$ be a domain, $A\in \R^{2\times 2}$ be a uniformly elliptic matrix and $f$ be a quasiconformal homeomorphism of $\C$. For any $x\in \Omega$ and $E\subset \partial\Omega$, we have
		$$
		\omega_{\Omega,A}^{x} (E)=\omega_{f^{-1}(\Omega),\widetilde A_f}^{f^{-1}(x)} (f^{-1}(E)).
		$$
	\end{lemma}
	
	For a detailed proof for bi-Lipschitz mappings see \cite[Lemmas 3.8 and 3.9]{Azzam2019}. For a more general version of these results see {\cite[14.39 and 14.51]{Heinonen2006}} respectively.
	
	\subsection{The \texorpdfstring{$L_A$}{A}-harmonic conjugate in simply connected domains}\label{section:general harmonic conjugate}
	
	This section is a brief summary of \cite[Section 16.1.3]{Astala2009}. Let $\Omega \subset \R^2$ be a bounded simply connected domain. We remark that the simply connected condition on $\Omega$ used in this section is just an auxiliary assumption to be able to define the auxiliary $L_A$-harmonic conjugate function below.
	
	Consider $u$ a solution to
	$$
	\divv A(z) \nabla u(z) =0, \quad u\in W^{1,2}_{\loc} (\Omega),
	$$
	and consider the vector fields $\nabla u$ and $A(z) \nabla u$. The vector field $\nabla u$ has zero curl since $\curl \nabla u = \partial_y \partial_x u - \partial_x \partial_y u =0$, and $A(z) \nabla u$ is divergence-free as a solution to the previous equation.
	
	The Hodge star operator $*$ is defined as
	$$
	* \coloneqq \begin{pmatrix}
		0 & -1 \\
		1 & 0
	\end{pmatrix} : \R^2 \to \R^2,
	$$
	which is simply the counterclockwise rotation by $90$ degrees, and so $* * = -Id$. The Hodge star operator $*$ transforms curl-free fields into divergence-free fields and vice versa. Therefore
	$$
	\divv (*\nabla u) = 0 \text{ \, and } \curl (*A(z)\nabla u)=0 .
	$$
	
	By the Poincaré lemma\footnote{It is at this step that we need to work with simply connected domains $\Omega \subset \R^2$.}, see \cite[Lemma 16.1.3]{Astala2009} for instance, we have that $*A(z)\nabla u$ is a gradient field. That is, there exists a real-valued function $v\in W^{1,2}_{\loc} (\Omega)$, called the $L_A$-harmonic conjugate of $u$, such that 
	\begin{equation}\label{relation derivatives harmonic conjugate}
		\nabla v = *A(z) \nabla u .
	\end{equation}
	
	Define now
	$$
	A^* \coloneqq *^t A^{-1} * = -* A^{-1} *,
	$$
	which is also uniformly elliptic and
	$$
	A^* (z) \nabla v = A^* (z) (*A(z) \nabla u) = (-* A^{-1} *) (*A(z) \nabla u) = *\nabla u .
	$$
	Recall that $\curl \nabla u = 0$, and therefore $\divv (*\nabla u) = 0$, i.e., the function $v$ satisfies
	$$
	\divv A^* (z) \nabla v =0, \quad v\in W^{1,2}_{\loc} (\Omega).
	$$
	In the converse direction, the existence of $v\in W^{1,2}_{\loc} (\Omega)$ such that $\nabla v = *A\nabla u$ implies $\divv A\nabla u = 0$, because $\curl (*A\nabla u) = \curl \nabla v = 0$, the Hodge star operator $*$ sends zero-curl to zero-divergence and $**=-Id$. As before, $\divv A\nabla u = 0$ implies $\divv A^* \nabla v = 0$.
	
	\begin{rem}[Symmetric matrices with determinant $1$]\label{sym and det 1 then A*=A}
		Wherever the matrix is symmetric with determinant $1$ then $A^*=A$, and hence both $u$ and its $L_A$-harmonic conjugate function $v$ are $L_A$-harmonic.
	\end{rem}
	
	The following theorem describes the strong connection between the Beltrami equation and PDEs.
	
	\begin{theorem}[Weak {\cite[Theorem 16.1.6]{Astala2009}}]\label{thm relation beltrami and pde}
		Let $\Omega \subset \R^2$ be a simply connected domain, and let $u\in W^{1,1}_{\loc} (\Omega)$ be a solution 
		$$
		\divv A(z) \nabla u = 0.
		$$
		If $v\in W^{1,1} (\Omega)$ is a solution to $\nabla v = *A \nabla u$, then the function $f=u+iv$ satisfies the generalized Beltrami equation
		\begin{equation}\label{generalized Beltrami equation}
			\bar\partial f = \mu_A(z) \partial f + \nu_A(z) \overline{\partial f},
		\end{equation}
		where the coefficients $\mu_A,\nu_A$ are given by
		\begin{equation}\label{relation coefficients beltrami and pde}
			\begin{cases}
				\mu_A =& \frac{1}{\det (Id + A)} \left( a_{22}-a_{11}-i(a_{12}+a_{21}) \right), \\
				\nu_A =& \frac{1}{\det (Id + A)} \left( 1-\det A + i(a_{12}-a_{21}) \right) .
			\end{cases} 
		\end{equation}
		Conversely, if $f\in W^{1,1}_{\loc} (\Omega)$ is a mapping satisfying the generalized Beltrami equation \rf{generalized Beltrami equation}, then $u=\real f$ and $v = \imag f$ satisfy $\nabla v = *A\nabla u$ with $A$ given by solving the complex equations in \rf{relation coefficients beltrami and pde}.
	\end{theorem}
	
	Notice also from \rf{relation coefficients beltrami and pde} that when the matrix is symmetric and with $\det A=1$, then $\nu_A=0$ and \rf{generalized Beltrami equation} reads as the classical Beltrami equation
	$$
	\bar\partial f = \mu_A(z) \partial f \text{ with }\mu_A = \frac{ a_{22} - a_{11} - 2ia_{12}}{2+a_{11} + a_{22}}.
	$$
	In this case, we have the following relations:
	\begin{coro}\label{all relations beltrami and pde}
		Let $A = \left(a_{ij}\right)_{i,j=1}^2$ be a uniformly elliptic matrix. Wherever $A$ is symmetric with determinant 1, then the Beltrami coefficients in \rf{relation coefficients beltrami and pde} are
		\begin{equation}\label{beltrami coefficients for sym and det 1 matrices}
			\mu_A = \frac{ a_{22} - a_{11} - 2i a_{12}}{2+a_{11} + a_{22}} , \quad \nu_A = 0,
		\end{equation}
		and given a domain $\Omega\subset\R^2$ and $f\in W^{1,1}_{\loc} (\Omega)$, the following are equivalent:
		\begin{enumerate}
			\item \label{thm qc pde cond1}$\bar \partial f = \mu_A (z) \partial f$,
			\item \label{thm qc pde cond2}$\nabla (\imag f) = *A\nabla (\real f)$,
			\item \label{thm qc pde cond3}$\bar \partial (f^{-1}) = - \mu_A (f^{-1} (z)) \overline{\partial (f^{-1})}$,
			\item \label{thm qc pde cond4}$\nabla \left(\imag (f^{-1}) \right) = *B\nabla \left(\real (f^{-1}) \right)$,
		\end{enumerate}
		where
		\begin{equation}\label{matrix of inverse function beltrami}
			B \coloneqq \begin{pmatrix}
				1/a_{11} & a_{12}/a_{11} \\
				-a_{12}/a_{11} & 1/a_{11}
			\end{pmatrix}  \circ f^{-1}
			\text{ and }
			B^*
			=\begin{pmatrix}
				1/a_{22} & -a_{12}/a_{22} \\
				a_{12}/a_{22} & 1/a_{22}
			\end{pmatrix}  \circ f^{-1}.
		\end{equation}
		are uniformly elliptic. In this case, we also have,
		\begin{align*}
			&\divv A\nabla (\real f) = 0, & &\divv  A\nabla (\imag f) = 0,\\
			&\divv B \nabla \left(\real (f^{-1}) \right) = 0, & &\divv B^* \nabla \left(\imag (f^{-1}) \right) = 0 .
		\end{align*}
		\begin{proof}
			During the proof we write $\mu=\mu_A$. From \cref{thm relation beltrami and pde} we have $\bar \partial f = \mu \partial f$ if and only if $\nabla (\imag f) = *A\nabla (\real f)$, with Beltrami coefficient $\mu$ as in \rf{beltrami coefficients for sym and det 1 matrices}, which gives $\rf{thm qc pde cond1}\Leftrightarrow \rf{thm qc pde cond2}$.
			
			By \cite[Theorem 5.5.6 or Lemma 10.3.1]{Astala2009}, $\bar \partial f = \mu (z) \partial f$ if and only if
			\begin{equation}\label{beltrami equation inverse}
				\bar\partial (f^{-1}) = - \mu (f^{-1} (z)) \overline{\partial (f^{-1})},
			\end{equation}
			and so $\rf{thm qc pde cond1}\Leftrightarrow \rf{thm qc pde cond3}$. Hence, from \cref{thm relation beltrami and pde} there exists\footnote{Actually, \cref{thm relation beltrami and pde} gives us how to find the matrix $B$ from $A$. Directly from \rf{beltrami equation inverse} and \rf{relation coefficients beltrami and pde} we deduce that the matrix $B$ is antisymmetric with a common diagonal term.} a matrix $B$ such that (\ref{beltrami equation inverse}) is equivalent to having $\nabla \left(\imag (f^{-1}) \right) = *B\nabla \left(\real (f^{-1}) \right)$. The matrix $B$ defined in \rf{matrix of inverse function beltrami} satisfies (\ref{relation coefficients beltrami and pde}) with $\nu = -\mu \circ f^{-1}$, i.e.,
			$$
			-\frac{ a_{22} - a_{11} - 2ia_{12}}{2+a_{11} + a_{22}} =-\mu 
			= \frac{1-\det (B \circ f)  + 2 i b_{12}}{\det(Id + B \circ f)} ,
			$$
			and hence we have $\rf{thm qc pde cond3} \Leftrightarrow \rf{thm qc pde cond4}$. Moreover, a routine computation $B^* = -\ast B^{-1} \ast$ gives the second matrix in \rf{matrix of inverse function beltrami}.
			
			It remains to see the last part. Since $\nabla \left(\imag (f^{-1}) \right) = *B\nabla \left(\real (f^{-1}) \right)$ by \rf{thm qc pde cond4}, we have that 
			\begin{equation*}
				\divv B \nabla \left(\real (f^{-1}) \right)=0, \quad
				\divv B^* \nabla \left(\imag (f^{-1}) \right) = 0,
			\end{equation*}
			as we did in the paragraph before \cref{sym and det 1 then A*=A}. The same holds for the real and imaginary parts of $f$ with the matrix $A$ and $A^*$ respectively, but $A^* = A$ by \cref{sym and det 1 then A*=A}.
		\end{proof}
	\end{coro}
	
	Notice that even though the matrix $A$ is symmetric with determinant 1, the new matrix $B$ in \rf{matrix of inverse function beltrami} is no longer symmetric with determinant 1. However, we still have that $\real (f^{-1})$ (resp. $\imag (f^{-1})$) is $L_B$-harmonic (resp. $L_{B^\star}$-harmonic).
	
	The following result gives the precise distortion of a quasiconformal mapping solving the generalized Beltrami equation \rf{generalized Beltrami equation}.
	
	\begin{lemma}\label{ellipticity condition beltrami coefficient}
		Let $A\in \R^{2\times 2}$ be a uniformly elliptic matrix with ellipticity constant $\lambda \geq 1$, $\mu_A$ and $\nu_A$ as in \rf{relation coefficients beltrami and pde}, and $K_\lambda$ as in \rf{K-lambda ellipticity condition beltrami coefficient}. Then
		$$
		\||\mu_A| + |\nu_A|\|_\infty \leq \frac{K_\lambda -1}{K_\lambda + 1}<1.
		$$
		
		\begin{proof}
			See \cite[Proposition 1.8]{Alessandrini2009} for general matrices and \cite[Lemma 4.1]{Davey2017} for the symmetric case.
		\end{proof}
	\end{lemma}
	
	\begin{rem}
		Note that in both cases $K_\lambda \to 1^+$ as $\lambda \to 1^+$. In particular $\frac{K_\lambda -1}{K_\lambda + 1} \to 0^+$ as $\lambda \to 1^+$.
	\end{rem}
	
	\begin{rem}\label{rem:interpolation to preserve determinant}
		Let $A$ be a uniformly elliptic and symmetric matrix with determinant $1$. Recall from \cref{rem:modify beltrami coef instead of the matrix} we noted that the modification $\widetilde A= \varphi A + (1-\varphi)Id$ is symmetric but it has no longer determinant $1$ in general.  The right way to interpolate the matrices $A$ and $Id$ to preserve also that the modified matrix has determinant $1$ is by replacing $A$ by the uniformly elliptic (with the same ellipticity constant as $A$), symmetric and with determinant $1$ matrix
		$$
		\widehat{A}=\frac{
			\begin{pmatrix}
				a_{11}(1+\varphi)^2+a_{22}(\varphi-1)^2+2(1-\varphi^2) & 4\varphi a_{12} \\
				4\varphi a_{12} & a_{11}(\varphi -1)^2+(1+\varphi)(a_{22}(1+\varphi)+2(1-\varphi))
			\end{pmatrix}
		}{a_{11}(1-\varphi^2)+a_{22}(1-\varphi^2)+2(1+\varphi^2)},
		$$
		which agrees with $A$ when $\varphi=1$ and with $Id$ when $\varphi=0$. This is precisely the matrix such that $\mu_{\widehat A}=\varphi \mu_A$ and $\nu_{\widehat A} = \varphi \nu_A=0$, by the relation between uniformly elliptic matrices and Beltrami coefficients in \cref{thm relation beltrami and pde}. As $\|\mu_{\widehat A}\|_\infty \leq \|\mu_A\|_\infty$, in particular we have that $\widehat A$ has the same ellipticity constant as $A$.
	\end{rem}
	
	\subsection{Harmonic factorization in general domains}\label{section:harmonic factorization}
	
	In this section, we show that there exists a quasiconformal homeomorphism factorizing $L_A$-harmonic functions into harmonic functions in its image domains.

        Let $\Omega \subset \R^2$ be an open set, $A$ be a uniformly elliptic matrix, and $\mu_A$ and $\nu_A$ as in \rf{relation coefficients beltrami and pde}. Given an $L_A$-harmonic function $u\in W^{1,2}_{\loc} (\Omega)$ and any simply connected domain $D\subset \Omega$, let $v \in W^{1,2}_{\loc} (D)$ be the $L_A$-harmonic conjugate function of the $L_A$-harmonic function $u\in W^{1,2}_{\loc} (D)$. By \cref{thm relation beltrami and pde}, the function $f = u+iv \in W^{1,2}_{\loc} (D)$ satisfies the Beltrami equation
		\begin{equation*}
			\bar\partial f = \mu_A \partial f + \nu_A \overline{\partial f} \text{ a.e.\ in } D.
		\end{equation*}
		In particular, the function $f$ satisfies the classical Beltrami equation
        \begin{equation}\label{classical beltrami equation f}
            \bar\partial f = \widetilde \mu \partial f \text{ a.e.\ in }D,\text{ with } 
            \widetilde \mu \coloneqq
            \begin{cases}
				\mu_A + \nu_A \frac{\overline{\partial f}}{\partial f}  &\text{where } \overline{\partial f}/\partial f \text{ is defined},\\
				0 &\text{otherwise}.
			\end{cases}
        \end{equation}
        Note that the Beltrami coefficient $\widetilde \mu$ seems to depend on $\overline{\partial f} / \partial f$, but in fact it only depends on (the derivative of) $u$. Indeed, by the relation $\nabla v = *A(z) \nabla u$, see (\ref{relation derivatives harmonic conjugate}), we can write in $\partial f$ the terms $\partial_x v$ and $\partial_y v$ using $\partial_x u$ and $\partial_y u$. More precisely,
        \begin{equation}\label{partial f only depends on u}
        \partial f = (1+ a_{11})\partial_x u + a_{12}\partial_y u 
				- i(a_{21}\partial_x u+(1+a_{22})\partial_y u
				),
        \end{equation}
        and in particular, $\widetilde \mu$ in \rf{classical beltrami equation f} depends on $\mu_A$, $\nu_A$ and the derivatives of the function $u$, but not on the auxiliar $L_A$-harmonic conjugate function $v$ in $D$. By \rf{classical beltrami equation f} and \rf{partial f only depends on u} we can redefine $\widetilde\mu$ in whole $\Omega$ as
        \begin{equation}\label{coefficients classical beltrami equation in terms of u and A}
        \mu_{A,u} \coloneqq 
        \mu_A + \nu_A \frac{(1+ a_{11})\partial_x u + a_{12}\partial_y u 
				+ i(a_{21}\partial_x u+(1+a_{22})\partial_y u
				)}{
				(1+ a_{11})\partial_x u + a_{12}\partial_y u 
				- i(a_{21}\partial_x u+(1+a_{22})\partial_y u
				)} \text{ a.e.\ in }\Omega.
		\end{equation}
	
	The main result of this section is the following.
	
	\begin{lemma}\label{harmonic factorization}
		Let $\Omega \subset \R^2$ be an open set, $A$ be a uniformly elliptic matrix with ellipticity constant $\lambda \geq 1$, and $\mu_A$ and $\nu_A$ as in \rf{relation coefficients beltrami and pde}. For any $L_A$-harmonic function $u\in W^{1,2}_{\loc} (\Omega)$, if $\phi \in W^{1,2}_{\loc}(\Omega)$ be a homeomorphic solution of the Beltrami equation
        \begin{equation}\label{classical beltrami equation in terms of u and A}
			\bar\partial \phi = \mu_{A,u} \partial \phi \text{ a.e.\ in } \Omega, \text{ where }\mu_{A,u} \text{ as in \rf{coefficients classical beltrami equation in terms of u and A},}
		\end{equation}
	then $u\circ \phi^{-1}: \phi (\Omega) \to \R$ is harmonic.
	\end{lemma}
	
	The existence of $\phi$ is given by the Riemann mapping theorem, see \cref{riemann mapping thm}. Moreover, $\phi$ is $K_\lambda$-quasiconformal with $K_\lambda$ as in \rf{K-lambda ellipticity condition beltrami coefficient}, since $\|\mu_{A,u} \|_\infty \leq \||\mu_A|+|\nu_A|\|_\infty \leq \frac{K_\lambda -1}{K_\lambda +1} <1$ by \cref{ellipticity condition beltrami coefficient}.
	
	\begin{proof}[Proof of \cref{harmonic factorization}]
		We claim that it suffices to prove the lemma for simply connected domains. Indeed, for a general open set $\Omega$, for each point $x\in\Omega$ take an open ball $B$ with $x\in B \subset \Omega$, and in particular $u\in W_{\loc}^{1,2} (B)$ and $u$ is $L_A$-harmonic in $B$. Assuming the lemma for simply connected domains, we have that $u\circ \phi_{u,A}^{-1}: \phi_{u,A} (B) \to \R$ is harmonic. Since it is true for any $x \in \Omega$, we obtain that $u \circ \phi_{u,A}^{-1}$ is harmonic in $\phi_{u,A} (\Omega)$.
		
		Let us see the proof assuming also that $\Omega$ is simply connected. Let $v \in W^{1,2}_{\loc} (\Omega)$ be the $L_A$-harmonic conjugate function of the $L_A$-harmonic function $u\in W^{1,2}_{\loc} (\Omega)$. By \cref{thm relation beltrami and pde}, the function $f = u+iv \in W^{1,2}_{\loc} (\Omega)$ satisfies the Beltrami equation $\bar\partial f = \mu_A \partial f + \nu_A \overline{\partial f}$ a.e.\ in $\Omega$, and in particular, also safisfies the classical Beltrami equation in \rf{classical beltrami equation f}. By the discussion before the statement of this lemma, we have $\widetilde\mu=\mu_{A,u}$ in $\Omega$ and therefore both $f=u+iv \in W^{1,2}_{\loc} (\Omega)$ and $\phi \in W^{1,2}_{\loc} (\Omega)$ solve \rf{classical beltrami equation in terms of u and A} in $\Omega$. By Stoilow factorization (see \cite[Theorem 5.5.1]{Astala2009}), the function $f \circ \phi^{-1} :  \phi (\Omega) \to \C$ is holomorphic. Taking the real part on both sides we obtain
		$$
		u(\phi^{-1} (z))= \real ( f ) (\phi^{-1} (z))  =  \real \left( f \circ \phi^{-1} \right)  ( z ),
		$$
		and so $u\circ \phi^{-1} : \phi (\Omega) \to \R$ is harmonic.
	\end{proof}
	
	\subsubsection{Symmetric matrices with determinant 1}
	
	Let us move to the case where the matrix $A$ is symmetric with $\det A =1$ in $\Omega$. Recall that we simply have $K_\lambda = \lambda$ in $\Omega$ since the matrix is symmetric there, see \rf{K-lambda ellipticity condition beltrami coefficient}. Note that in this case, by \rf{beltrami coefficients for sym and det 1 matrices} in \cref{all relations beltrami and pde} we have that, for any $L_A$-harmonic function $u\in W^{1,2}_{\loc} (\Omega)$, the Beltrami coefficient $\mu_{A,u}$ in \rf{coefficients classical beltrami equation in terms of u and A} equals $\mu_A$ in $\Omega$. In particular, we directly obtain the following result.
	
	\begin{coro}\label{harmonic factorization symmetric matrices with determinant 1}
		Let $\Omega \subset \R^2$ be an open set and $A$ be a uniformly elliptic matrix with ellipticity constant $\lambda \geq 1$. Assume also that $A$ is symmetric with $\det A=1$ in $\Omega$. Let $\phi \in W^{1,2}_{\loc} (\Omega)$ be a homeomorphic solution of the Beltrami equation
		$$
		\bar \partial \phi = \mu_A \partial \phi \text{ a.e.\ in }\Omega,\text{ where } 
		\mu_A = \frac{ a_{22} - a_{11} - 2ia_{12}}{2+a_{11} + a_{22}}.
		$$
		For any $L_A$-harmonic function $u\in W_{\loc}^{1,2} (\Omega)$, then $u\circ \phi^{-1}: \phi (\Omega) \to \R$ is harmonic.
	\end{coro}
	
	As the function $\phi$ above factorizes every $L_A$-harmonic function in $\Omega$ into a harmonic function in $\phi (\Omega)$, the corollary suggests, and a straightforward computation confirms, that
	\begin{equation}\label{change of variables instead of stoilow factorization}
		\widetilde A_{\phi^{-1}} = Id \text{ a.e.\ in } \phi (\Omega),
	\end{equation}
	equivalently $\widetilde{Id}_{\phi}=A$ in a.e.\ $\Omega$; recall the definition of $\widetilde A_{\phi^{-1}}$ and $\widetilde{Id}_{\phi}$ in \rf{matrix pushforward}. That is, the factorization from $L_A$-harmonic functions to harmonic functions is not only given by the Stoilow factorization (in the proof of \cref{harmonic factorization}) but by the previous pointwise relation given by just applying the change of variables $\phi$ in the weak form of $L_A$.

    \section{Push-forward of the harmonic measure and proof of Theorems~\ref{main thm general domains} and \ref{main thm general domains and continuous matrices}}\label{part:general domains}

    This section is dedicated to the proof of \cref{main thm general domains,main thm general domains and continuous matrices}, where we work with general uniformly elliptic matrices $A\in \R^{2\times 2}$ (not necessarily symmetric nor constant determinant). To do that, we are first interested in representing the elliptic measure as a push-forward of a harmonic measure by a quasiconformal change of variables.

    For the particular case of symmetric matrices with determinant 1, the harmonic factorization of $L_A$-harmonic functions given by \rf{change of variables instead of stoilow factorization} leads us to obtain in the following lemma that the elliptic measure $\omega^p_{\Omega, A}$ is precisely the push-forward of the harmonic measure $\omega^{\phi (p)}_{\phi (\Omega), Id}$ under the map $\phi^{-1}$ given by \cref{harmonic factorization symmetric matrices with determinant 1}. The following result is a consequence of \rf{change of variables instead of stoilow factorization} and \cref{elliptic measure deformation qc}.
	
	\begin{lemma}\label{lemma:relation elliptic measure for sym and det 1}
		Under the assumptions in \cref{harmonic factorization symmetric matrices with determinant 1}, then
		$$
			\omega^p_{\Omega, A} (E) =  \omega^{\phi (p)}_{\phi (\Omega), Id} (\phi (E))\text{ for any } p\in\Omega\text{ and }E\subset \partial \Omega.
		$$
	\end{lemma}
    
    Although the same may not be true for general uniformly elliptic matrices, in this section we prove that the elliptic measure is mutually absolutely continuous with respect to a push-forward of a harmonic measure by a quasiconformal change of variables.

    \begin{lemma}\label{lemma:elliptic measure factorization in general wiener regular domains}
        Let $\Omega\subset \R^2$ be a bounded Wiener regular domain, $p\in\Omega$, $A\in\R^{2\times 2} (\Omega)$ be a uniformly elliptic matrix with constant $\lambda\geq 1$, and $K_\lambda$ as in \rf{K-lambda ellipticity condition beltrami coefficient}. Then there exists a $K_\lambda$-quasiconformal homeomorphism $\phi : \R^2 \to \R^2$, conformal in $\overline{\Omega}^c$, such that
        $$
        \omega_{\Omega,A}^p (\cdot) \ll \omega_{\phi(\Omega),Id}^{\phi(p)} (\phi(\cdot)) \ll \omega_{\Omega,A}^p (\cdot).
        $$
    \end{lemma}

    To prove this lemma, we in fact show that if we assume, in addition to the hypotheses of \cref{lemma:elliptic measure factorization in general wiener regular domains}, that $A=Id$ in an open neighborhood of $p$, then
    $$
    \omega_{\Omega,A}^p (E) \approx \omega_{\phi(\Omega),Id}^{\phi(p)} (\phi(E)) \text{ for all }E\subset\partial\Omega,
    $$
    where the constant depends on $\lambda$. The proof of this is referred to \cref{sec:the green function,sec:harmonic fact of the green function}. As we will see in \rf{eq:beltrami eq for the lemma2 about elliptic harmonic measure} below, the quasiconformal change of variables $\phi$ satisfies a Beltrami equation with Beltrami coefficient depending on the (derivatives of the) Green function for $A^T$, to be introduced in \cref{sec:the green function} below. For the particular case when the matrix is symmetric and has determinant $1$, it follows from \rf{eq:beltrami eq for the lemma2 about elliptic harmonic measure} that the Beltrami coefficient depends only on the matrix, and in fact, one recovers \cref{lemma:relation elliptic measure for sym and det 1}.

    \subsection{Proof of Theorem~\ref{main thm general domains}}\label{sec:main thm general domains}

    Assuming \cref{lemma:elliptic measure factorization in general wiener regular domains} to be true, we are now ready to prove \cref{main thm general domains}.

	\begin{proof}[Proof of \cref{main thm general domains}]\label{proof of main thm general domains}
		Recall we can assume that $\Omega \subset B_{1/4} (0)$ is bounded Wiener regular by \cref{reductions all} and we can replace finitely connected by simply connected by \cref{claim:candidate for each index}. Let $\phi:\R^2\to \R^2$ be the $K_\lambda$-quasiconformal mapping given by \cref{lemma:elliptic measure factorization in general wiener regular domains}, which is in particular conformal in $\C \setminus \overline{\mathbb D}$.

		By Wolff's \cref{wolff thm} there is a set $F'\subset \partial (\phi (\Omega))$ such that $\omega^{\phi (p)}_{\phi (\Omega), Id} (F') = 1$ and with $\sigma$-finite length. Therefore, the set $F = \phi^{-1} (F') \subset \partial \Omega$ satisfies $\omega^p_{\Omega, A} (F) = 1$ by \cref{lemma:elliptic measure factorization in general wiener regular domains}, and by the distortion of Hausdorff measures under quasiconformal mappings in \cref{distortion of hausdorff measure thm}, the set $F$ has $\sigma$-finite $\HH^{2K_\lambda / (K_\lambda+1)}$.
		
		Now consider the case where $\Omega$ is simply connected. Recall that the $K_\lambda$-quasiconformal homeomorphism $\phi$ is conformal in $\C \setminus \overline{\mathbb D}$. By \cref{lemma:elliptic measure factorization in general wiener regular domains}, Makarov's \cref{makarov thm} and \cref{distortion zero measure of makarov gauge function} respectively, there exists $C_M>0$ such that
    	$$
            \omega_{\Omega, A}^p (\cdot) \ll \omega^{\phi(p)}_{\phi (\Omega), Id} (\phi(\cdot)) \ll \HH^{\varphi_{1,C_M}} (\phi (\cdot)) \ll \HH^{\varphi_{K_\lambda,C_M}} (\cdot),
    	$$
        as claimed.
	\end{proof}

    \subsection{The Green function and elliptic measure}\label{sec:the green function}

    Let $\Omega\subset \R^{n+1}$, $n\geq 1$, be a bounded Wiener regular domain and $A\in \R^{(n+1)\times (n+1)} (\Omega)$ be a uniformly elliptic matrix. We denote by $g_x^{\Omega,A} (\cdot)$ the Green function in $\Omega$ with pole at $x\in \Omega$ of the operator $L_A = -\divv A\nabla$ constructed in \cite[Theorem 2.12]{Dong2009} in the planar case and \cite[Theorem 1.1]{Gruter1982} in higher dimensions. For all $x\in \Omega$, $g_x^{\Omega,A}$ satisfies\footnote{If $\Omega\subset\R^2$ is bounded, then $Y^{1,2}_0 (\Omega)$ in \cite[Theorem 2.12]{Dong2009} is precisely $W^{1,2}_0 (\Omega)$, see between Lemma 3.1 and Section 3.1 of \cite{Dong2009}.}
    \begin{equation}\label{sobolev regularity of green function}
    g_x^{\Omega,A}\in W^{1,2} (\Omega\setminus B_r(x)) \cap W^{1,s}_0 (\Omega), \text{ for all }s\in [1,(n+1)/n) \text{ and } 0<r<\dist(x,\partial\Omega),
    \end{equation}
    and
    \begin{equation}\label{dirac delta definition green function}
    	\int_\Omega  A (z) \nabla g_x^{\Omega,A} (z) \nabla \varphi (z) \, dz = \varphi (x),\text{ for all } \varphi \in C_c^\infty (\Omega).
    \end{equation}
    In particular, $g_x^{\Omega,A}$ is $L_A$-harmonic in $\Omega \setminus \{x\}$. By \rf{sobolev regularity of green function} and a density argument, \rf{dirac delta definition green function} remains true for $\varphi \in W^{1,s^\prime}_c (\Omega)$, with $n+1 < s^\prime < \infty$.

    We remark that $g_x^{\Omega,A}(y) \geq 0$ for each $y\in \Omega\setminus \{x\}$, see \cite[Theorem 1.1]{Gruter1982} in higher dimensions and \cite[item (4) in p.~13 of 79]{Guillen-Prats-Tolsa} for the planar case. For extra properties with proofs in the planar case, see \cite[Section 2.4]{Guillen-Prats-Tolsa} and the references therein.

    \begin{notation}
        If the set $\Omega$, the matrix $A$, or the pole $x\in \Omega$ is clear from the context, we will omit it in $g_x^{\Omega,A}$. We will mainly skip the domain $\Omega$ in $g_x^{\Omega,A}$, that is, $g_x^A$.
    \end{notation}

    The following lemma shows how the Green function relates to the elliptic measure.
    
    \begin{lemma}
    Let $\Omega\subset \R^{n+1}$, $n\geq 1$, be a bounded Wiener regular domain and $A\in \R^{(n+1)\times (n+1)} (\Omega)$ be a uniformly elliptic matrix. For all $\varphi \in C_c^\infty(\R^{n+1})$ and all $x\in \Omega$, 
    \begin{equation}\label{boundary integral to interior integral}
		\int_{\partial \Omega} \varphi (\xi) \, d \omega^x_{\Omega,A} (\xi) -\varphi(x) = -\int_{\Omega} A^T (z)  \nabla g_x^T (z) \nabla \varphi (z) \, dz.
    \end{equation}
    \begin{proof}
    A proof for all $\varphi \in C(\overline\Omega)\cap W^{1,2}(\Omega)$ and for a.e.\ $x\in\Omega$ can be found in \cite[(2.6)]{Azzam2022} and \cite[Lemma 2.19 (2.4)]{Guillen-Prats-Tolsa} in the planar case. Here, we prove that the identity in fact holds for all $x\in\Omega$. We extend $A$ outside $\Omega$ by setting $A=Id$, and throughout the proof we write $\omega = \omega_{\Omega,A}$, $g = g^{\Omega,A}$, and $g^T = g^{\Omega,A^T}$.
    
    We begin by introducing the fundamental solution. We denote by $\EE_x = \EE_x^A$ the fundamental solution with pole at $x$ for $L_A$ in $\R^{n+1}$, $n\geq 1$, so that $-L_A \EE_x^A = \delta_x$ in the distributional sense, where $\delta_x$ is the Dirac mass at the point $x\in \R^{n+1}$. For a construction of the fundamental solution for real and uniformly elliptic matrices we refer to \cite[Section 3]{Hofmann2007} for higher dimensions, $\R^{n+1}$ with $n\geq 2$, and to \cite[Appendix]{Kenig1985} for the planar case.

    We now list some properties of the fundamental solution that will be used below. Letting $\EE_x^T = \EE_x^{A^T}$, we first have $\EE_x (y)=\EE_y^T(x)$ for all $x\not= y$, see \cite[(3.43)]{Hofmann2007} and \cite[Lemma 2.7]{Kenig2009}. Second, it satisfies $0<\EE_x (y) \lesssim |y-x|^{1-n}$ in higher dimensions (see \cite[Theorem 3.1 (3.55)]{Hofmann2007}) and $|\EE_x(y)|\lesssim 1+|{\log |x-y|}|$ in the planar case (see \cite[Theorem A-2]{Kenig1985}). In particular, this implies that
    \begin{equation}\label{eq:fundamental sol is in L1}
    \int_{B(x,r)} |\EE_x (y)| \, dy \lesssim C(r).
    \end{equation}
    Finally, its weak derivative satisfies
    \begin{equation}\label{eq:gradient fundamental sol is in L1}
    \int_{B(x,r)} |\nabla \EE_x (y)|\, dy \lesssim r,
    \end{equation}
    see \cite[Theorem 3.1 (3.52)]{Hofmann2007} for the higher dimensional case and \cite[Theorem 0.1]{Chanillo1992} for the planar case.
    
    We now proceed with the proof of the lemma. The equality \rf{boundary integral to interior integral} for a.e.\ $x\in\Omega$ is sufficient to establish the identity
    \begin{equation}\label{eq:green and fundamental sol}
    g_y (x) = \EE_y (x) - \int_{\partial\Omega} \EE_y (\xi) \, d\omega^x (\xi) \text{ for all }x,y\in\Omega,
    \end{equation}
    see, for instance, \cite[(2.5)]{Guillen-Prats-Tolsa}.

    Now, fix $x\in\Omega$. Using $g_x^T (y)=g_y(x)$ and $\EE_x^T (y)=\EE_y (x)$, together with \rf{eq:green and fundamental sol}, we obtain
    $$
    g_x^T (y) = \EE_x^T(y) - \int_{\partial\Omega} \EE_\xi^T (y) \, d\omega^x (\xi),
    $$
    which implies that the right-hand side of \rf{boundary integral to interior integral} becomes
    $$
    -\int A^T(y) \nabla \EE_x^T(y) \nabla \varphi (y) \, dy
    + \int A^T(y) \nabla_y \left(\int_{\partial\Omega} \EE_\xi^T (y) \, d\omega^x (\xi)\right) \nabla \varphi (y)\, dy.
    $$
    The first term is simply $-\varphi (x)$, since $-L_{A^T} \EE_x^T = \delta_x$. Thus, it remains to prove that
    $$
    \int A^T(y) \nabla_y \left(\int_{\partial\Omega} \EE_\xi^T (y) \, d\omega^x (\xi)\right) \nabla \varphi (y)\, dy = \int_{\partial\Omega} \varphi (\xi) \, d\omega^x (\xi).
    $$
    By \rf{eq:fundamental sol is in L1}, \rf{eq:gradient fundamental sol is in L1}, the boundedness of $\Omega$, and the Fubini-Tonelli theorem, we have that the weak derivative of $\int_{\partial\Omega} \EE_\xi^T (y) \, d\omega^x (\xi)$ is $\int_{\partial\Omega} \nabla \EE_\xi^T (y) \, d\omega^x (\xi)$, and therefore,
    $$
    \int A^T(y) \nabla_y \left(\int_{\partial\Omega} \EE_\xi^T (y) \, d\omega^x (\xi)\right) \nabla \varphi (y)\, dy
    = \int \int_{\partial\Omega} A^T(y) \nabla \EE_\xi^T (y) \nabla \varphi (y) \, d\omega^x (\xi) \, dy.
    $$
    Again by \rf{eq:gradient fundamental sol is in L1}, the boundedness of $\Omega$, and the Fubini-Tonelli theorem, we can interchange the order of integration, whence we get
    $$
    \int A^T(y) \nabla_y \left(\int_{\partial\Omega} \EE_\xi^T (y) \, d\omega^x (\xi)\right) \nabla \varphi (y)\, dy
    = \int_{\partial\Omega} \int A^T(y) \nabla \EE_\xi^T (y) \nabla \varphi (y) \, dy \, d\omega^x (\xi).
    $$
    Finally, using $-L_{A^T} \EE_\xi^T = \delta_\xi$, the inner integral in the right-hand side evaluates to $\varphi(\xi)$, yielding that the latter term equals $\int_{\partial\Omega} \varphi(\xi)\, d\omega^x (\xi)$. This completes the proof of the lemma.
    \end{proof}
    \end{lemma}

    The previous lemma suggests that the Green function determines, in a certain sense, the elliptic measure. This relationship is made precise in the following lemma.

    \begin{lemma}\label{lemma:green function determines elliptic measure}
        Let $\Omega\subset \R^{n+1}$, $n\geq 1$, be a bounded Wiener regular domain, $p\in\Omega$, and $A_0,A_1 \in \R^{(n+1)\times (n+1)} (\Omega)$ be two uniformly elliptic matrices. If $g_p^{\Omega,A_0^T}=g_p^{\Omega,A_1^T}$ in $\Omega$, then
        $$
        \omega_{\Omega,A_0}^p (E) \approx \omega_{\Omega,A_1}^p (E) \text{ for all }E\subset\partial\Omega,
        $$
        where the implicit constants depend only on $n$ and the ellipticity constants of the matrices.
        \begin{proof}
            From the definition, there exists $t_0>0$, sufficiently small depending on the ellipticity constants, such that
            $$
            A_t\coloneqq t A_1 + (1-t) A_0 \text{ is uniformly elliptic for all } t\in [-t_0, 1+t_0].
            $$
            We will use the values $t=-t_0$ and $t=1+t_0$ in the computations below.
            
            We first note that $g_p^T\coloneqq g_p^{\Omega,A_0^T}=g_p^{\Omega,A_1^T}$ is also the Green function for $A_t^T$ for all $t\in [-t_0, 1+t_0]$. Indeed, for any $\varphi\in C_c^\infty(\Omega)$ we have
            $$
            \int_\Omega A_t^T \nabla g_p^T \nabla \varphi
            = t \int_\Omega A_1^T \nabla g_p^T \nabla \varphi
            + (1-t) \int_\Omega A_0^T \nabla g_p^T \nabla \varphi \overset{\text{\rf{dirac delta definition green function}}}{=} \varphi(p).
            $$

            Since $g_p^T$ is the Green function for $A_t^T$ for all $t\in [-t_0, 1+t_0]$, we can use the identity in \rf{boundary integral to interior integral} to relate the elliptic measures $\omega_{A_0}^p$, $\omega_{A_1}^p$ and $\omega_{A_t}^p$ in the sense that, for any $\varphi \in C^\infty_c(\R^{n+1})$, we have
            $$
            \begin{aligned}
            \int_{\partial\Omega} \varphi \, d\omega^p_{A_t}
            &=\varphi(p) - \int_\Omega A_t^T \nabla g_p^T \nabla \varphi\\
            &=t\left(\varphi(p) - \int_\Omega A_1^T \nabla g_p^T \nabla \varphi\right) + (1-t)\left(\varphi(p) - \int_\Omega A_0^T \nabla g_p^T \nabla \varphi\right)\\
            &=t \int_{\partial\Omega} \varphi \, d\omega^p_{A_1} + (1-t)\int_{\partial\Omega} \varphi \, d\omega^p_{A_0}.
            \end{aligned}
            $$
            In particular, if $\varphi\geq 0$, then
            \begin{equation}\label{eq:elliptic measure A0 and A1 for t}
            (t-1)\int_{\partial\Omega} \varphi \, d\omega^p_{A_0} \leq t \int_{\partial\Omega} \varphi \, d\omega^p_{A_1} \text{ for all }t\in [-t_0,1+t_0].
            \end{equation}

            Let $B$ be a ball centered at $\partial\Omega$. For $\varepsilon>0$, let $\varphi \in C_c^\infty(\R^{n+1})$ be such that $\characteristic_{(1+\varepsilon/2) B}\leq \varphi\leq \characteristic_{(1+\varepsilon)B}$. Taking $t=1+t_0$ in \rf{eq:elliptic measure A0 and A1 for t}, we obtain
            $$
            t_0 \omega_{A_0}^p (\overline B)
            \leq t_0 \int_{\partial\Omega} \varphi \, d\omega^p_{A_0} \leq (1+t_0) \int_{\partial\Omega} \varphi \, d\omega^p_{A_1} \leq (1+t_0) \omega^p_{A_1} ((1+\varepsilon)\overline B),
            $$
            and letting $\varepsilon\to 0$ we get $t_0 \omega_{A_0}^p (\overline B) \leq (1+t_0) \omega^p_{A_1} (\overline B)$. By the same argument, taking $t=-t_0$ in \rf{eq:elliptic measure A0 and A1 for t}, all in all we have
            $$
            \omega_{A_0}^p (\overline B) \leq \frac{1+t_0}{t_0} \omega_{A_1}^p (\overline B)
            \text{ and }
            \omega_{A_1}^p (\overline B) \leq \frac{1+t_0}{t_0} \omega_{A_0}^p (\overline B).
            $$
            This proves the lemma by standard arguments, see for instance \cite[Lemma 2.13]{Mattila1995}.
        \end{proof}
    \end{lemma}
    
	\subsection{Harmonic factorization of the Green function}\label{sec:harmonic fact of the green function}

    The latter result of the previous section, \Cref{lemma:green function determines elliptic measure}, suggests that to factor the elliptic measure $\omega_{\Omega,A}^p$, we must first factor the elliptic Green function $g_p^{\Omega,A^T}$. The following lemma establishes the harmonic factorization of the Green function for $L_A$, which turns out to be the harmonic Green function in the image domain.
    
	\begin{lemma}\label{relation green functions}
		Let $\Omega \subset \R^2$ be a bounded Wiener regular domain, $p\in\Omega$, $A$ be a uniformly elliptic matrix with $A=Id$ in an open neighborhood of $p$. If $\mu_{A,g_p^A}$ and $\phi$ are as in \cref{harmonic factorization}, then $g_p^{\Omega,A} \circ \phi^{-1} = g_{\phi (p)}^{\phi(\Omega),Id}$ is the Green function in $\phi(\Omega)$ with pole $\phi(p)$ for the Laplacian operator $L_{Id}=-\Delta$.
	\end{lemma}
	
	\begin{rem}
		Note that the Green function $g_p^A$ has a singularity at $p\in \Omega$. However, the Beltrami coefficient $\mu_{A,g_p^A}$ should be understood as $\mu_{A,g_p^A}=\mu_A$ in the open neighborhood of $p$ where the matrix $A=Id$, as $\nu_A=0$ in this case. In particular, $g_p^A$ is $L_A$-harmonic in $\Omega \setminus \{x: A(x)\not = Id\} \subset \Omega\setminus \{p\}$.
	\end{rem}
	
	\begin{proof}[Proof of \cref{relation green functions}]
		We check $\int_{\phi (\Omega)} \nabla (g_p^A \circ \phi^{-1}) \nabla \varphi =\varphi (\phi (p))$ for any $\varphi \in C^\infty_c (\phi (\Omega))$. Let $B\subset 10B\subset \Omega$ be a ball centered at $p\in \Omega$ with small enough radius such that $A=Id$ in $2B$, and let $\psi\in C^\infty_c (\R^2)$ be so that $\psi|_{B/2}=1$ and $\psi|_{B^c}=0$. Then
		$$
		\int_{\phi (\Omega)} \nabla (g_p^A \circ \phi^{-1}) \nabla \varphi
		=\int_{\phi (\Omega)} \nabla (g_p^A \circ \phi^{-1}) \nabla (\varphi\cdot(\psi\circ\phi^{-1}))
		+\int_{\phi (\Omega)} \nabla (g_p^A \circ \phi^{-1}) \nabla (\varphi \cdot(1-\psi\circ\phi^{-1})) .
		$$
		Since $\varphi (1-\psi\circ\phi^{-1})\in W^{1,2}_c (\phi (\Omega))$ with $\varphi (1-\psi\circ\phi^{-1})=0$ in $\phi(B/2)\ni \phi(p)$, and $g_p^A \circ \phi^{-1}$ is harmonic in $\phi(\Omega\setminus B/2)$ by \cref{harmonic factorization}, we have that the second term on the right-hand side is zero, and therefore
		$$
		\int_{\phi (\Omega)} \nabla (g_p^A \circ \phi^{-1}) \nabla \varphi
		=\int_{\phi (\Omega)} \nabla (g_p^A \circ \phi^{-1}) \nabla (\varphi\cdot(\psi\circ\phi^{-1})) .
		$$
		As the matrix $A=Id$ in $2B$ and $\phi$ is conformal in $2B$, we have that the matrix after the change of variables is $\widetilde A_{\phi^{-1}} = Id$ in $\phi(2B)$, recall its definition in \rf{matrix pushforward}. In particular, since $\supp \nabla (\varphi \cdot (\psi\circ\phi^{-1})) \subset \overline{\phi(B)}$, we can just write
		$$
		\int_{\phi (\Omega)} \nabla (g_p^A \circ \phi^{-1}) \nabla \varphi
		=\int_{\phi (\Omega)} \widetilde A_{\phi^{-1}} \nabla (g_p^A \circ \phi^{-1}) \nabla (\varphi\cdot(\psi\circ\phi^{-1})).
		$$
            By \cite[Theorem 3.8.2]{Astala2009} (applied to every $\Omega^\prime \subset \Omega\setminus\{p\}$) we have
            $$
            \nabla (g_p^A \circ \phi^{-1}) (z) = (D \phi^{-1} (z))^T \nabla g_p^A (\phi^{-1} (z))\text{ for a.e.\ } z\in \phi(\Omega),
            $$
            and applying the change of variables given by $\phi$, see \cite[Theorem 3.8.1]{Astala2009}, we have
		$$
		\int_{\phi (\Omega)} \widetilde A_{\phi^{-1}} \nabla (g_p^A \circ \phi^{-1}) \nabla (\varphi\cdot(\psi\circ\phi^{-1}))
		=\int_{ \Omega}  A \nabla g_p^A \nabla \left( (\varphi\circ \phi)\cdot\psi\right).
		$$
            By the higher integrability of quasiconformal mappings in \cite[(13.24)]{Astala2009}, see also \rf{higher integrability determinant} below, there exists $\varepsilon>0$ such that $(\varphi\circ \phi)\cdot \psi \in W^{1,2+\varepsilon}_c (\Omega)$, see \cite[Theorem 2.7]{Oliva-Prats-2017-JMAA} for instance\footnote{Cover $\supp\varphi$ by a finite number of balls $B_i$ with $2B_i\subset \phi(\Omega)$.}, and as noted right below \rf{dirac delta definition green function}, the integral above equals $\varphi(\phi(p))\psi(p)=\varphi(\phi(p))$, as claimed.
		
		As we have that both $g_p^A \circ \phi^{-1}$ and $g_{\phi(p)}^{Id}$ are the dirac delta in the weak sense, we obtain
		$$
            \int_{\phi (\Omega)} \nabla (g_{\phi(p)}^{Id} - g_p^A \circ \phi^{-1}) \nabla \varphi =0 \text{ for any } \varphi \in C^\infty_c (\phi (\Omega)) .
		$$
            Since $g_p^A \in W^{1,s} (\Omega)$ and $g_{\phi(p)}^{Id}\in W^{1,s} (\phi(\Omega))$ for all $1\leq s < 2$ by \rf{sobolev regularity of green function}, in particular $g_p^A \circ \phi^{-1}, g_{\phi(p)}^{Id}\in W^{1,1} (\phi (\Omega))$, see \cite[Theorem 2.7]{Oliva-Prats-2017-JMAA} for instance. The existence of weak derivatives in $L^1 (\phi(\Omega))$ implies that the left-hand side is $-\int_{\phi(\Omega)} \varphi \Delta (g_{\phi(p)}^{Id} - g_p^A \circ \phi^{-1})$, and therefore
            $$
            \int_{\phi (\Omega)} \varphi \Delta (g_{\phi(p)}^{Id} - g_p^A \circ \phi^{-1}) =0 \text{ for any } \varphi \in C^\infty_c (\phi (\Omega)).
            $$
		By Weyl's lemma, $g_{\phi(p)}^{Id} - g_p^A \circ \phi^{-1} \in C^\infty (\phi(\Omega))$ and is harmonic in $\phi(\Omega)$. Since both functions $g_{\phi(p)}^{Id}$ and $g_p^A \circ \phi^{-1}$ vanish on $\partial (\phi (\Omega))$, in particular $g_{\phi(p)}^{Id} - g_p^A \circ \phi^{-1}$ vanishes on $\partial (\phi(\Omega))$. By the maximum principle for harmonic functions, the function $g_{\phi(p)}^{Id} - g_p^A \circ \phi^{-1}$ must be zero in $\phi(\Omega)$, and therefore $g_{\phi(p)}^{Id} = g_p^A \circ \phi^{-1}$ in $\phi(\Omega)$, as claimed.
	\end{proof}
	
	Given a function $\phi$ as in \cref{relation green functions} for the matrix $A^T$, using the notation in \rf{matrix pushforward} and \cref{linear change of variables qc}, recall that $\widetilde{Id}_{\phi}$ denotes the matrix such that
	$$
	u \text{ is harmonic in }\phi(\Omega) \iff u\circ \phi \text{ is } L_{\widetilde{Id}_\phi}\text{-harmonic in } \Omega .
	$$
    With this definition, in the previous lemma we saw that $g_p^{A^T}$ is the Green function in $\Omega$ for both operators $L_{A^T}$ and $L_{\widetilde{Id}_\phi}$. Consequently, by \cref{lemma:green function determines elliptic measure} we are now ready to prove \cref{lemma:elliptic measure factorization in general wiener regular domains}.

    \begin{proof}[Proof of \cref{lemma:elliptic measure factorization in general wiener regular domains}]
        We extend $A=Id$ in $\Omega^c$. Furthermore, by \cref{coro:elliptic measure only depends on how is the matrix around the boundary}, we can assume that $A=Id$ in an open neighborhood of $p$.
        
        Let $\phi_{A^T,g_p^{A^T}}\in W_{\loc}^{1,2}(\C)$ be the normalized solution of the Beltrami equation
        \begin{equation}\label{eq:beltrami eq for the lemma2 about elliptic harmonic measure}
        \bar\partial\phi_{A^T,g_p^{A^T}} = \mu_{A^T,g_p^{A^T}}\partial \phi_{A^T,g_p^{A^T}} \text{ a.e.\ in }\C, \text{ where }\mu_{A^T,g_p^{A^T}} \text{ as in }\rf{coefficients classical beltrami equation in terms of u and A},
        \end{equation}
        given by the measurable Riemann mapping theorem, see \cref{riemann mapping thm}. In particular $\phi_{A^T,g_p^{A^T}}$ is a $K_\lambda$-quasiconformal homeomorphism with $K_\lambda$ as in \rf{K-lambda ellipticity condition beltrami coefficient} by \cref{ellipticity condition beltrami coefficient}, and conformal in $\overline{\Omega}^c$. We remark that the dependence of the function $\phi_{A^T,g_p^{A^T}}$ on $g_p^{A^T}$ and $A^T$ is due to the fact that $\mu_{A^T,g_p^{A^T}}$ depends on (the derivatives of) $g_p^{A^T}$ and the coefficients of $A^T$.

        By \cref{relation green functions}, we have that $g_p^{A^T}$ is the Green function in $\Omega$ for both operators $L_{A^T}$ and $L_{\widetilde{Id}_\phi}$. The lemma then follows from this and \cref{lemma:green function determines elliptic measure}.
    \end{proof}

	\subsection{Continuous matrices. Proof of Theorem~\ref{main thm general domains and continuous matrices}}\label{proof of main thm general domains and contiuous matrices}
	
	In this section, we prove \cref{main thm general domains and continuous matrices} via a localization argument. More precisely, we use the continuity of the matrix to see that locally, it is as close as we want to a constant matrix.
	
	\begin{proof}[Proof of \cref{main thm general domains and continuous matrices}]
		First, we prove the upper bound for general domains, and then the lower bound for finitely connected domains.
		
		\emph{``Upper bound''}. It suffices to see that for every $\tau >0$ there exists $F_\tau\subset \partial \Omega$ with $\dim_\HH F_\tau \leq 1+\tau$ and $\omega_{\Omega,A} (F_\tau)=1$. Indeed, $F\coloneqq\bigcap_{n\geq 1} F_{1/n}$ satisfies $\dim_\HH F \leq 1$ and $\omega_{\Omega,A} (F)=1$.
		
		Since $A$ has continuous coefficients in $\overline{\Omega}$, for each $\xi \in \partial\Omega$ and every $\varepsilon > 0$ there is $\delta_{\xi, \varepsilon} >0$ such that
		\begin{equation}\label{continuity of the matrix C0}
			z\in \overline{4B_{\delta_{\xi,\varepsilon}} (\xi) \cap \Omega} \implies |a_{ij} (z) - a_{ij} (\xi)| < \varepsilon \text{ for any } i,j\in \{1,2\}.
		\end{equation}
		
		Fix $\varepsilon >0$. Let $\{B_i\}_i \subset \{B_s (\xi)\}_{\xi \in \partial \Omega, 0<s\leq \delta_{\xi, \varepsilon}}$ be a countable subfamily of disjoint balls such that $\omega^p (\partial \Omega \setminus \bigcup_i B_i) = 0$, by Vitali's covering theorem.
		
		We claim that it suffices to see that for each index $i$ there exists a set $F_i \subset \partial (\Omega \cap 4B_i)$ with $\omega_{\Omega \cap 4B_i,A} (F_i) = 1$ and $\dim_\HH F_i \leq 1+\tau$ provided $\varepsilon>0$ is small enough. Indeed, this implies that the set $F_\tau\coloneqq\partial \Omega \cap \bigcup_i F_i$ satisfies $\dim_\HH F_\tau \leq 1+\tau$ by definition, and $\omega_{\Omega,A} (F_\tau)=1$ by \cref{localization argument}.
		
		Applying \cref{perturbation of constant matrix} for each index $i$, by \rf{continuity of the matrix C0} we have that there exists a set $F_i\subset \partial (\Omega \cap 4B_i)$ with $\omega_{\Omega \cap 4B_i,A} (F_i) = 1$ and $\dim_\HH F_i \leq \overline{\dim}_\HH ((1-2\varepsilon\lambda)^{-1})$. By \cref{main thm general domains} and taking $\varepsilon>0$ small enough respectively we have
		$$
		\overline\dim_\HH ((1-2\varepsilon\lambda)^{-1}) \leq \frac{2K_{(1-2\varepsilon\lambda)^{-1}}}{K_{(1-2\varepsilon\lambda)^{-1}}+1}\leq 1+\tau,
		$$
		and the upper bound is proved.
		
		\emph{``Lower bound''}. We turn now to the proof of the lower bound for finitely connected domains. Let $F\subset \partial\Omega$ with $\omega_{\Omega,A}(F)=1$ and fix $\tau >0$. Given $\varepsilon >0$, for each $\xi \in \partial \Omega$ let $\delta_{\xi, \varepsilon}>0$ be small enough to satisfy \rf{continuity of the matrix C0} and that $\Omega\cap B(\xi,\delta_{\xi,\varepsilon})$ is simply connected, see \cref{claim:candidate for each index}. Let $\{B_i\}_i \subset\{B_s (\xi)\}_{\xi \in \partial \Omega, 0<s\leq \delta_{\xi, \varepsilon}}$ be a countable subfamily with finite overlapping by Besicovitch covering theorem. For each ball $B_i$, by \cref{claim:candidate for each index} we have $\omega_{\Omega \cap B_i} (\partial \Omega \setminus F) =0$, which implies
		$$
		\omega_{\Omega \cap B_i} ((\partial B_i \cap \Omega)\cup(F\cap B_i)) = 1.
		$$
		From this, $\dim_\HH \partial B=1$, \rf{continuity of the matrix C0} and \cref{perturbation of constant matrix} we have
		$$
		\dim_\HH F \geq \dim_\HH (F\cap B_i) \geq \underline\dim_\HH^{\textrm{sc}} ((1-2\varepsilon\lambda)^{-1}).
		$$
		By \cref{main thm general domains} and taking $\varepsilon>0$ small enough respectively we have
		$$
		\underline\dim_\HH^{\rm{sc}} ((1-2\varepsilon\lambda)^{-1}) \geq \frac{2}{K_{(1-2\varepsilon\lambda)^{-1}}+1}\geq 1-\tau.
		$$
		As this can be done for every $\tau>0$ we have that $\dim_\HH F\geq 1$ and the lower bound of $\dim_\HH \omega_{\Omega,A}$ for finitely connected domains is proved.
	\end{proof}

	\section{Symmetric matrices with determinant 1}\label{part:sym + det 1}
	
	Let $\Omega \subset \R^2$ be a domain and $A \in \R^{2\times 2} (\overline \Omega)$ be a symmetric and uniformly elliptic (with ellipticity constant $\lambda \geq 1$) matrix with $\det A = 1$. Note that during all \cref{part:sym + det 1} we will use $K_\lambda = \lambda$ since the matrix is symmetric, see \rf{K-lambda ellipticity condition beltrami coefficient}.
	
	In the following sections, we will assume that the matrix $A$, and therefore also the Beltrami coefficient $\mu_A$, enjoys more regularity. By regularity theory, we will have that the quasiconformal mapping $\phi_A$ given in \cref{harmonic factorization symmetric matrices with determinant 1} has better estimates than the Hölder continuity in \cref{qc are locally holder} or the distortion of Hausdorff dimension (resp. measure) in \rf{dimension distortion theorem} (resp. \cref{distortion of hausdorff measure thm} or \cref{distortion zero measure of makarov gauge function}). The proof of the results in the following sections will have the same structure as the proof of \cref{main thm general domains} in \cref{sec:main thm general domains}, replacing the use of \cref{lemma:elliptic measure factorization in general wiener regular domains} with \cref{lemma:relation elliptic measure for sym and det 1}. The difference is that we will now use the fact that the change of variables $\phi_A$ is not only $\lambda$-quasiconformal, but it is more regular.
	
	\subsection{Mean oscillation under quasiconformal mappings}
	
	In this subsection, we define the space BMO functions and collect some properties to control the regularity of the coefficients arising in the following sections.
	
	\begin{definition}\label{p-mean oscillation}
		Let $1\leq p <\infty$, $f\in L^p_{\loc} (\Omega)$ and a set $S \subset \Omega$. Denote the $p$-mean oscillation of $f$ in $S$ the quantity
		$$
		[f]_{S, p} \coloneqq \left(\avint_S \left| f(z) - m_S f \right|^p \, dz\right)^{1/p},
		$$
		where $m_S f \coloneqq \avint_S f(y)\, dy \coloneqq \frac{1}{|S|} \int_S f(y) \, dy$.
	\end{definition}
	
	Bounded functions satisfy $[f]_{S, p} \leq 2\|f\|_{L^\infty (S)}$. In this lemma, we collect some fundamental bounds of the $p$-mean oscillation. We omit its proof.
	\begin{lemma}\label{properties of p-mean oscillation}
		Let $1\leq p < \infty$. The following holds:
		\begin{enumerate}
			\item \label{sum in vmo} $[f+g]_{S,p}^p \lesssim [f]_{S,p}^p + [g]_{S,p}^p$.
			
			\item \label{product in vmo} $[fg]_{S,p}^p \lesssim \|g\|_{L^\infty (S)}^p [f]_{S,p}^p + \|f\|_{L^\infty (S)}^p [g]_{S,p}^p$.
			
			\item \label{reciproc in vmo} $[1/f]_{S,p}^p \lesssim \left\| 1/f \right\|_{L^\infty (S)}^{2p} [f]_{S,p}^p$.
			
			\item \label{fraction in vmo} $[f/g]_{S,p}^p 
			\lesssim \|1/g\|_{L^\infty (S)}^p [f]_{S,p}^p + \|f\|_{L^\infty (S)}^p  \left\| 1/g \right\|_{L^\infty (S)}^{2p} [g]_{S,p}^p$.
		\end{enumerate}
	\end{lemma}
	
	Let us now recall the definition of bounded mean oscillation functions.
	
	\begin{definition}
		A function $f\in L_{\loc}^1 (\Omega)$ is of bounded mean oscillation in $\Omega$, $f\in \bmo (\Omega)$, if
		$$
		\|f\|_{\Omega,*} \coloneqq \sup_{B \subset \Omega\text{ ball}} \avint_B |f(z)-m_B f | \, dz  <\infty.
		$$
		If $\Omega = \R^{n+1}$ we write $\|f\|_{*}$.
	\end{definition}
	
	With the notation in \cref{p-mean oscillation}, $\|f\|_{\Omega,\ast} = \sup_{B\subset \Omega \text{ ball}} [f]_{B,1}$. Recall that $f\in \bmo$ implies $|f|\in\bmo$. Note also that it is equivalent considering cubes $Q$ instead of balls $B$ in the $\bmo (\R^{n+1})$ norm.
	
	Given a function $f\in \bmo (\R^{n+1})$, the John-Nirenberg inequality, see \cite[Theorem 6.11]{Duoandikoetxea2000} for example, states that there exist dimensional constants $C_1$ and $C_2$ such that for every cube $Q \subset \R^{n+1}$ and every $\lambda > 0$,
	$$
	\left| \{ z\in Q : |f(z) - m_Q f| > \lambda \} \right| \leq C_1 e^{-C_2 \lambda / \|f\|_*} |Q| .
	$$
	Consequently, see \cite[Corollary 6.12]{Duoandikoetxea2000} for a proof, for all $1<p<\infty$,
	$$
	\|f\|_* \approx_p \|f\|_{*,p} \coloneqq \sup_{B \subset \Omega\text{ ball}} \left( \avint_B |f(z)-m_B f |^p \, dz \right)^{1/p} .
	$$
	
	A close look at the proof of the John-Nirenberg inequality in \cite[Theorem 6.11]{Duoandikoetxea2000} reveals that, in fact, for every cube $Q\subset \R^{n+1}$ and every $\lambda > 0$, if
	$$
	\|f\|_{Q,*}^{\mathrm{cubes}} \coloneqq \sup_{R\subset Q \text{ cube}} \avint_R |f(z)-m_R f | \, dz 
	$$
	then
	$$
	\left| \{ z\in Q : |f(z) - m_Q f| > \lambda \} \right| \leq C_1 e^{-C_2 \lambda / \|f\|_{Q,*}^{\mathrm{cubes}}} |Q|.
	$$
	Following the steps in the proof of \cite[Corollary 6.12]{Duoandikoetxea2000}, given a cube $Q\subset \R^{n+1}$, we have
	\begin{equation}\label{local equivalence bmo norm}
		\sup_{R\subset Q \text{ cube}}  \avint_R |f(z)-m_R f | \, dz \approx_p \sup_{R\subset Q \text{ cube}} \left( \avint_R |f(z)-m_R f |^p \, dz \right)^{1/p} .
	\end{equation}
	
	Functions of bounded mean oscillation are invariant under quasiconformal mappings. This is a result of Reimann in \cite{Reimann1974}, which, in fact, also proved that if a $W^{1,1}_{\loc}$-homeomorphism $f:\C \to \C$ preserves the $\bmo$ space, then $f$ is $K$-quasiconformal, with $K=K(\|f\|_*)$. In the following lemma, we control the mean oscillation under global quasiconformal mappings. As the statement of the lemma and its proof are slightly different from the one in \cite{Reimann1974} and \cite[Theorem 13.4.1]{Astala2009}, for the sake of completeness we include the details.
	
	\begin{lemma}\label{bmo under quasiconformal homeo}
		Let $K\geq 1$, $\phi \in W_{\loc}^{1,2} (\C)$ be a $K$-quasiconformal homeomorphism of $\C$, and $\widetilde B \subset \R^2$ be an open ball. Given an open ball $B \subset \widetilde B$ (with center $c_B$ and radius $r_B$), consider
		\begin{equation}\label{ball in the image vmo}
			B_- \coloneqq B \left( \phi^{-1} (c_B) , \max_{|\xi-c_B|=r_B} |\phi^{-1} (\xi)-\phi^{-1} (c_B)| \right) ,
		\end{equation}
		Then $r_{B_-} \leq C_{K,\widetilde B,\phi} \cdot r_B^{1/K}$, and for any $p\in (K,\infty)$,
		\begin{equation}\label{distortion of mean oscillation under qc}
			\avint_B |f (\phi^{-1} (x)) - m_B (f \circ \phi^{-1})| \, dx
			\lesssim_{p,K} \left( \avint_{ B_-} |f (x)- m_{ B_-} f |^p \, dx \right)^{1/p} .
		\end{equation}
		In particular $\|f\circ \phi^{-1}\|_* \lesssim_K \|f\|_*$.
		\begin{proof}
			The notation ``$-$'' in $B_-$ is used to specify that the ball $B_-$ ``lives'' in the image of $\phi^{-1}$.
			
			We control the radius of $B_-$. From the definition of $B_-$ and the Hölder continuity with exponent $1/K$ of quasiconformal mappings in \cref{qc are locally holder} applied to $\phi^{-1}$, we conclude
			\begin{equation}\label{vmo new radi}
				r_{B_-} = \max_{|\xi-c_B|=r_B} |\phi^{-1} (\xi)-\phi^{-1} (c_B)| 
				\leq C_{K,\widetilde B,\phi} \max_{|\xi-c_B|=r_B} |\xi- c_B |^{1/K} =  C_{K,\widetilde B,\phi} \cdot r_B^{1/K} .
			\end{equation}
			
			It remains to see \rf{distortion of mean oscillation under qc}. We can change from $m_B (f \circ \phi^{-1})$ to $ m_{ B_-} f$ because
			\begin{equation}\label{vmo to other constant}
				\begin{aligned}
					\avint_B |f (\phi^{-1} (x)) - m_B (f \circ \phi^{-1})| \, dx
					&\leq \avint_B \avint_B |f (\phi^{-1} (x)) - f (\phi^{-1} (y))|  \, dy \, dx \\
					&\leq 2 \avint_B |f (\phi^{-1} (x)) - m_{ B_-} f | \, dx .
				\end{aligned}
			\end{equation}
			Let $q$ be the Hölder conjugate of $p$, i.e., $1/p+1/q=1$. In particular $p\in (K,\infty)$ if and only if $q\in \left(1,K/(K-1)\right)$. By the change of variables $x=\phi (y)$, $\phi^{-1} (B) \subset  B_-$ and Hölder's inequality, we bound the last term in (\ref{vmo to other constant}) by
			\begin{equation}\label{vmo and holder ineq}
                \begin{aligned}
				\avint_B |f (\phi^{-1} (x)) - m_{ B_-} f | \, dx 
				&= \frac{1}{|B|} \int_{\phi^{-1} (B)} |f (y)- m_{ B_-} f | |\det D\phi (y)| \, dy \\
				&\leq \frac{1}{|B|} 
				\left( \int_{ B_-} |f (x)- m_{ B_-} f |^p \, dx \right)^{1/p}
				\left( \int_{ B_-} \left| \det D\phi (x) \right|^q \, dx \right)^{1/q} .
                \end{aligned}
			\end{equation}
			Since $q\in (1,K/(K-1)) \subset [0,K/(K-1))$, the higher integrability of $\left| \det D \phi \right|$, see \cite[(13.24)]{Astala2009}, implies that the last term in (\ref{vmo and holder ineq}) is controled by
			\begin{equation}\label{higher integrability determinant}
				\left( \int_{ B_-} \left| \det D\phi (x) \right|^q \, dx \right)^{1/q} \lesssim_{p,K} | B_-|^{1/q} \frac{|\phi ( B_-)|}{| B_-|} .
			\end{equation}
			The inequality in \rf{higher integrability determinant} implies that (\ref{vmo and holder ineq}) reduces to
			\begin{equation}\label{vmo preimage}
				\begin{aligned}
					\avint_B |f (\phi^{-1} (x)) - m_{ B_-} f | \, dx 
					&\lesssim_{p,K} \frac{| B_-|^{1/p}}{|B|} 
					\left( \avint_{ B_-} |f (x)- m_{ B_-} f |^p \, dx \right)^{1/p}
					| B_-|^{1/q} \frac{|\phi ( B_-)|}{| B_-|} \\
					&= \frac{|\phi ( B_-)|}{|B|} 
					\left( \avint_{ B_-} |f (x)- m_{ B_-} f |^p \, dx \right)^{1/p},
				\end{aligned}
			\end{equation}
			where we used $1/p + 1/q = 1$ in the last inequality.
			
			Now we want to bound $|\phi ( B_-)|/|B|$. We have the inclusion
			$$
			\phi ( B_-) \subset  B_+ \coloneqq B \left( c_B , \max_{|\xi-\phi^{-1}(c_B)|=r_{B_-}} |\phi (\xi)-c_B|\right).
			$$
			Here the notation $B_+$ is used to specify that the ball $B_+$ ``lives'' in the image of $\phi$. By \cite[Lemma 3.4.5]{Astala2009}, there exists a constant $\delta = \delta (K)$ such that $\delta B_+ \subset \phi (B_-) \subset B_+$. By definition $\delta B_+ \subset B$, and hence
			\begin{equation}\label{comparable balls}
				\frac{|\phi ( B_-)|}{|B|}  \leq \frac{|B_+|}{|\delta B_+|} \approx_K 1 .
			\end{equation}
			Inequality in \rf{distortion of mean oscillation under qc} follows from \rf{vmo to other constant}, and \rf{comparable balls} applied to \rf{vmo preimage}. Therefore, \cref{bmo under quasiconformal homeo} is proved.
        \end{proof}
	\end{lemma}
	
	\subsection{Proof of Theorems~\ref{sym + det 1 Holder} and \ref{sym + det 1 all cases} via quasiconformal regularity}\label{sec:quasiconformal regularity}
	
	In this section, via quasiconformal theory, we will see that assuming regularity of the matrix, the global $\lambda$-quasiconformal homeomorphism $\phi_A$ in \cref{harmonic factorization symmetric matrices with determinant 1} gains more regularity.
	
	Let us present the sketch of this. If $\phi_A$ is a solution of $\bar \partial \phi_A = \mu_A \partial\phi_A$, then $\phi_A$ and $\phi_A^{-1}$ are $\lambda$-quasiconformal, in particular $\phi_A, \phi_A^{-1} \in C^{1/\lambda}_{\loc}$, and moreover $\phi_A^{-1}$ satisfies the Beltrami equation $\bar\partial (\phi_A^{-1}) = -(\mu_A \circ \phi_A^{-1}) \overline{\partial (\phi_A^{-1})}$. From $\phi_A^{-1} \in C^{1/\lambda}_{\loc}$, assuming regularity on the matrix $A$ will not only provide regularity on $\mu_A$, but also on $-\mu_A \circ \phi_A^{-1}$. As a consequence, we will obtain extra regularity on the functions $\phi_A$ and $\phi_A^{-1}$. For example, if $A\in C^\alpha$, then $\mu_A \in C^\alpha$ and $-\mu_A \circ \phi_A^{-1}\in C^{\alpha/\lambda}_{\loc}$, whence we obtain that $\phi_A\in C^{1,\alpha}_{\loc}$ and $\phi_A^{-1} \in C^{1,\alpha/\lambda}_{\loc}$, see \cite[Theorem 15.6.2]{Astala2009}.
	
	\subsubsection{Hölder continuous coefficients}\label{sec:holder continous}

        In this subsection we prove \cref{sym + det 1 Holder}. The sketch is what is already explained above. For completeness, we provide the proof in full detail.
 
        Although we saw in \cref{coro:elliptic measure only depends on how is the matrix around the boundary} that the dimension of elliptic measures only depends on the nature of the matrix around the boundary, during the following proof we emphasize that the regularity of the matrix is only needed around the boundary. Recall $U_t (\partial \Omega)$ denotes the $t$-neighborhood of the boundary, i.e.,
	$$
	U_t (\partial \Omega) \coloneqq \{x\in \R^2 : \dist (x,\partial \Omega) < t\} .
	$$
	
	\begin{proof}[Proof of \cref{sym + det 1 Holder}]\label{proof of sym + det 1 holder}
		Recall we can assume that $\Omega\subset B_{1/4}(0)$ is bounded Wiener regular by \cref{reductions all}, and we can replace finitely connected by simply connected by \cref{claim:candidate for each index}.
		
		Assume that $A\in C^\alpha (U_\varepsilon(\partial\Omega))$ for $\varepsilon >0$. Let $\phi_A \in W^{1,2}_{\loc}(\C)$ be the normalized solution of the Beltrami equation
		$$
		\bar\partial\phi_A = \mu_A \partial\phi_A \text{ a.e.\ in }\C, \text{ where }\mu_A \text{ as in }\rf{beltrami coefficients for sym and det 1 matrices},
		$$
		given by the measurable Riemann mapping theorem, see \cref{riemann mapping thm}. In particular $\phi_A$ is a $\lambda$-quasiconformal homeomorphism. Since A is uniformly elliptic, in particular $a_{11} , a_{22} \approx 1$ and hence $2+a_{11} + a_{22} \geq 2$, and therefore $A\in C^\alpha (U_\varepsilon (\partial \Omega))$ implies $\mu_A \in C^\alpha (U_\varepsilon (\partial \Omega))$.
		
		By the Hölder regularity of quasiconformal mappings in \cref{qc are locally holder} we have that $\phi_A^{-1} \in C^{1/\lambda}_{\loc}$, and hence $-\mu_A \circ\phi_A^{-1}\in C^{\alpha/\lambda}_{\loc}$ in an open neighborhood of $\partial (\phi_A (\Omega))$. Therefore, we get that $\phi_A \in C^{1,\alpha}_{\loc}$ (resp. $\phi_A^{-1} \in C^{1,\alpha/\lambda}_{\loc}$) in an open neighborhood of $\partial \Omega$ (resp. $\partial (\phi_A (\Omega))$), see \cite[Theorem 15.6.2]{Astala2009}. In particular both $\phi_A$ and $\phi_A^{-1}$ are Lipschitz in an open neighborhood of $\partial \Omega$ and $\partial (\phi_A (\Omega))$ respectively.
		
		The rest of the proof follows by the same proof of \cref{main thm general domains} on \cpageref{proof of main thm general domains} using \cref{lemma:relation elliptic measure for sym and det 1} and the fact that $\phi_A$ is bi-Lipschitz in an open neighborhood of $\partial \Omega$.
	\end{proof}
	
	\subsubsection{Only one variable dependence coefficients}\label{sec:only one variable dependence}
	
	In this subsection, we consider matrices with no regularity, but satisfying $A(z) = A(\real(z))$ (or $A(z) = A(\imag(z))$). Even with no regularity on the coefficients (only uniformly elliptic), in this case, the quasiconformal change of variables is bi-Lipschitz.
	
	\begin{lemma}[See {\cite[Proposition 5.23]{Bojarski2013}}]\label{regularity when mu only depends on x}
		Let $\mu : \C\to \C$ be an arbitrary measurable function with $\|\mu\|_\infty \leq k < 1$ that depends on $x=\real(z)$ only and let
		$$
		f(x) = \int_0^x \frac{1+\mu(t)}{1-\mu(t)} \, dt .
		$$
		Then $\phi(z) = f(x) + iy$ represents a unique quasiconformal mapping on $\overline \C$ onto itself with complex dilatation $\mu$ and normalizations $\phi(0)=0$, $\phi(i)=i$ and $\phi(\infty)=\infty$. Moreover, $\phi$ is bi-Lipschitz.
	\end{lemma}

        \begin{proof}[Proof of \cref{sym + det 1 all cases}: $A(z)=A(\real(z))$ case]\label{proof of main thm only one variable dependence}
            Recall we can assume that $\Omega$ is Wiener regular by \cref{reductions all}, and we can replace finitely connected by simply connected by \cref{claim:candidate for each index}.
			
			We assume without loss of generality that $A(z) = A(\real(z))$. In particular, $\mu_A$ as in \rf{beltrami coefficients for sym and det 1 matrices} satisfies $\mu_A(z) = \mu_A(\real(z))$ and hence $\phi_A$, given by \cref{regularity when mu only depends on x}, is bi-Lipschitz. The proof follows as in the proof of \cref{main thm general domains} on \cpageref{proof of main thm general domains} using \cref{lemma:relation elliptic measure for sym and det 1} and the fact that $\phi_A$ is bi-Lipschitz.
        \end{proof}
	
	\subsubsection{VMO coefficients}
	
	Let us introduce the space of functions with vanishing mean oscillation.
	
	\begin{definition}\label{def:vmo}
		A functions $f$ has vanishing mean oscillation, written $f\in \vmo$, if $f\in \bmo$ and
		$$
		\lim_{r\to 0} \sup_{x\in \R^2} \avint_{B_r (x)} |f(z)-m_{B_r(x)} f | \, dz =0.
		$$
	\end{definition}

        For the classical Beltrami equation with $\vmo_c$ coefficient, its quasiconformal solution preserves the dimension of sets.

	\begin{lemma}\label{homeo with vmo coeff preserve dimension}
		Let $\mu \in \vmo_c$ be a Beltrami coefficient and $\phi \in W^{1,2}_{\loc} (\C)$ be a $K$-quasiconformal homeomorphism of $\C$ satisfying $\bar \partial \phi = \mu \partial \phi$. Then both $\phi$ and $\phi^{-1}$ are $C^\alpha_{\loc}$ for every $\alpha \in (0,1)$. In particular
		$$
		\dim_\HH (\phi (S)) = \dim_\HH (S) \text{ for any set } S.
		$$
	\end{lemma}

        This is well-known in the literature. For the sake of completeness, a proof of \cref{homeo with vmo coeff preserve dimension} can be found below, as the author did not find any proof in the literature.
	
        Using the result above, for $\vmo$ matrices we recover the Jones and Wolff theorem in \cite{Jones1988} and the lower bound of the dimension of elliptic measure in simply connected domains.
	
		\begin{proof}[Proof of \cref{sym + det 1 all cases}: $\vmo$ case]\label{proof of vmo matrix everywhere}
			Recall we can assume that $\Omega\subset B_{1/4}(0)$ is bounded Wiener regular by \cref{reductions all}, and we can replace finitely connected by simply connected by \cref{claim:candidate for each index}.
			
			Let us fix a function $\varphi \in C^\infty_c (\R^2)$ satisfying
			\begin{equation}\label{definition of the support function}
				0\leq \varphi \leq 1,\,  \varphi|_{B_{1/2}(0)}=1,\, \varphi|_{B_{1}(0)^c}=0 \text{ and } |\nabla \varphi|\lesssim 1,
			\end{equation}
			and let $\phi_A \in W^{1,2}_{\loc}(\C)$ be the normalized solution of the Beltrami equation
			$$
			\bar\partial\phi_A = \varphi \mu_A \partial\phi_A \text{ a.e.\ in }\C, \text{ where }\mu_A \text{ as in }\rf{beltrami coefficients for sym and det 1 matrices},
			$$
			given by the measurable Riemann mapping theorem, see \cref{riemann mapping thm}. Since A is uniformly elliptic, in particular $a_{11} , a_{22} \approx 1$ and hence $2+a_{11} + a_{22} \geq 2$. Since $A\in \vmo (\R^2)$, by \cref{fraction in vmo} in \cref{properties of p-mean oscillation} we have $\mu_A \in \vmo (\R^2)$ and so $\varphi\mu_A\in \vmo_c (\R^2)$. Therefore, by \cref{homeo with vmo coeff preserve dimension}, the function $\phi_A$ preserves the dimension of sets. The theorem follows from this by the same proof of \cref{main thm general domains} on \cpageref{proof of main thm general domains}, using \cref{lemma:relation elliptic measure for sym and det 1}.
		\end{proof}

	We now turn to the proof of \cref{homeo with vmo coeff preserve dimension}. First, in the following result we see the regularity of a quasiconformal mapping with $\vmo_c$ coefficient.
	
	\begin{theorem}[See {\cite[p.~42-43]{Iwaniec1992}}]\label{regularity for vmo matrices}
		Let $\mu \in \vmo_c$ be such that $\|\mu\|_\infty = k < 1$, and $p \in (0,\infty)$. Then for each $h\in L^p$ the equation
		$$
		\bar \partial f - \mu \partial f = h
		$$
		admits a solution $f$ with $\nabla f \in L^p$ which is unique up to an additive constant.
	\end{theorem}
        This result also holds for the generalized Beltrami equation, see \cite[Theorem 4.6]{Koski2011} and \cite[Theorem 8]{Clop-Cruz-2013}.
        
	\begin{rem}\label{beltrami in vmo particular case}
	    Note that if $\bar\partial\phi = \mu \partial\phi$ with $\mu \in \vmo_c$, then $\phi \in L^p_{\loc}$ for all $p\in (1,\infty)$. Thus, $\phi \in C_{\loc}^\alpha$ for every $\alpha \in (0,1)$, and as a consequence, $\dim_\HH (\phi (S))\leq \dim_\HH (S)$ for any set $S$.
	\end{rem}
        
        Next we want to see the regularity of the Beltrami coefficient of $\phi^{-1}$. Before that, as a consequence of the estimate \rf{distortion of mean oscillation under qc} in \cref{bmo under quasiconformal homeo}, we see that $\vmo_c\cap L^\infty$ is invariant under quasiconformal mappings.

        \begin{coro}\label{vmo compact and bounded are invariant under qc}
            $\vmo_c\cap L^\infty$ is invariant under global homeomorphic quasiconformal mappings.
            \begin{proof}
                Let $f\in \vmo_c\cap L^\infty$ be non-identically zero, otherwise we are done. By \cref{def:vmo}, for all $\varepsilon>0$ there is $r_0 = r_0 (\varepsilon)$ such that if $r<r_0$ then
                $$
                \sup_{x\in \R^2} \avint_{B_r(x)} |f-m_{B_r(x)}f|<\varepsilon.
                $$
                Let $\phi \in W^{1,2}_{\loc}(\C)$ be a $K$-quasiconformal homomorphism of $\C$, with $K\geq 1$, $\widetilde B$ be a ball such that $\supp (f\circ\phi^{-1}) \subset \widetilde B/2$ and fix $\widetilde C \coloneqq C_{K,2\widetilde B,\phi}$, the constant in \cref{bmo under quasiconformal homeo}. We want to see that for all $\varepsilon^\prime>0$ there is $r_0^\prime = r_0^\prime (\varepsilon^\prime) \leq r(\widetilde B)/100$ such that if $r<r_0^\prime$ then
                $$
                \sup_{x\in \R^2} \avint_{B_r(x)} |f\circ\phi^{-1}-m_{B_r(x)}(f\circ\phi^{-1})|<\varepsilon^\prime.
                $$
                Given $x\in \widetilde B$ and $r<r_0^\prime$ (otherwise the integral above is zero), let $B=B_r(x)$ and $B_-$ as in \rf{ball in the image vmo}. By \rf{distortion of mean oscillation under qc} (with $p=K+1$) we have
                $$
                \begin{aligned}
                \avint_B |f\circ\phi^{-1}-m_B(f\circ\phi^{-1})|
                &\leq C_K \left(\avint_{B_-} |f-m_{B_-}f|^{K+1}\right)^{1/(K+1)}\\
                &\leq C_K (2\|f\|_{L^\infty})^{K/(K+1)} \left(\avint_{B_-} |f-m_{B_-}f|\right)^{1/(K+1)},
                \end{aligned}
                $$
                where $r(B_-)\leq \widetilde C r(B)^{1/K}$. So, given $\varepsilon^\prime$, take $\varepsilon = (\varepsilon^\prime /(C_K (2\|f\|_{L^\infty})^{K/(K+1)}))^{K+1}$ and $r_0^\prime = \min\{(r_0 (\varepsilon)/\widetilde C)^K,r(\widetilde B)/100\}$. With this choice we have $r(B_-) < \widetilde C (r_0^\prime)^{1/K}\leq r_0 (\varepsilon)$ and therefore
                $$
                C_K (2\|f\|_{L^\infty})^{K/(K+1)} \left(\avint_{B_-} |f-m_{B_-}f|\right)^{1/(K+1)} < \varepsilon^\prime,
                $$
                as claimed.
            \end{proof}
        \end{coro}
	
	In the following lemma we prove that if $\phi$ solves the Beltrami equation $\bar\partial\phi=\mu\partial\phi$ with $\mu \in \vmo_c$, then $\phi^{-1}$ not only solves the generalized Beltrami equation $\bar\partial(\phi^{-1})=-(\mu\circ\phi^{-1})\overline{\partial(\phi^{-1})}$ with $\mu\circ\phi^{-1}\in \vmo_c$, but it also satisfies $\bar\partial(\phi^{-1})=\widetilde\mu \partial\phi$ with $\widetilde\mu\in\vmo_c$. In particular, we can also apply \cref{regularity for vmo matrices} to $\phi^{-1}$, see \cref{beltrami in vmo particular case}.
	
	\begin{lemma}\label{general to classical beltrami also is vmo}
		Let $\mu \in \vmo_c$ be a Beltrami coefficient and $\phi \in W^{1,2}_{\loc} (\C)$ be a $K$-quasiconformal homeomorphism of $\C$ satisfying $\bar \partial \phi = \mu \partial \phi$. Then $\phi^{-1}$ satisfies
		\begin{equation}\label{general to classical beltrami equation}
			\bar \partial (\phi^{-1})
			= - (\mu \circ \phi^{-1}) \overline{\partial (\phi^{-1})} 
			= - \left(\left(\mu \frac{\partial \phi}{\overline{\partial \phi}} \right) \circ \phi^{-1}\right)  \partial (\phi^{-1})
		\end{equation}
		with $\mu \circ \phi^{-1}, (\mu \partial \phi / \overline{\partial \phi}) \circ \phi^{-1} \in \vmo_c$.
		\begin{proof}
			The equalities in \rf{general to classical beltrami equation} are proved in \cref{all relations beltrami and pde} and \cite[Proposition 5]{Clop2009}.
			
			From \cref{product in vmo} in \cref{properties of p-mean oscillation} and \cref{vmo compact and bounded are invariant under qc} it suffices to see $\partial \phi / \overline{\partial \phi} \in \vmo$. The strategy is the same as in the proof of \cite[Proposition 5]{Clop2009}. Due to the work of Hamilton in \cite{Hamilton1989} we have $\lambda = \log \partial \phi \in \vmo$. Writing $\partial \phi = e^\lambda$, in particular we have
			$$
			\frac{\partial \phi}{\overline{\partial \phi}} 
			= \frac{e^\lambda}{\overline{e^\lambda}} 
			= \frac{e^{\real \lambda} e^{i \imag \lambda}}{e^{\real \lambda} e^{-i \imag \lambda}}
			= e^{2i \imag \lambda} .
			$$
			The function $\imag \lambda$ belongs in $\vmo$ as $\lambda \in \vmo$. Since $\R\ni t\mapsto e^{it}$ is $1$-Lipschitz, for any ball $B$ we have
			\begin{equation*}
                \begin{aligned}
				\avint_B \left| e^{2i\imag \lambda (x)} - m_B e^{2i\imag \lambda} \right| \, dx 
				&\leq  \avint_B \avint_B \left| e^{2i\imag \lambda (x)}-e^{2i\imag \lambda (y)} \right| \, dy \, dx \\
				&\leq 2\avint_B \avint_B \left| \imag \lambda (x)-\imag \lambda (y) \right| \, dy \, dx 
				\leq  4\avint_B \left| \imag \lambda (x) - m_B \imag \lambda \right| \, dx,
                \end{aligned}
			\end{equation*}
			implying $\partial \phi / \overline{\partial \phi} = e^{2i\imag \lambda} \in \vmo$ as $\imag \lambda \in \vmo$.
		\end{proof}
	\end{lemma}

        We are now ready to prove \cref{homeo with vmo coeff preserve dimension}.

        \begin{proof}[Proof of \cref{homeo with vmo coeff preserve dimension}]
            It follows from \cref{regularity for vmo matrices,beltrami in vmo particular case,general to classical beltrami also is vmo}.
        \end{proof}
	
	\subsubsection{\texorpdfstring{$W_{\loc}^{1,p}$}{Sobolev} coefficients\texorpdfstring{ with $p>\frac{2\lambda^2}{\lambda^2+1}$}{}}
	
	In \cref{sec:holder continous,sec:only one variable dependence}, we had that the quasiconformal mapping $\phi_A$ was bi-Lipschitz, which allowed us to prove the $\sigma$-finite length property of the elliptic measure. In this section we assume $A\in W_{\loc}^{1,p}$ with $p>\frac{2\lambda^2}{\lambda^2+1}$. Although this Sobolev regularity on $A$, and so also for the Beltrami coefficient $\mu_A$, does not provide a bi-Lipschitz quasiconformal solution in general, it still preserves sets with $\sigma$-finite length, see \cref{regularity for sobolev matrices} below. In particular, this will give the Wolff's \cref{wolff thm} in this situation.
	
	\begin{theorem}[See {\cite[Corollary 11 and Proposition 21(c)]{Clop2009}}]\label{regularity for sobolev matrices}
		Let $K\geq 1$ and $\frac{2K^2}{K^2+1} < p \leq 2$. Let $\mu \in W^{1,p}_c (\R^2)$ be a compactly supported Beltrami coefficient with $\|\mu\|_\infty \leq \frac{K-1}{K+1}$, and $\phi$ any $\mu$-quasiconformal mapping. For every compact set $E$,
		$$
		\HH^1 (E) \text{ is }\sigma\text{-finite} \iff \HH^1 (\phi(E)) \text{ is }\sigma\text{-finite}.
		$$
	\end{theorem}
	
	For a non-trivial example of non-bi-Lipschitz $\mu$-quasiconformal mapping with $\mu\in W^{1,2}$, see the function in \cite[(8)]{Clop2009}. We remark that there are also bi-Lipschitz $\mu$-quasiconformal mappings with $\mu\not\in W^{1,2}$. In fact, for $L\geq 1$, $L$-bi-Lipschitz mappings are $\mu$-quasiconformal with $\|\mu\|_\infty\leq (L^2-1)/(L^2+1)$, and in general $\mu\not\in W^{1,2}$. Also, $\mu(z)=\frac{1}{2}\chi_{\mathbb{D}} (z) \not \in W^{1,2}$ provides a bi-Lipschitz $\mu$-quasiconformal mapping. The latter is a particular case in \cite{Mateu-Orobitg-Verdera-2009}.
	
	The result above is precisely what we need to obtain the $\sigma$-finite length property when we invoke Wolff's \cref{wolff thm}.
	
		\begin{proof}[Proof of \cref{sym + det 1 all cases}: Sobolev case]\label{proof of main thm with sobolev}
			Recall we can assume that $\Omega \subset B_{1/4}(0)$ is bounded Wiener regular by \cref{reductions all}. When $p>2$ we have the Sobolev embedding into Hölder continuous functions $W^{1,p} \subset C^{1-2/p}$ by Morrey's inequality, and hence the result follows from \cref{sym + det 1 Holder}. So, assume $p\in \left( \frac{2\lambda^2}{\lambda^2+1} , 2\right]$.
			
			Let $\varphi$ be as in \rf{definition of the support function} and $\phi_A \in W^{1,2}_{\loc}(\C)$ be the normalized solution of the Beltrami equation
			$$
			\bar\partial\phi_A = \varphi \mu_A \partial\phi_A \text{ a.e.\ in }\C, \text{ where }\mu_A \text{ as in }\rf{beltrami coefficients for sym and det 1 matrices},
			$$
			given by the measurable Riemann mapping theorem, see \cref{riemann mapping thm}. Since A is uniformly elliptic, in particular $a_{11} , a_{22} \approx 1$ and hence $2+a_{11} + a_{22} \geq 2$. Since $A\in W_{\loc}^{1,p} (\R^2)$, we have that $\mu_A \in W_{\loc}^{1,p} (\R^2)$ and therefore $\varphi\mu_A\in W_c^{1,p}(\R^2)$. The theorem follows by the same proof of \cref{main thm general domains} on \cpageref{proof of main thm general domains}, using \cref{lemma:relation elliptic measure for sym and det 1} and the fact that the function $\phi_A$ preserve sets of $\sigma$-finite length, see \cref{regularity for sobolev matrices}. That is, if $F\subset \partial(\phi_A (\Omega))$ is the set given by Wolff's \cref{wolff thm}, then $\phi_A^{-1} (F)\subset \partial\Omega$ has full elliptic measure and $\sigma$-finite length.
		\end{proof}
	
	\subsection{Proof of Theorem~\ref{sym + det 1 all cases} via PDE regularity}
	
	Similarly as we did in the previous \cref{sec:quasiconformal regularity}, but using a PDE approach instead of applying quasiconformal theory, in this section we deduce better regularity of the quasiconformal change of variables assuming regularity on the matrix $A$, which allows proving the remaining cases of \cref{sym + det 1 all cases}.

    Let us sketch this when the matrix has Hölder coefficients. As in the proof of \cref{sym + det 1 Holder} on \cpageref{proof of sym + det 1 holder}, we can assume that $A\in C^\alpha (U_\varepsilon (\partial\Omega))$ for $0<\alpha<1$ and $\varepsilon>0$. Let $\phi_A \in W^{1,2}_{\loc}(\C)$ be the normalized solution of the Beltrami equation
    $$
    \bar\partial\phi_A = \mu_A \partial\phi_A \text{ a.e.\ in }\C, \text{ where }\mu_A \text{ as in }\rf{beltrami coefficients for sym and det 1 matrices},
    $$
    given by the measurable Riemann mapping theorem, see \cref{riemann mapping thm}. By \cref{all relations beltrami and pde}, $u\coloneqq \real \phi_A$, $v\coloneqq \imag \phi_A$, $R\coloneqq \real \phi_A^{-1}$ and $I\coloneqq \imag \phi_A^{-1}$ satisfy $\divv A \nabla u = 0$, $\divv A \nabla v = 0$, $\divv B \nabla u = 0$ and $\divv B^* \nabla v = 0$ respectively, where $B$ and $B^*$ are the uniformly elliptic matrices in \rf{matrix of inverse function beltrami}. By the Hölder regularity of quasiconformal mappings we have that $\phi^{-1}_A \in C^{1/\lambda}_{\loc}$, and hence $A\in C^\alpha (U_\varepsilon (\partial\Omega))$ implies $B, B^*\in C^{\alpha/\lambda}_{\loc}$ in an open neighborhood of $\partial(\phi_A (\Omega))$. All in all, $u,v,I,R$ satisfy a PDE with Hölder coefficients and therefore $\phi_A$ is bi-Lipschitz in an open neighborhood of $\partial\Omega$. \Cref{sym + det 1 Holder} follows from the same proof of \cref{main thm general domains} on \cpageref{proof of main thm general domains} using \cref{lemma:relation elliptic measure for sym and det 1} and the fact that $\phi_A$ is bi-Lipschitz in an open neighborhood of $\partial\Omega$.
	
	\subsubsection{Dini mean oscillation}\label{sec:dini mean oscillation}
	
	We consider matrices with $p$-Dini mean oscillation.

        \begin{definition}\label{def:dini mean oscillation}
		For $1\leq p <\infty$ and $f\in L_{\loc}^p(\R^{n+1})$, denote
		$$
		f_{{\rm MO},p} (r) \coloneqq \sup_{x\in \R^{n+1}} \left(\avint_{B(x,r)} \left|f(z)-m_{B(x,r)} f\right|^p \, dz\right)^{1/p},
		$$
		where $\rm MO$ stands for ``mean oscillation''. We say that $f$ is of $p$-Dini mean oscillation, $f\in {\rm DMO}_p$, if $f_{{\rm MO},p}$ satisfies the Dini condition
		$$
		\int_0^1 f_{{\rm MO},p} (t) \, \frac{dt}{t} < \infty .
		$$
	\end{definition}

        \begin{rem}\label{DMO inclusions + DINI and VMO}
            From Hölder's inequality, we have the inclusion ${\rm DMO}_q \subset {\rm DMO}_p$ if $1\leq p < q <\infty$. The Dini condition of the mean oscillation directly implies ${\rm DMO}_p \subset \vmo$ for every $1\leq p <\infty$. Moreover, every ${\rm DMO}_p$ ($1\leq p <\infty$) contains the (pointwise) Dini space, that is, functions $f$ satisfying the Dini condition
            \begin{equation}\label{pointwise dini condition}
            \int_0^1 \tau (t) \, \frac{dt}{t} < \infty,\text{ where } \tau (t) = \sup_{|x-y|\leq t} |f(x)-f(y)|.
            \end{equation}
            All in all,
            $$
            \text{Hölder} \subset {\rm Dini} \subset {\rm DMO}_q \subset {\rm DMO}_p \subset \vmo, \text{ for any }1\leq p < q <\infty.
            $$
        \end{rem}
	
	Under the $1$-Dini mean oscillation regularity on the matrix coefficients, $L_A$-harmonic functions are $C^1$. More precisely, we have the following result.
	
	\begin{theorem}[See {\cite[Theorem 1.5]{Dong2017}}]\label{regularity for dini mean oscillation matrices}
		Suppose the coefficients of $A(x)=(a_{ij}(x))_{i,j=1}^{n+1}$ satisfy
		$$
		\max_{1\leq i,j\leq n+1}\int_0^1 \eta_{a_{ij}} (t) \, \frac{dt}{t} < \infty,\text{ where } \eta_{a_{ij}} (r) \coloneqq \sup_{x\in B_3 (0)} \avint_{B(x,r)} \left|a_{ij}(z)-m_{B(x,r)} a_{ij}\right| \, dz,
		$$
		and let $u \in W^{1,2} (B_4 (0))$ be a $L_A$-solution in $B_4 (0)$. Then, we have $u \in C^1 (\overline{B_1 (0)})$.
	\end{theorem}
	
	In the following lemma, we see how the $1$-Dini mean oscillation regularity changes under $K$-quasiconformal mappings.
	
	\begin{lemma}\label{dini condition under qc}
		Let $K\geq 1$ and let $\phi\in W^{1,2}_{\loc} (\C)$ be a $K$-quasiconformal homeomorphism of $\C$. For any function $f$ and any $K<p<\infty$, there exists a constant $C_{K,\phi}$ such that
		$$
		\int_0^1 \eta_{(f\circ \phi^{-1})} (t) \frac{dt}{t} \lesssim_{p,K} \int_0^{C_{K,\phi}} f_{{\rm MO},p} (t) \frac{dt}{t}.
		$$
		\begin{proof}
			Fix $\widetilde B=B_{100} (0)$. Recall that $\phi^{-1}$ is $\eta_K$-quasisymmetric as it is global $K$-quasiconformal, see \rf{quasisymmetric condition}. Take $T=T_K\geq 2$ big enough so that $\eta_K (1/T) \leq 1/2$, and denote $r_m = T^{-m}$. Let $x\in B_3 (0)$. First note that for $r_{m+1} \leq r \leq r_m$, and changing from $m_{B(x,r)} (f \circ \phi^{-1})$ to $m_{B(x,r_m)} (f \circ \phi^{-1})$ as we did in \rf{vmo to other constant}, we have
			$$
			\eta_{(f\circ \phi^{-1})} (x,r) \coloneqq \avint_{B(x,r)} \left|(f\circ \phi^{-1} )(z)-m_{B(x,r)} (f \circ \phi^{-1})\right| \, dz \lesssim \eta_{(f\circ \phi^{-1})} (x,r_m) .
			$$
			Thus, $\int_0^1 \eta_{(f\circ \phi^{-1})} (x,t) \frac{dt}{t} \lesssim \sum_{m=0}^{\infty} \eta_{(f\circ \phi^{-1})} (x,r_m)$.
			
			Denote now $\widetilde r_m (x) \coloneqq \max_{|\xi-x|=r_m} |\phi^{-1} (\xi) - \phi^{-1} (x)|$, recall the ball defined in \rf{ball in the image vmo}. We simply write $\widetilde r_m$ as $x\in B_3 (0)$ is fixed in this proof and the constants below do not depend on $x$. By \rf{distortion of mean oscillation under qc} in \cref{bmo under quasiconformal homeo} we have
			$$
			\eta_{(f\circ \phi^{-1})} (x,r_m) \lesssim_{p,K}
			\left(\avint_{B(\phi^{-1}(x),\widetilde r_m)} \left|f(z)-m_{B(\phi^{-1}(x),\widetilde r_m (x))} f\right|^p \right)^{1/p} \eqqcolon \eta_f (\phi^{-1} (x),\widetilde r_m,p).
			$$
			
			By the quasisymmetry condition we have $1/\eta_K (T) \leq \widetilde r_{m} / \widetilde r_{m-1} \leq \eta_K (1/T) \leq 1/2$, and hence $1 \leq 2 \int_{\widetilde r_m}^{\widetilde r_{m-1}} \frac{dt}{t}$. As before, for every $\widetilde r_m \leq r \leq \widetilde r_{m-1}$ we have $\eta_f (\phi^{-1} (x), \widetilde r_m,p) \lesssim \eta_f (\phi^{-1} (x),r,p)$. By \cref{bmo under quasiconformal homeo} we have $\widetilde r_m \leq C_{K,\widetilde B,\phi} r_m^{1/K}$, and in particular $\widetilde r_{-1} \leq C_{K,\widetilde B,\phi} T^{1/K}$. As a consequence of these observations respectively, we get 
			\begin{equation*}
				\begin{aligned}
					\sum_{m=0}^{\infty} \eta_{(f\circ \phi^{-1})} (x,r_m)
					&\lesssim \sum_{m=0}^{\infty} \eta_{f} (\phi^{-1} (x),\widetilde r_m, p) \lesssim
					\sum_{m=0}^\infty \int_{\widetilde r_m}^{\widetilde r_{m-1}} \eta_{f} (\phi^{-1} (x),\widetilde r_m, p) \frac{dt}{t} \\
					&\lesssim \sum_{m=0}^\infty \int_{\widetilde r_m}^{\widetilde r_{m-1}} \eta_{f} (\phi^{-1} (x),t, p) \frac{dt}{t}
					= \int_{0}^{C_{K,\widetilde B,\phi} T^{1/K}} \eta_{f} (\phi^{-1} (x),t, p) \frac{dt}{t}.
				\end{aligned}
			\end{equation*}
			The last element is bounded by $\int_{0}^{C_{K,\widetilde B,\phi} T^{1/K}} f_{{\rm MO},p} (t) \frac{dt}{t}$. Taking the supremum over all the points $x\in B_3 (0)$ we obtain the lemma with the constant $C_{K,\widetilde B,\phi} T^{1/K}$.
		\end{proof}
	\end{lemma}
	
	Using the previous lemma we deduce the $C^1$ regularity $\phi_A^{-1}$ in the proof of the following result when $A\in {\rm DMO}_p$ with $p>\lambda$.
 
		\begin{proof}[Proof of \cref{sym + det 1 all cases}: ${\rm DMO}$ case]\label{proof of dini mean oscillation elliptic measure}
			Recall we can assume that $\Omega \subset B_{1/4} (0)$ is bounded Wiener regular by \cref{reductions all}, and we can replace finitely connected by simply connected by \cref{claim:candidate for each index}.
			
			Let us see first the simply connected case. Assume $A\in {\rm DMO}_1$ and let $\phi_A \in W^{1,2}_{\loc}(\C)$ be the normalized solution of the Beltrami equation
			$$
			\bar\partial\phi_A = \mu_A \partial\phi_A \text{ a.e.\ in }\C, \text{ where }\mu_A \text{ as in }\rf{beltrami coefficients for sym and det 1 matrices},
			$$
			given by the measurable Riemann mapping theorem, see \cref{riemann mapping thm}. In particular $u\coloneqq \real \phi_A$ and $v\coloneqq \imag \phi_A$ satisfy $\nabla v = *A\nabla u$, and $\divv A \nabla u = 0$, see \cref{all relations beltrami and pde}. Now, since $A \in \rm DMO_1$ we have that $|\nabla u|$ is bounded in $B_{1/2} (0)$, see \cref{regularity for dini mean oscillation matrices}, and by the relation $\nabla v = *A\nabla u$ we also have that $|\nabla v|$ is bounded in $B_{1/2} (0)$. In particular, $u$ and $v$ are Lipschitz and hence $\phi_A$ is Lipschitz in $B_{1/2} (0) \supset \partial \Omega$. This implies $\omega_{\Omega,A}\ll \HH^{\varphi_{1,C}}$ by the same proof of \cref{main thm general domains} on \cpageref{proof of main thm general domains}, using \cref{lemma:relation elliptic measure for sym and det 1}.
			
			We now turn to the $\sigma$-finite length for general domains. Suppose first $A\in {\rm DMO}_p$ with $p > \lambda$, and later we will see we can always reduce to this case. Let $\phi_A$ be as above. Recall from \cref{all relations beltrami and pde} that the functions $R\coloneqq \real (\phi_A^{-1})$ and $I\coloneqq \imag (\phi_A^{-1})$ satisfy $\nabla I = * B \nabla R$, and $\divv B \nabla R = 0$, where $B$ is the uniformly elliptic matrix in \rf{matrix of inverse function beltrami}. By \cref{dini condition under qc}, \cref{fraction in vmo} in \cref{properties of p-mean oscillation}, and $A\in {\rm DMO}_p$ with $p>\lambda$ we obtain that $B\in {\rm DMO}_1$. By \cref{regularity for dini mean oscillation matrices} we have that $R\in C^1$, i.e., $|\nabla R|$ is bounded around $\partial\phi_A(\Omega)$, and by the relation $\nabla I = * B \nabla R$ we also get that $|\nabla I|$ is bounded there. Hence, $R=\real (\phi_A^{-1})$ and $I=\imag (\phi_A^{-1})$ are Lipschitz, and in particular $\phi_A^{-1}$ is Lipschitz. The same proof of \cref{main thm general domains} on \cpageref{proof of main thm general domains}, using \cref{lemma:relation elliptic measure for sym and det 1}, implies the $\sigma$-finite length.
			
			Let us see now the case $A\in \mathrm{DMO}_p$ with $p>1$ follows from the previous reduction. As $A\in \mathrm{DMO}_p \subset \mathrm{DMO}_1$, by \cite[Appendix]{Hwang-Kim-2020} we have that it agrees a.e.\ with a matrix with (uniformly) continuous coefficients, which by the $L_A$ weak form we can assume that we are in this latter case. In particular, given $\varepsilon>0$, for each $\xi \in \overline\Omega$ there is $\delta_{\xi,\varepsilon}>0$ such that
			$$
			z\in \overline{6B_{\delta_{\xi,\varepsilon}}(\xi)} \implies |a_{ij}(z)-a_{ij}(\xi)|<\varepsilon \text{ for any }i,j\in \{1,2\}.
			$$ 
			Let $\{B_i\}_i\subset \{B_s(\xi)\}_{\xi\in \partial\Omega,0<s\leq \delta_{\xi,\varepsilon}}$ be a countably subfamily of disjoint balls such that $\omega^p (\partial\Omega\setminus \bigcup_i B_i)=0$, by Vitali's covering theorem. The goal is to prove that for each index $i$ there exists a set $F_i\subset \partial(\Omega\cap 4B_i)$ satisfying $\omega_{\Omega\cap 4B_i,A}(F_i)=1$ and with $\sigma$-finite $\HH^1(F_i)$, because by \cref{localization argument} we would have that the set $F=\bigcup_i F_i$ satisfies $\omega_{\Omega,A} (F)=1$ and of course it has $\sigma$-finite length.
			
			Let us now fix an index $i$. Let $\xi_i$ denote the center of the ball $B_i$, and $r_i$ its radius. Let $S=\sqrt{A(\xi_i)}$ be the matrix satisfying $S^T S=A(\xi_i)$, and $M$ as in \rf{matrix after change of constant sym change of variables}, i.e.,
			$$
			M(z)=Id + (S^T)^{-1}(A(S^T)-A(\xi))S^{-1} \overset{\text{\rf{matrix after change of constant sym change of variables - sym case}}}{=} (S^T)^{-1}A(S^T)S^{-1},
			$$
			which is symmetric and with determinant $1$, see \cref{matrix after change of constant sym change of variables - sym case - is sym and det 1}. Then, by \cref{elliptic measure deformation qc} we have
			$$
			\omega^{q}_{\Omega\cap 4B_i, A} (E) = \omega_{(S^T)^{-1}(\Omega\cap 4B_i), M}^{(S^T)^{-1}q} ((S^T)^{-1}E),
			$$
			for any $q\in \Omega\cap 4B_i$ and $E\in \partial (\Omega\cap 4B_i)$. Given $\varphi \in C^\infty_c (6B_i)$ such that $0\leq \varphi\leq 1$, $\varphi|_{5B_i}=1$ and $|\nabla \varphi |\lesssim 1/r_i$, consider the function $\widetilde \varphi = \varphi \circ S^T$. Using this function $\widetilde \varphi$, we can directly replace the matrix $M$ in the right-hand side above by simply the uniformly elliptic, symmetric and with determinant $1$ matrix $\widehat M$ such that $\widetilde\varphi\mu_M = \mu_{\widehat M}$, see the precise definition in \cref{rem:interpolation to preserve determinant}. That is,
			$$
			\omega^{q}_{\Omega\cap 4B_i, A} (E) = \omega_{(S^T)^{-1}(\Omega\cap 4B_i), \widehat M}^{(S^T)^{-1}q} ((S^T)^{-1}E)
			$$
			for any $q\in \Omega\cap 4B_i$ and $E\in \partial (\Omega\cap 4B_i)$.
			
			From the proof of \cref{perturbation of constant matrix}, $M$ has ellipticity constant $\lambda_i=(1-2\varepsilon\lambda)^{-1}$ in $(S^T)^{-1}(6B_i)$. This implies that $|\mu_M|\leq (\lambda_i-1)/(\lambda_i+1)$ in $(S^T)^{-1}(6B_i)$. In particular, $|\mu_{\widehat M}| = |\widetilde\varphi\mu_M|\leq |\characteristic_{(S^T)^{-1}(6B_i)}\mu_M|\leq (\lambda_i-1)/(\lambda_i+1)$ everywhere, implying that $\widehat M$ has ellipticity constant $\lambda_i=(1-2\varepsilon\lambda)^{-1}$ everywhere. If $\varepsilon>0$ is small enough then $1\leq\lambda_i=(1-2\varepsilon\lambda)^{-1} <p$; recall that $p>1$ is the value so that $A\in \mathrm{DMO}_p$.
			
			Let us see $\widehat M \in\mathrm{DMO}_p$. As $A\in \mathrm{DMO}_p$, we have that $(S^T)^{-1}A (\cdot)S^{-1}\in \mathrm{DMO}_p$, and as $z\mapsto S^T z$ is a linear mapping we get $(S^T)^{-1}A(S^T \cdot)S^{-1}\in \mathrm{DMO}_p$. Using now that $\widetilde \varphi = \varphi \circ S^T$ is smooth, bounded and with compact support, applying \cref{properties of p-mean oscillation} we conclude that $\widehat M\in \mathrm{DMO}_p$.
			
			As the ellipticity constant of $\widehat M$ is $1\leq \lambda_i <p$ and $\widehat M \in \mathrm{DMO}_{p}$, we are in the case of the second paragraph in this proof, and hence there is a set $F_i\subset \partial(\Omega\cap 4B_i)$ satisfying $\omega_{(S^T)^{-1}(\Omega\cap 4B_i), \widehat M}(F)=1$ and with $\sigma$-finite $\HH^1(F_i)$. In particular $\omega_{\Omega\cap 4B_i, A} (S^T(F))=1$ and $S^T (F)$ has $\sigma$-finite length, and the theorem follows.
		\end{proof}
	
	\subsubsection{Mean oscillation with logarithm decay}
	
	In this section, we consider matrices with the quantitative $\log$ decay of the mean oscillation in \cite[(1.3)]{Acquistapace1992}.
	\begin{definition}\label{def:log mean oscillation}
		We say that $f\in L_{\loc}^1 (\R^{n+1})$ is of vanishing mean oscillation with logarithmic decay, $f \in \vmo_{\log}$ for short, if
		\begin{equation}\label{log decay}
			\sup_{\substack{x\in \R^{n+1}\\ s\in (0,r]}} \left[\left(1+|{\log s}|\right)\avint_{Q_s(x)} \left|f-m_{Q_s(x)} f\right| \right] \xrightarrow{r\to 0} 0,
		\end{equation}
		recall $Q_s(x)$ denotes the cube with center $x$ and side length $2s$.
	\end{definition}

        \begin{rem}\label{DMO subset VMOlog subset VMO}
            The logarithmic decay in \rf{log decay} implies $\vmo_{\log} \subset \vmo$. On the other hand, for a function $f\in {\rm DMO}_1$, given $\varepsilon >0$ we can take $s_0 =s_0(\varepsilon)$ such that $\int_0^{s_0} f_{{\rm MO},1} (t) \, \frac{dt}{t} < \varepsilon$, and hence $(1+|{\log t}|)f_{{\rm MO},1} (t) < t (1+|{\log t}|) \varepsilon / s_0$ for all $0<t< s_0$. Thus, we can take $\delta = \delta (\varepsilon)$ small enough so that $0<t<\delta$ implies $(1+|{\log t}|)f_{{\rm MO},1} (t) <\varepsilon$. Interchanging balls $B(x,t)$ by cubes $Q_t(x)$ in the definition of $f_{{\rm MO},1} (t)$ we obtain that $f$ satisfies \rf{log decay}. That is, ${\rm DMO}_1 \subset \vmo_{\log}$. All in all,
            $$
            {\rm DMO}_1 \subset \vmo_{\log} \subset \vmo.
            $$
        \end{rem}
	
	The fundamental property \rf{local equivalence bmo norm} of $\bmo$ functions implies that, for $1\leq p <\infty$, $f\in \bmo$ and $r\in(0,1)$, we have
	\begin{equation}\label{eq:vmo log only matters one exponent}
		\sup_{\substack{x\in \R^{n+1}\\ s\in (0,r]}} \left[\left(1+|{\log s}|\right) \avint_{Q_s(x)} |f-m_{Q_s(x)} f| \right]
        \approx_p \sup_{\substack{x\in \R^{n+1}\\ s\in (0,r]}} \left[\left(1+|{\log s}|\right) \left(\avint_{Q_s(x)} |f-m_{Q_s(x)} f|^p\right)^{1/p} \right],
	\end{equation}
	and therefore $f\in \vmo_{\log}$ if and only if for some (and hence for all) $1\leq p <\infty$ the right-hand side term above goes to $0$ as $r\to 0$. Indeed, for $x\in \R^2$ and $s\in (0,r]$, using \rf{local equivalence bmo norm} we have
	\begin{equation}\label{eq:vmo log only matters one exponent step 1}
		\left(\avint_{Q_s(x)} |f-m_{Q_s(x)} f|^p\right)^{1/p}
		\leq \sup_{Q \subset Q_s(x)} \left(\avint_{Q} |f-m_{Q} f|^p\right)^{1/p}
		\approx_p \sup_{Q \subset Q_s(x)} \avint_{Q} |f-m_{Q} f|.
	\end{equation}
	Since $|{\log t}|$ is decreasing in $(0,1)$ and every $Q\subset Q_s(x)$ satisfies $\ell (Q) \leq 2s$, we obtain
	\begin{equation}\label{eq:vmo log only matters one exponent step 2}
		\begin{aligned}
			\sup_{Q\subset Q_s(x)} (1+|{\log s}|) \avint_{Q} |f-m_{Q} f| 
			&\leq \sup_{Q\subset Q_s(x)} \left(1+\left|{\log \frac{\ell(Q)}{2}}\right|\right) \avint_{Q} |f-m_{Q} f| \\
			&\leq \sup_{\substack{x\in \R^2\\ s\in (0,r]}} \left[\left(1+|{\log s}|\right) \avint_{Q_s(x)} |f-m_{Q_s(x)} f| \right],
		\end{aligned}
	\end{equation}
	and inequality $\gtrsim_p$ in \rf{eq:vmo log only matters one exponent} follows from \rf{eq:vmo log only matters one exponent step 1} and \rf{eq:vmo log only matters one exponent step 2} by taking the supremum over $x\in \R^2$ and $s\in(0,r]$. The other inequality in \rf{eq:vmo log only matters one exponent} is just Hölder's inequality.
	
	Assuming that the matrix has $\vmo_{\log}$ coefficients, the following result states that $L_A$-harmonic solutions have the gradient in $\bmo$.
	
	\begin{theorem}[See {\cite[Theorem 4.1]{Acquistapace1992}}]\label{regularity for log mean oscillation matrices}
		Let $A \in L^\infty (Q_3(0)) \cap \vmo_{\log}$ and $u \in W^{1,2} (Q_3 (0))$ be a $L_A$-solution in $Q_3 (0)$. Then $\nabla u \in \bmo (Q_1 (0))$.
	\end{theorem}
	
	In the following lemma, we see that functions with derivatives in the space $\bmo$ satisfy the so-called $\log$-Lip regularity.
	
	\begin{lemma}\label{log lip when grad in bmo}
		Let $f : B_1(0) \subset \R^{n+1} \to \R$ be a function with $\nabla f \in L^1 (B_1 (0)) \cap \bmo (B_1(0))$, and let $x,y\in B_{1/2} (0)$. If $|x-y|$ is small enough, only depending on the $L^1$ and $\bmo$ norms, then
		$$
		|f(x)-f(y)| \lesssim \|\nabla f\|_* |x-y| \log \frac{1}{|x-y|}.
		$$
		\begin{proof}
			Consider $B\coloneqq B(\frac{x+y}{2}, |x-y|) \subset B_1 (0)$, and write $r_B = |x-y|$. By triangle inequality $|f(x)-f(y)| \leq |f(x)-m_B f|+|f(y)-m_B f|$, and by symmetry, it suffices to control one of them. By \cite[Lemma 7.16]{Gilbarg2001},
			$$
			|f(x)-m_B f| \lesssim \int_B \frac{|\nabla f (z)|}{|x-z|} \, dz .
			$$
			Splitting the domain of integration $B$ of the right-hand side integral in annulus, we obtain
			\begin{align*}
				\int_B \frac{|\nabla f (z)|}{|x-z|} \, dz 
				&= \sum_{k\geq 0} \int_{2^{-k}B \setminus 2^{-(k+1)}B} \frac{|\nabla f (z)|}{|x-z|} \, dz 
				\approx \frac{2^k}{r_B}\sum_{k\geq 0} \int_{2^{-k}B \setminus 2^{-(k+1)}B} |\nabla f (z)| \, dz \\
				& \lesssim r_B \sum_{k\geq 0} 2^{-k} \avint_{2^{-k} B} |\nabla f (z)| \, dz.
			\end{align*}
			
			Let us bound each element on the right-hand side in terms of $k\geq 0$. For each $k\geq0$, let $N_k \in \N$ the largest index such that $2^{N_k} r_{2^{-k} B} = 2^{N_k} 2^{-k} r_B \leq 1/2$, i.e., $N_k \coloneqq \lfloor \log_2 \frac{2^k}{2|x-y|} \rfloor$. With this choice,
			\begin{equation*}
				\avint_{2^{-k} B} |\nabla f (z)| \, dz
				= \sum_{m=0}^{N_k -1} \left(\avint_{2^m 2^{-k} B} |\nabla f (z)|\, dz - \avint_{2^{m+1} 2^{-k} B} |\nabla f (z)| \, dz\right) 
				+\avint_{2^{N_k} B} |\nabla f (z)| \, dz .
			\end{equation*}
			Each element in the sum is controlled in absolute value by $\lesssim \|\nabla f\|_*$. The term $\avint_{2^{N_k} B} |\nabla f (z)| \, dz$ is controlled by $\lesssim\|\nabla f\|_{L^1 (B_1 (0))} $ as $2^{N_k} B \subset B_1 (0)$ and $2^{N_k}r_B \approx 1$ by the choice of $N_k$. In particular
			$$
			\avint_{2^{-k} B} |\nabla f (z)| \, dz \lesssim N_k \|\nabla f\|_* + \|\nabla f\|_{L^1 (B_1 (0))} .
			$$
			If $|x-y|$ is small enough depending on $L^1$ and $\bmo$ norm, more precisely, take small enough $|x-y|$ to satisfy $ \|\nabla f\|_{L^1 (B_1 (0))} \leq \lfloor \log_2 \frac{1}{2|x-y|} \rfloor \|\nabla f\|_*$, then $\|\nabla f\|_{L^1 (B_1 (0))} \leq N_k \|\nabla f\|_*$, and therefore
			$$
			\avint_{2^{-k} B} |\nabla f (z)| \, dz \lesssim N_k \|\nabla f\|_* .
			$$
			
			Summing over $k\geq 0$,
			$$
			|f(x)-m_B f| \lesssim \int_B \frac{|\nabla f (z)|}{|x-z|} \, dz 
			\lesssim r_B \sum_{k\geq 0} 2^{-k} \avint_{2^{-k} B} |\nabla f (z)| \, dz 
			\lesssim r_B \|\nabla f\|_* \sum_{k\geq 0} 2^{-k} N_k .
			$$
			Recall $N_k =  \lfloor \log_2 \frac{2^k}{2|x-y|} \rfloor \approx \log \frac{2^k}{|x-y|}$. Hence, the sum in the right-hand side is controlled by
			$$
			\sum_{k\geq 0} 2^{-k} \log \frac{2^k}{|x-y|} 
			= \sum_{k\geq 0} 2^{-k} \log 2^k 
			+ \log \frac{1}{|x-y|}  \sum_{k\geq 0} 2^{-k} .
			$$
			As $\sum_{k\geq 0} 2^{-k} \log 2^k < \infty$, for small enough $|x-y|$ we have $\sum_{k\geq 0} 2^{-k} \log 2^k \leq \log \frac{1}{|x-y|}$, and hence $\sum_{k\geq 0} 2^{-k} \log \frac{2^k}{|x-y|}  \lesssim \log \frac{1}{|x-y|}$.
		\end{proof}
	\end{lemma}
	
	Here we see the distortion of Hausdorff measures under $\log$-Lip functions.
	
	\begin{lemma}\label{distortion under log lip functions}
		If $f : \R^{n+1}\to \R^{n+1}$ satisfies $|f(x)-f(y)| \leq C_{\mathrm{logLip}} |x-y| \log \frac{1}{|x-y|}$ for $|x-y|$ small enough, $g(t) = t/\log \frac{1}{t}$, $h(t) = t\log\frac{1}{t}$, and $A \subset \R^{n+1}$, then
		\begin{enumerate}
			\item $\HH^g (f(A)) \leq 2C_{\mathrm{logLip}} \HH^1 (A)$, and
			\item $\HH^{\varphi_{\rho,C}} (f(A)) \leq 2C_{\mathrm{logLip}}^{\frac{2}{\rho+1}} \HH^{(\varphi_{\rho,C}) \circ h} (A)$.
		\end{enumerate}
		\begin{proof}
			For small enough $\delta >0$, the function $t \mapsto t\log \frac{1}{t}$ is increasing in $t\in (0,\delta)$. Let $\eta = C_{\mathrm{logLip}} \delta \log \frac{1}{\delta}>0$, and note that $\eta \to 0$ as $\delta \to 0$.
			
			For $\varepsilon >0$, let $\{E_i\}_i$ such that $A\subset \bigcup_i E_i$, $\diam\, E_i \leq \delta$ and $\sum_i \diam\, E_i \leq \HH^1_\delta (A) + \varepsilon$. Then $f(A) \subset \bigcup_i f(E_i)$ and $\diam(f(E_i)) \leq C_{\mathrm{logLip}}\diam\, E_i \log \frac{1}{\diam\, E_i} \leq C_{\mathrm{logLip}}\delta \log \frac{1}{\delta} = \eta$. As $g(t) = t / \log \frac{1}{t}$ is increasing in $(0,\eta)$ if $\eta$ is small enough, we have
			\begin{align*}
				g(\diam f(E_i)) & \leq g\left(C_{\mathrm{logLip}}\diam\, E_i \log \frac{1}{\diam\, E_i}\right) = \frac{C_{\mathrm{logLip}}\diam\, E_i \log \frac{1}{\diam\, E_i}}{\log \left(\frac{1}{C_{\mathrm{logLip}}\diam\, E_i \log \frac{1}{\diam\, E_i}}\right)} \\
				& = \frac{C_{\mathrm{logLip}}\diam\, E_i \log \frac{1}{\diam\, E_i}}{\log \frac{1}{C_{\mathrm{logLip}}} + \log \frac{1}{\diam\, E_i} + \log \log \frac{1}{\diam\, E_i}} 
				\leq 2C_{\mathrm{logLip}} \diam\, E_i ,
			\end{align*}
			where we used that $\diam\, E_i \leq \delta$ is small enough. All in all,
			$$
			\HH^g_\eta (f(A)) \leq  \sum_i g(\diam f(E_i)) \leq 2C_{\mathrm{logLip}} \sum_i \diam\, E_i \leq 2C_{\mathrm{logLip}} (\HH^1_\delta (A) + \varepsilon) \leq 2C_{\mathrm{logLip}} (\HH^1 (A) + \varepsilon) .
			$$
			Letting $\varepsilon \to 0$ and $\delta \to 0$, i.e., $\eta \to 0$, we get $\HH^g (f(A)) \leq 2C_{\mathrm{logLip}} \HH^1 (A)$ as claimed.
			
			Similarly, for $\varepsilon >0$, let $\{E_i\}_i$ such that $A\subset \bigcup_i E_i$, $\diam\, E_i \leq \delta$ and $\sum_i ((\varphi_{\rho,C}) \circ h) (\diam\, E_i) \leq \HH_\delta^{(\varphi_{\rho,C}) \circ h} (A) + \varepsilon$. As before,
			\begin{align*}
				\HH^{\varphi_{\rho,C}}_\eta (f(A)) & \leq \sum_i \varphi_{\rho,C} (\diam f(E_i)) \leq \sum_i \varphi_{\rho,C} (C_{\mathrm{logLip}} h(\diam\, E_i)) \\
				&\leq 2C_{\mathrm{logLip}}^{\frac{2}{\rho+1}} \sum_i ((\varphi_{\rho,C}) \circ h) (\diam\, E_i)
				\leq 2C_{\mathrm{logLip}}^{\frac{2}{\rho+1}} (\HH^{(\varphi_{\rho,C}) \circ h} (A) + \varepsilon),
			\end{align*}
			where we used $\varphi_{\rho,C} (C_{\mathrm{logLip}} h(\diam\, E_i))\leq 2C_{\mathrm{logLip}}^{\frac{2}{\rho+1}}\varphi_{\rho,C}(h(\diam\, E_i))$, see the proof of \cref{distortion zero measure of makarov gauge function}. Letting $\varepsilon, \delta \to 0$, and hence also $\eta \to 0$, we obtain $\HH^{\varphi_{\rho,C}} (f(A))\leq 2C_{\mathrm{logLip}} \HH^{(\varphi_{\rho,C}) \circ h} (A)$.
		\end{proof}
	\end{lemma}
	
	\begin{lemma}\label{vmo log invariant under qc}
		$\vmo_{\log,c}$ is invariant under global homeomorphic quasiconformal mappings.
		\begin{proof}
			Let $f\in \vmo_{\log}$ be non-identically zero (otherwise we are done) with compact support and $\phi\in W^{1,2}_{\loc} (\C)$ be a $K$-quasiconformal homeomormism of $\C$, with $K\geq 1$. Let $\widetilde B$ be a ball such that $\supp (f\circ \phi^{-1})\subset \widetilde B/2$, $r < r(\widetilde B)/100$ and fix $\widetilde C \coloneqq C_{K,2\widetilde B, \phi}$, the constant in \cref{bmo under quasiconformal homeo} so that $r_{B_-} \leq \widetilde C r_B^{1/K}$ for every ball $B\subset 2\widetilde B$. For $x\in \widetilde B$ and $s\in (0,r]$ (otherwise the first integral below is zero), by \cref{bmo under quasiconformal homeo} and \rf{local equivalence bmo norm}\footnote{The last step using \rf{local equivalence bmo norm} is unnecessary since the rest of the proof works with the exponent $2K$ by \cref{eq:vmo log only matters one exponent}. However, to simplify the notation, we remove the exponent $2K$.} we have
			\begin{equation}\label{log decay new matrix B step 1}
				\begin{aligned}
					\avint_{Q_s(x)} \left|f\circ\phi^{-1}-m_{Q_s(x)} (f\circ\phi^{-1})\right|
					&\leq \sup_{Q \subset Q_s(x)} \avint_{Q} \left|f\circ\phi^{-1}-m_{Q} (f\circ\phi^{-1})\right| \\
					& \lesssim
					\sup_{Q \subset Q(\phi(x),Cs^{1/K})} \left(\avint_{Q} \left|f-m_{Q} f\right|^{2K}\right)^{\frac{1}{2K}}\\
					&\approx \sup_{Q \subset Q(\phi(x),Cs^{1/K})} \avint_{Q} \left|f-m_{Q} f\right|.
				\end{aligned}
			\end{equation}
			As we did in \rf{eq:vmo log only matters one exponent step 2}, we obtain
			\begin{equation}\label{log decay new matrix B step 2}
				(1+|{\log (Cs^{1/K})}|) \sup_{Q \subset Q(\phi(x),Cs^{1/K})} \avint_{Q} \left|f-m_{Q} f\right|
				\leq \sup_{\substack{x\in \R^2\\ s\in (0,Cr^{1/K}]}} \left[\left(1+\left|{\log s}\right|\right) \avint_{Q_s(x)} \left|f-m_{Q_s(x)} f\right| \right].
			\end{equation}
			The lemma follows from \rf{log decay new matrix B step 1}, \rf{log decay new matrix B step 2}, the fact that $(1+|{\log s}|) \leq 2K (1+|{\log (Cs^{1/K})}|)$ for small enough $s$, and finally taking $r\to 0$.
		\end{proof}
	\end{lemma}
	
	Using the previous results, we conclude this section with the following result about the dimension of the elliptic measure for $\vmo_{\log}$ matrices.
 
		\begin{proof}[Proof of \cref{sym + det 1 all cases}: $\vmo_{\log}$ case]\label{proof of square mean oscillation elliptic measure}
			Recall we can assume that $\Omega \subset B_{1/4} (0)$ is bounded Wiener regular by \cref{reductions all}, and we can replace finitely connected by simply connected by \cref{claim:candidate for each index}. In particular $\Omega \subset Q_{1/3} (0)$ (cube centered at $0$ with side length $2/3$).
			
			Let $\widetilde \varphi (\cdot) = \varphi(\cdot/100)$, with $\varphi$ as in \rf{definition of the support function}, and let $\phi_A \in W^{1,2}_{\loc}(\C)$ be the normalized solution of the Beltrami equation
			$$
			\bar\partial\phi_A = \widetilde\varphi\mu_A \partial\phi_A \text{ a.e.\ in }\C, \text{ where }\mu_A \text{ as in }\rf{beltrami coefficients for sym and det 1 matrices},
			$$
			given by the measurable Riemann mapping theorem, see \cref{riemann mapping thm}. Let $\widehat A$ be the uniformly elliptic, symmetric and with determinant $1$ matrix such that $\widetilde\varphi \mu_A = \mu_{\widehat A}$, see \cref{rem:interpolation to preserve determinant}. Therefore, the Beltrami equation above reads as
                $$
			\bar\partial\phi_A = \mu_{\widehat A} \partial\phi_A \text{ a.e.\ in }\C, \text{ with }\mu_{\widehat A} = \mu_A \text{ in } B_{50} (0).
			$$
            In particular $u\coloneqq \real \phi_A$ and $v\coloneqq \imag \phi_A$ satisfy $\divv A \nabla u = 0$ and $\divv A \nabla v = 0$ in $B_{50} (0)$, see \cref{all relations beltrami and pde}. Now, since $A\in \vmo_{\log}$ we have that $\nabla u, \nabla v \in \bmo$, see \cref{regularity for log mean oscillation matrices}.
			
			Now we want to see the same property for $\phi_A^{-1}$. By \cref{all relations beltrami and pde}, the functions $R\coloneqq  \real (\phi_A^{-1})$ and $I\coloneqq \imag (\phi_A^{-1})$ satisfy $\divv B\nabla R=0$ and $\divv B^* \nabla I=0$ in $\phi_A(B_{50}(0))$, where $B$ and $B^*$ are the uniformly elliptic matrices in \rf{matrix of inverse function beltrami}. Since $A\in \vmo_{\log}$, by \cref{properties of p-mean oscillation} we have that $\mu_A \in \vmo_{\log}$ and hence $\mu_{\widehat A}\in \vmo_{\log,c}$. By the invariance of $\vmo_{\log,c}$ under quasiconformal mappings in \cref{vmo log invariant under qc}, we have that $-\mu_{\widehat A}\circ \phi_A^{-1}\in \vmo_{\log,c}$, the coefficient of the conjugated Beltrami equation in \cref{thm qc pde cond3} of \cref{all relations beltrami and pde}. Therefore, by \cref{properties of p-mean oscillation} the uniformly elliptic matrices $\widehat B$ and ${\widehat B}^*$ in \rf{matrix of inverse function beltrami} (in terms of $\widehat A$) belong to $\vmo_{\log}$. By \cref{regularity for log mean oscillation matrices} and since $\widehat B = B$ and ${\widehat B}^* = B^*$ in $\phi_A(B_{50}(0))$, we conclude $\nabla R, \nabla I\in \bmo$.
			
			Since both $\nabla \phi_A, \nabla \phi_A^{-1} \in \bmo$, by the $\log$-Lip regularity in \cref{log lip when grad in bmo} and the distortion under these type of functions in \cref{distortion under log lip functions}, continuing as in the end of the proof of \cref{main thm general domains} on \cpageref{proof of main thm general domains}, using now \cref{lemma:relation elliptic measure for sym and det 1}, we have that there is a set $F\subset \partial \Omega$ satisfying $\omega_{\Omega, A} (F)=1$ and with $\sigma$-finite $\HH^{t/\log \frac{1}{t}}$ measure; and if the domain is simply connected, by Makarov's \cref{makarov thm} there exists $C_M>0$ such that $\omega_{\Omega, A} \ll \HH^{(\varphi_{1,C_M}) \circ (t\log \frac{1}{t})}$.
		\end{proof}
	
	
	\vv
	
	\vv

	\bibliographystyle{alpha}
	\bibliography{references-phd.bib} 

\end{document}